\documentclass[12pt]{article}
\usepackage{authblk}
\usepackage[T1]{fontenc}
\usepackage{inputenc}
\usepackage{babel}
\usepackage{array}
\usepackage{float}
\usepackage{textcomp}
\usepackage{multirow}
\usepackage{enumitem}
\usepackage{graphicx}
\usepackage{natbib}
\usepackage{hyperref}
\usepackage{graphicx} 
\usepackage{amsmath}
\usepackage{amssymb}
\usepackage{setspace}
\usepackage[linesnumbered,ruled,vlined]{algorithm2e}
\usepackage[most]{tcolorbox}
\usepackage{geometry}
\usepackage{soul}
\usepackage{empheq}
\usepackage{moresize}
\usepackage[left]{lineno}
\allowdisplaybreaks

\DeclareMathOperator*{\argmin}{arg\,min}
\usepackage{amsthm}
\usepackage{authblk}
\usepackage{tabularray}
\usepackage{longtable}
\usepackage{varwidth}
\usepackage{pdflscape}
\usepackage{caption}
\usepackage{subfig}
\usepackage{longtable}
\usepackage{array}
\usepackage{booktabs}
\usepackage{ragged2e}
\usepackage{caption}
\usepackage{longtable}
\usepackage{array}
\usepackage{booktabs}
\usepackage{ragged2e}

\hypersetup{
	colorlinks=true,
	linkcolor=blue,
	filecolor=blue,      
	urlcolor=cyan,
	citecolor=blue,
	pdftitle={Overleaf Example},
	pdfpagemode=FullScreen,
}
\newtheoremstyle{exampstyle1}
{8pt} 
{8pt} 
{} 
{} 
{\bfseries} 
{.} 
{3pt} 
{} 
\newtheorem{theorem}{Theorem}[section]
\newtheorem{lemma}[theorem]{Lemma}
\newtheorem{corollary}[theorem]{Corollary}
\theoremstyle{exampstyle1}\newtheorem{assumption}{Assumption}
\newtheorem{definition}{Definition}[subsection]
\newtheorem{Remark}{Remark}
\title{Estimation of multiple precision matrices under shared support with heterogeneous edge strengths}
\vspace{1cm}
\author{Sayan Ranjan Bhowal\thanks{Corresponding author: Theoretical Statistics and Mathematics Unit, Indian Statistical Institute, Kolkata} \;  Debashis Paul\thanks{Applied Statistics Division, Indian Statistical Institute, Kolkata} \; Gopal K Basak\thanks{Theoretical Statistics and Mathematics Unit, Indian Statistical Institute, Kolkata} \; Samarjit Das\thanks{Economic-Research Unit, Indian Statistical Institute, Kolkata}}
\date{}

\begin{document}
	\maketitle
	\begin{abstract}
		Estimating multiple precision matrices in high-dimension presents significant challenges, particularly when distinct datasets share a common conditional dependency structure but exhibit population-specific interaction strengths. We address this problem by introducing the Multiplicative Graphical Lasso (Mglasso), a method for jointly estimating precision matrices across multiple Gaussian graphical models under a shared sparsity constraint. Each precision matrix is decomposed as a Schur-Hadamard product of a shared structural matrix $\boldsymbol{\Theta}$, which encodes the common conditional independence graph, and a population-specific matrix $\boldsymbol{\Gamma}_{l}$, which captures variation in edge strengths across populations. We optimize a penalized log-likelihood that utilizes an $\ell_1$-penalty to enforce common sparsity and a Frobenius norm penalty to regulate population-specific variations. The optimization is efficiently performed using the Alternating Direction Method of Multipliers (ADMM) algorithm integrated with gradient descent. Theoretically, we establish the local strict convexity of the objective function and provide rigorous high-dimensional consistency guarantees, including supremum norm error bounds and exact support recovery under sub-Gaussian tail conditions. Extensive simulations show superior model selection consistency at smaller sample sizes compared to the benchmark Group Graphical Lasso (GGL). Finally, the method's practical utility is further validated through real-world applications.
	\end{abstract}
	\vspace{1em}
	\noindent\textbf{Keywords:} Precision matrix estimation; Sparsity; Graphical lasso; ADMM; Penalized likelihood; Sub-Gaussian distributions.
	\newpage
	\tableofcontents
	\section{Introduction}\label{sec-intro}
	
	Estimating precision matrices in a high-dimensional regime, where the number of variables $p$ is comparable to the sample size $n$, presents significant challenges. The sample covariance matrix $\hat{\boldsymbol{\Sigma}}=\frac{1}{n}\boldsymbol{X}^{T}\boldsymbol{X}$ becomes unreliable in this situation: while consistent for fixed $p$, it performs poorly as $p$ increases with $n$. In particular, when $p/n\rightarrow c>0$, its spectral quantities fail to consistently estimate their population counterparts (see e.g., \cite{johnstone2001distribution},\cite{johnstone2009sparseprincipalcomponentsanalysis}). This limitation is critical, particularly in a Gaussian graphical model, where $\boldsymbol{\Omega}=\boldsymbol{\Sigma}^{-1}$ encodes conditional dependencies: entries $\Omega_{ij}=0$ indicate conditional independence between variables $x_i$ and $x_j$.
	
	\indent To address this, regularized estimators of $\boldsymbol{\Sigma}$ or $\boldsymbol{\Omega}$ have been developed. The graphical lasso (\cite{friedman2008sparse}) estimates $\boldsymbol{\Omega}$ by maximizing a penalized Gaussian log-likelihood with an $l_1$-penalty that promotes sparsity, under the constraint $\boldsymbol{\Omega}\succ 0$. To attenuate the bias problem inherent in the $l_1$-penalty, \cite{10.1214/08-AOAS215} introduced nonconcave penalties, such as SCAD and the adaptive LASSO, recasting the estimation as a sequence of weighted $l_1$ penalized problems via local linear approximation. On the theoretical side, \cite{10.1214/08-EJS176} established consistency rates in Frobenius and operator norms that explicitly depend on the sparsity level of the true inverse covariance. Expanding on the graphical lasso approach, \cite{10.1214/11-EJS631} reframed the estimator as minimizing an $\ell_1$-penalized log-determinant Bregman divergence. The authors derived sharper high-dimensional convergence rates and exact support recovery beyond Gaussian assumptions. Building on this corpus, \cite{10.1214/15-EJS1031} proposed a principled debiasing scheme derived via inversion of the KKT conditions, thereby enabling valid inference for low-dimensional functionals of the precision matrix even when $p$ is large relative to $n$. Further extensions include rank-based estimation for nonparanormal models (\cite{xue2012regularized}) and constrained likelihood estimation to directed acyclic graphs (\cite{yuan2019constrained}).
	
	\indent Interest in jointly estimating precision matrices across multiple Gaussian graphical models has grown considerably, with methods exploiting shared structure to improve estimation when networks exhibit common sparsity patterns. Early work by \cite{guo2011joint} decomposed each precision entry as $\Omega_{l,ij}=\theta_{ij}\gamma_{l,ij}$, penalizing shared and individual components separately. \cite{danaher2014joint} introduced the Fused Graphical Lasso (FGL) and the Group Graphical Lasso (GGL), penalizing discrepancies between group estimates and enforcing consistent sparsity, respectively. \cite{lee2015joint} decomposed each precision matrix as the sum of a common and a group-specific varying component. The estimation procedure is motivated by CLIME estimation (\cite{cai2011constrained}). \cite{ma2016joint} took a neighborhood-selection approach (\cite{10.1214/009053606000000281}) that incorporates structural information by first estimating each model's edge set, and then maximizing the group-specific likelihood constrained to have zeros outside it. \cite{shan2020joint} proposed a two-level graphical lasso that clusters variables into subgroups and models conditional dependencies as the product of a higher-level subgroup factor and a lower-level individual factor. Numerous other joint estimation strategies have been proposed since. For a comprehensive summary of methods, see the survey paper by \cite{tsai2022joint}.

	\indent Although several methods have been developed, none of them truly bring out the common sparsity pattern of the graphs with the leverage of differences in the common edges. \cite{guo2011joint} has the penalty term of $\sum_{l=1}^{L}\sum_{i\neq j}|\gamma_{ij}^{l}|$, which is too strict and underestimates the individual components. The closest among the regularization-based approaches that can be used for estimating precision matrices under a common sparsity constraint is the GGL. However, there may be situations in which the networks contain unique edges. In this paper, we propose an estimation procedure called Multiplicative Graphical Lasso (Mglasso), which strictly works with the common edges among the networks and estimates the precision matrices. By first identifying shared structural patterns and then modeling individual variations only where that shared structure exists, this approach reduces parametric complexity. Simulation studies further demonstrate that it outperforms GGL, the benchmark competitor, in both model selection consistency and estimation precision.
	The remainder of this paper is structured as follows. Sections \ref{Section 2.1} and \ref{Section 2.2} introduce the Mglasso optimization problem and examine its bi-convexity property. Section \ref{Section 2.3} outlines the proposed methodology and the corresponding algorithm for solving the problem. Section \ref{Section 3} establishes the theoretical properties of the estimator. Section \ref{Section 4} demonstrates the performance of the proposed approach through experiments on synthetic and real-world datasets. A simulation example with heterogeneous sparsity structure has been presented in Section \ref{Section 4.2}. Detailed algorithm steps and outline, proofs of the local strict convexity and the main result along with all relevant lemmas, and supplementary application details are given in the Supplementary materials.
	
	\section{Multiplicative Graphical Lasso (Mglasso)}\label{Section 2}
	\subsection{Problem setup}\label{Section 2.1}
	Let $L$ denote the number of populations. The observations from population $l$, $\boldsymbol{x}_{li}$, are independently random samples from a $p$-dimensional Gaussian distribution with mean $\boldsymbol{0}$ and covariance matrix $\boldsymbol{\Sigma}_{l}$, $i=1,2,\ldots,n_l$, where $n_l$ is the number of samples from population $l$. We jointly estimate the precision matrices $\boldsymbol{\Omega}_l=\boldsymbol{\Sigma}_{l}^{-1}$ for all $L$ populations from their respective $n_l$ observations, assuming that the precision matrices share a common sparsity structure. We represent $\boldsymbol{\Omega}_{l}$ as the  Schur-Hadamard product of two matrices, i.e., $\boldsymbol{\Omega}_{l}=\boldsymbol{\Theta}\odot \boldsymbol{\Gamma}_{l}\ ,l=1,2,\ldots,L$. The proposed multiplicative decomposition is motivated by applications where multiple populations share a common sparsity pattern while exhibiting population-specific interaction strengths. Examples include brain connectivity networks across disease groups, gene regulatory networks across tissues, and financial dependency networks across market regimes. In such settings, the existence of an interaction between two variables is often governed by a common underlying mechanism, whereas the magnitude of interaction varies across populations. The matrix $\boldsymbol{\Theta}$ therefore captures the shared structural support, while $\boldsymbol{\Gamma}_{l}$ characterizes population-specific modulation of edge strengths. For identifiability of the parameters, we need to impose constraints on the matrices. In addition to assuming symmetry of $\boldsymbol{\Theta}$ and $\boldsymbol{\Gamma}_{l}$ for all $l=1,2,\ldots,L$, we need the following assumptions:
	\begin{assumption}\label{Assumption 1}
		If $\Theta_{ij}=0$ for some $(i,j)\in V\times V$, where $V$ is the set of vertices (coordinates), then $\Gamma_{l,ij}=0$ for all $l=1,2,\ldots,L$.
	\end{assumption}
	\begin{assumption}\label{Assumption 2}
		If $\Theta_{ij}\neq 0$ for some $(i,j)\in V\times V$, then $\overline{\Gamma}_{ij}=1$.
	\end{assumption}
	For some $(i,j) \in V \times V$, if $\Theta_{ij}=0$, Assumption \ref{Assumption 1} resolves the identifiability issue of $\Gamma_{l,ij}$ at those pairs. Additionally, note that $\boldsymbol{\Omega}_{l}=\boldsymbol{\Theta}\odot \boldsymbol{\Gamma}_{l}=a\boldsymbol{\Theta}\odot \frac{1}{a}\boldsymbol{\Gamma}_{l}$, for some scalar $a\neq 0$. Under Assumption \ref{Assumption 2}, the identifiability issue for the entire matrix $\boldsymbol{\Gamma}_l$ matrix is solved. As a result, we have $\overline{\boldsymbol{\Omega}}=\boldsymbol{\Theta}$. Also, since $\overline{\Gamma}_{ij}=\mathbb{I}(|\Theta_{ij}|>0)$, $\overline{\boldsymbol{\Gamma}}$ is a function of $\boldsymbol{\Theta}$. Under these assumptions, we solve the penalized log-likelihood optimization problem
	\begin{equation}\label{(1)}
		(\hat{\boldsymbol{\Omega_l}})_{\lambda,\mu}\in \argmin_{\substack{\boldsymbol{\Theta,\Gamma_l}\\ l=1,2,\ldots,L}}\sum_{l=1}^{L}-\log\ \det(\boldsymbol{\Theta}\odot \boldsymbol{\Gamma}_l)+tr(\boldsymbol{S}_l(\boldsymbol{\Theta}\odot \boldsymbol{\Gamma}_l))+\lambda l_1(\boldsymbol{\Theta})+\mu \sum_{l=1}^{L}||\boldsymbol{\Gamma}_l-\boldsymbol{\overline{\Gamma}(\Theta)}||_{F}^2,
	\end{equation} where $||.||_{F}^2$ is the square of the Frobenius norm of a matrix, subject to the constraint $\frac{1}{L}\sum_{l=1}^{L}\boldsymbol{\Gamma_l}=\boldsymbol{\overline{\Gamma}(\Theta)}$. The first penalty term $l_1(\boldsymbol{\Theta})$ controls sparsity across all the models, while the second penalty term $||\boldsymbol{\Gamma}_l-\boldsymbol{\overline{\Gamma}(\Theta)}||_{F}^2$ controls variation in the population-specific components of the precision matrices. Moreover, $\boldsymbol{\Gamma}_l$ depends on $\boldsymbol{\Theta}$, since $\Theta_{ij}=0$ implies $\Gamma_{l,ij}=0$. But,  since $\Theta_{ij}=0$ also implies $\overline{\Gamma}_{ij}=0$, the Frobenius norm penalty will force the corresponding terms of $\Gamma_{l,ij}$ to be negligible after optimization. Hence, we can ignore the dependence of $\boldsymbol{\Gamma_l}$ on $\boldsymbol{\Theta}$ and just consider $\boldsymbol{\overline{\Gamma}}$ to be a function of $\boldsymbol{\Theta}$.  For the optimization, we need to look at the convexity of the negative log-determinant term of the likelihood with respect to the Schur-Hadamard product structure, and if the log-determinant term is biconvex in its arguements, then alternating minimization techniques can be employed.
	
	\subsection{Bi-convexity of log-determinant}\label{Section 2.2}
	Let $\boldsymbol{x}$ be a $p$-dimensional Gaussian random vector with mean $\boldsymbol{0}$ and covariance matrix $\boldsymbol{\Sigma}$. Define $g(\boldsymbol{\Sigma}^{-1})=g(\boldsymbol{\Theta},\boldsymbol{\Gamma})=-\log\ \det(\boldsymbol{\Sigma}^{-1})=-\log\ \det(\boldsymbol{\Theta}\odot\boldsymbol{\Gamma})$, and let $g_1(\boldsymbol{\Gamma})=g(\boldsymbol{\Theta},\boldsymbol{\Gamma})$, with $\boldsymbol{\Theta}$ fixed.
	\begin{lemma}
		A function is convex if and only if its restriction to every line is convex. (\cite{citeulike:163662}, p. 67)
	\end{lemma}
	Let $g_2(t)=-\log\ \det(\boldsymbol{\Theta}\odot(\boldsymbol{\Gamma}+t\boldsymbol{V}))$, where $\boldsymbol{\Gamma}+t\boldsymbol{V}$ is positive definite. Then,
	\begin{eqnarray*}
		g_2(t)&=&-\log\ \det(\boldsymbol{\Theta}\odot(\boldsymbol{\Gamma}+t\boldsymbol{V}))\\
		&=& -\log\ \det((\boldsymbol{\Theta}\odot\boldsymbol{\Gamma})+t(\boldsymbol{\Theta}\odot\boldsymbol{V}))\\
		&=& -\log\ \det(\boldsymbol{\Sigma^{-1}}+t\boldsymbol{\Delta}),\text{ where }\boldsymbol{\Delta}=\boldsymbol{\Theta}\odot\boldsymbol{V}
	\end{eqnarray*}
	Since $\boldsymbol{\Sigma}^{-1}$ is a positive definite matrix, we have $\boldsymbol{\Sigma}^{-1}=\boldsymbol{\Sigma}^{-1/2}\boldsymbol{\Sigma}^{-1/2}$. Then,
	\begin{eqnarray*}
		g_2(t)&=&-\log\ \det(\boldsymbol{\Sigma^{-1/2}}(I_p+t\boldsymbol{\Sigma^{1/2}\Delta\Sigma^{1/2}})\boldsymbol{\Sigma^{-1/2}})\\
		&=& -\log\ \det (\boldsymbol{\Sigma^{-1}})-\sum_{i=1}^{p}\log(1+t\lambda_i),
	\end{eqnarray*}
	where $\lambda_i$'s are the eigenvalues of $\boldsymbol{\Sigma^{1/2}\Delta\Sigma^{1/2}}$. Thus, $g_{2}^{''}(t)=\sum_{i=1}^{p}\frac{\lambda_{i}^{2}}{(1+t\lambda_i)^2}\geq 0$. Hence, $g_1(\boldsymbol{\Gamma})$ is a convex function of $\boldsymbol{\Gamma}$. Similarly, we can also prove that $g_1(\boldsymbol{\Theta})=g(\boldsymbol{\Theta},\boldsymbol{\Gamma})$, with $\boldsymbol{\Gamma}$ fixed, is also a convex function of $\boldsymbol{\Theta}$. Hence, the negative log-determinant function is biconvex when the precision matrix is assumed to be the Schur-Hadamard product of two matrices $\boldsymbol{\Theta}$ and $\boldsymbol{\Gamma}$.
	\subsection{Optimization via ADMM with Gradient Descent}\label{Section 2.3}
	We solve the optimization problem in (\ref{(1)}) by alternatively fixing $\boldsymbol{\Gamma}_{l}$ and $\boldsymbol{\Theta}$. By denoting the objective function from (\ref{(1)}) as $\mathcal{L}_{\lambda,\mu}(\boldsymbol{\Theta,\Gamma}_{l})$, we can state the problem as
	\begin{equation}\label{(2)}
		\underset{\boldsymbol{\Theta,\Gamma}_l,l=1,2,\ldots,L}{min}\mathcal{L}_{\lambda,\mu}(\boldsymbol{\Theta,\Gamma}_{l}),
	\end{equation}
	subject to the constraint $\frac{1}{L}\sum_{l=1}^{L}\boldsymbol{\Gamma}_{l}=\boldsymbol{\overline{\Gamma}(\Theta)}$. We solve this problem using the Alternating Direction Method of Multipliers (ADMM) algorithm (\cite{boyd2011distributed}). The augmented Lagrangian form of the optimization problem is given by
	\begin{eqnarray}
		\notag
		& &\underset{\boldsymbol{\Theta},\boldsymbol{\mathcal{G}},\boldsymbol{R}}{min}-\sum_{l=1}^{L}\log\:\det(\boldsymbol{\Theta}\odot\boldsymbol{\Gamma}_{l})+\sum_{l=1}^{L}tr(\boldsymbol{S}_{l}(\boldsymbol{\Theta}\odot\boldsymbol{\Gamma}_{l}))+\lambda l_{1}(\boldsymbol{R})+tr(\boldsymbol{B}(\boldsymbol{\Theta}-\boldsymbol{R}))\\
		& &+\frac{1}{2\delta}||\boldsymbol{\Theta}-\boldsymbol{R}||_{F}^{2}+\mu\sum_{l=1}^{L}||\boldsymbol{\Gamma}_{l}-\bar{\boldsymbol{\Gamma}}(\boldsymbol{\Theta})||_{F}^{2}=\underset{\boldsymbol{\Theta},\boldsymbol{\mathcal{G}},\boldsymbol{R}}{min}\mathcal{L}_{\lambda,\mu,\delta}(\boldsymbol{\Theta},\boldsymbol{\mathcal{G}},\boldsymbol{R},\boldsymbol{B}),
	\end{eqnarray}
	subject to the constraint $\frac{1}{L}\sum_{l=1}^{L}\boldsymbol{\Gamma}_l=\overline{\boldsymbol{\Gamma}}(\boldsymbol{\Theta})$, where $\mathcal{G}=\{\boldsymbol{\Gamma}_{l},l=1,2,\ldots,L\}$. Here $\lambda$ and $\mu$ are the penalty parameters. The ADMM updates at the $k^{th}$ iteration are 
	\begin{eqnarray}\label{(4)}
		\boldsymbol{\Theta}^k &=& \argmin_{\boldsymbol{\Theta}}\mathcal{L}_{\lambda,\mu,\delta}(\boldsymbol{\Theta},\mathcal{G}^{k-1},\boldsymbol{R}^{k-1},\boldsymbol{B}^{k-1})\\
		\label{(5)}
		\boldsymbol{\mathcal{G}}^k &=& \argmin_{\boldsymbol{\mathcal{G}}}\mathcal{L}_{\lambda,\mu,\delta}(\boldsymbol{\Theta}^k,\mathcal{G},\boldsymbol{R}^{k-1},\boldsymbol{B}^{k-1})\\
		\label{(6)}
		\boldsymbol{R}^k &=& \argmin_{\boldsymbol{\boldsymbol{R}}}\mathcal{L}_{\lambda,\mu,\delta}(\boldsymbol{\Theta}^k,\mathcal{G}^k,\boldsymbol{R},\boldsymbol{B}^{k-1})\\
		\label{(7)}
		\boldsymbol{B}^k &=& \boldsymbol{B}^{k-1}+\frac{1}{\delta}(\boldsymbol{\Theta}^k-\boldsymbol{R}^k)
	\end{eqnarray}
	Since there are no closed-form expressions for the inverse of the Schur-Hadamard product of two matrices, we cannot solve the optimization problems (\ref{(4)}) and (\ref{(5)}) explicitly. So, we use the gradient descent method to solve them within each ADMM step. The detailed steps for the solution are given in the Appendix. We call the resulting algorithm the Multiplicative Graphical Lasso (Mglasso). The entire algorithm is summarised in the supplementary materials.
	\begin{Remark}
		(Tuning-parameter selection) The Extended Bayesian Information Criterion (EBIC) is a measure used for model and tuning-parameter selection. For $L$ populations, the EBIC takes the form
		\begin{eqnarray}\label{(16)}
			\notag
			\text{EBIC}_{\gamma}(\lambda)&=&-\sum_{l=1}^{L}\log\ \det(\boldsymbol{\Theta}\odot \boldsymbol{\Gamma_l})+\sum_{l=1}^{L}tr(\boldsymbol{S_l}(\boldsymbol{\Theta}\odot \boldsymbol{\Gamma_l}))+(L+1)l_0(\boldsymbol{\Theta})\log(N)\\
			& &+4\gamma(L+1)l_0(\boldsymbol{\Theta})\log(p),
		\end{eqnarray}
		$N=\sum_{l=1}^{L}n_{l}$, $p$ denotes the dimension, and $\gamma\in [0,1]$ is a parameter that penalizes complex models. If $\gamma=0$, we have the classical BIC. However, if $\gamma$ is positive, the criterion imposes a stronger penalty on non-sparse models; therefore, selection using EBIC will be more inclined to select sparse models when $\gamma>0$. For a detailed discussion on EBIC, we refer to \cite{chen2008extended} and \cite{foygel2010extended}. We select the optimal $\lambda$ and $\mu$ by EBIC minimization. We evaluate EBIC over a grid of $(\lambda,\mu)$ values and select $(\hat{\lambda}_{opt},\hat{\mu}_{opt})$ as the pair $(\lambda,\mu)$, for which the EBIC in (\ref{(16)}) is minimized.
	\end{Remark}
	\begin{Remark}
		\label{Remark 3}
		(Initial estimates) For the ADMM loop to start, we need an initial estimate of $\boldsymbol{\Theta},\boldsymbol{\Gamma_l},l=1,2,\ldots, L,\boldsymbol{R},\boldsymbol{B}$. Due to Assumptions \ref{Assumption 1} and \ref{Assumption 2}, we have the property that $\overline{\boldsymbol{\Omega}}=\boldsymbol{\Theta}$. We start with the graphical lasso estimates for the individual populations as the initial estimates for $\boldsymbol{\Omega}_{l}$, i.e., $\hat{\boldsymbol{\Omega}}_{l0}=\boldsymbol{\Omega}_{l}^{g}$. Then we have the initial estimates of $\boldsymbol{\Theta}_{0}=\frac{1}{L}\sum_{l=1}^{L}\boldsymbol{\Omega}_{l}^{g}$, and $\boldsymbol{R}_0=\boldsymbol{\Theta}_{0}$. For the initial estimates of $\boldsymbol{\Gamma}_{l}$, we have
		\begin{equation*}
			\Gamma_{l0,ij}=\begin{cases}
				\frac{\hat{\Omega}_{l0,ij}}{\Theta_{0,ij}},& \text{if }|\Theta_{0,ij}|>0\\
				0,& \text{if }\Theta_{0,ij}=0
			\end{cases}.
		\end{equation*}
		One limitation of this method is that if the initial estimate $\boldsymbol{\Theta}$ is sparser than the actual network, then this method may fail, since the updates are done coordinate-wise in each step, and if a coordinate value becomes zero at a particular iteration, it will change only if the corresponding coordinate of the Lagrangian multiplier $\boldsymbol{B}$ at initialization is non-zero. 
	\end{Remark}
	\section{Asymptotics}\label{Section 3}
	This section establishes the large-sample properties of the estimators derived from the optimization problem (\ref{(2)}) in a setting where both the dimension $p$ and the sample sizes $n_l,l=1,2,\ldots,L$ approach infinity. To facilitate the theoretical analysis, we first introduce the necessary notation. 
	\subsection{Notations}\label{Section 3.1}
	Denote $\boldsymbol{A}=(A_{ij})_{i,j=1}^{p}$ as a matrix, where $A_{ij}$ denotes the $(i,j)^{th}$ entry of $\boldsymbol{A}$. We denote $\boldsymbol{A}\otimes \boldsymbol{B}$ as the Kronecker product of two matrices $\boldsymbol{A}$ and $\boldsymbol{B}$. We denote by $\boldsymbol{A}^{+}$ the matrix $diag(\boldsymbol{A})$ and define $\boldsymbol{A}^{-}=\boldsymbol{A}-\boldsymbol{A}^{+}$ which is the matrix containing only the off-diagonal elements of $\boldsymbol{A}$. For any vector $\boldsymbol{x}\in\mathbb{R}^{p}$, we denote the $d$-norm of $\boldsymbol{x}$ by $||\boldsymbol{x}||_{d}$, $d\in(0,\infty]$. We use $||\boldsymbol{A}||_{F}^{2}=\sum_{i,j}A_{ij}^{2}$, $||\boldsymbol{A}||_{\infty}=\underset{i,j}{max}|A_{ij}|$, $|||\boldsymbol{A}|||_{\infty}=\underset{i}{max}\sum_{j}|A_{ij}|$, and $|||\boldsymbol{A}|||_{1}=|||\boldsymbol{A}^{T}|||_{\infty}$ as the squared Frobenius norm, supremum norm, $l_{\infty}$-operator norm, and $l_{1}$-operator norm of a matrix $\boldsymbol{A}$. We also denote $||\boldsymbol{A}||_{1}=\sum_{i,j}|A_{ij}|$ as the $l_1$-norm of $\boldsymbol{A}$. For sequences, we write $f_{n}=O(g_{n})$, if $f_{n}\leq w_1g_{n}$, for some $w_1<\infty$, and $f_{n}=\phi(g_{n})$, if $f_{n}>w_2g_{n}$, for some $w_2>0$. Then $f_{n}\asymp g_{n}$ if $f_{n}=O(g_{n})$, and $f_{n}=\phi(g_{n})$. We denote the set $\mathcal{S}_{++}^{p}=\{\boldsymbol{\Omega}:\boldsymbol{\Omega}=\boldsymbol{\Omega}^{T},\boldsymbol{\Omega}\succ0\}$ as the set of all symmetric positive definite matrices.
	\subsection{Tail conditions}\label{Section 3.2}
	We need bounds on the difference between the sample covariance matrices $\boldsymbol{S_l}$ and the true covariance matrices $\boldsymbol{\Sigma}_{l}^{0},l=1,2,\ldots,L$, for obtaining consistency results for our estimators. We refer to some definitions and lemmas from \cite{bhowal2025statisticalinferenceusingdebiased} for the theoretical analysis. However, for completeness, we restate the required tail conditions.
	\begin{definition}
		\label{Definition 1}
		Suppose we have random vectors $\boldsymbol{x}_{l}$ from population $l,l=1,2,\ldots,L$. Define $V$ as the set of coordinates of the random vectors. Then $\boldsymbol{x}_{l},l=1,2,\ldots,L$ satisfies the tail conditions $\mathcal{T}(f,v_{l0})$, if we have a $v_{l0}\in (0,\infty],l=1,2,\ldots,L$, and a function $f:\mathbb{N}\times(0,\infty)\rightarrow(0,\infty)$, such that $(i,j)\in V\times V$,
		\begin{equation}
			\label{(17)}
			\mathbb{P}(|S_{l,ij}-\Sigma_{l,ij}^{0}|>\delta)\leq \frac{1}{f(n_l,\delta)},
		\end{equation}
		for all $\delta\in \bigg(0,\underset{l}{min}\ \frac{1}{v_{l0}}\bigg],l=1,2,\ldots,L$.
	\end{definition}
	In the case of a single population, Definition \ref{Definition 1} reduces to the condition that $\boldsymbol{x}$ satisfies $\mathcal{T}(f,v_{0})$ if there exists a constant $v_{0}\in (0,\infty]$ and a mapping $f:\mathbb{N}\times(0,\infty)\rightarrow(0,\infty)$ such that for every variable pair $(i,j)\in V\times V$, the tail probability is bounded by:
	$$
	\mathbb{P}(|S_{ij}-\Sigma_{ij}^{0}|>\delta)\leq \frac{1}{f(n,\delta)}
	$$
	for any deviation $\delta\in (0,\frac{1}{v_{0}}]$, where $n$ represents the sample size.
	
	We can deduce the relationship between $f$, $n$, and $\delta$. As the sample size $n$ grows, the probability of a large deviation (the tail probability) should naturally shrink, meaning the denominator $f(n,\delta)$ must increase. Likewise, a larger allowable deviation $\delta$ also results in a smaller tail probability, which again requires $f(n,\delta)$ to increase. 
	
	Based on this monotonic behavior of $f(n,\delta)$ in these models, we can formally define two corresponding inverse functions for a given threshold $r$:
	\begin{equation}
		n_{f}(\delta,r)=\underset{n}{\arg\max}\{f(n,\delta)\leq r\}
		\label{(18)}
	\end{equation}
	and
	\begin{equation}
		\delta_{f}(n,r)=\underset{\delta}{\arg\max}\{f(n,\delta)\leq r\}
		\label{(19)}
	\end{equation}
	
	A direct, intuitive consequence of defining these inverse functions is the following implication:
	\begin{equation}
		n>n_{f}(\delta,r)\text{ for some }\delta>0\implies\delta_{f}(n,r)\leq \delta
		\label{(20)}
	\end{equation}
	
	\subsubsection{Sub-Gaussian distributions}
	\begin{definition}
		\label{Definition 2}
		(Sub-Gaussian random variables) A real-valued random variable $x$ with $\mathbb{E}(x)=0$ is sub-Gaussian if there exists a constant $K_1>0$ such that $\mathbb{E}(e^{x^2/K_1^2})\leq 2$. In other words, the tails of the distribution of $x$ decay at least as fast as those of a centered Gaussian random variable with variance $K_1^2$.
	\end{definition}
	\begin{definition}
		\label{Definition 3}
		(Sub-Gaussian condition for multiple populations) Let $\boldsymbol{x}_{l}$ denote a $p-$variate random vector from population $l$, with mean $\boldsymbol{0}$ and covariance matrix $\boldsymbol{\Sigma}_{l}^{0},l=1,2,\ldots,L$. We say $\boldsymbol{x}_{l}$ satisfy the sub-Gaussian condition if each standardized covariate $x_{li}/\sqrt{\Sigma_{l,ii}^{0}},i=1,2,\ldots,p$ is a sub-Gaussian random variable with a common parameter $K_1>0$ uniformly across all populations $l=1,2,\ldots,L$.
	\end{definition}
	We state the following lemma regarding the tail conditions of the sub-Gaussian random variables for multiple populations.
	\begin{lemma}
		Consider $L$ independent populations of zero-mean random vectors $\{\boldsymbol{x}_{l}\}_{l=1}^{L}$, where the $l^{th}$ population has covariance matrix $\boldsymbol{\Sigma}_{l}^{0}$ and satisfies the sub-Gaussian condition for multiple populations (Definition \ref{Definition 3}) with a constant $K_1>0$. Let $n_l$ denote the sample size for population $l$, and let $\boldsymbol{S}_l$ be its sample covariance matrix. Then for every pair of indices $(i,j)$, the absolute difference of the coordinates $|S_{l,ij}-\Sigma_{l,ij}^{0}|$ admits the exponential tail bound
		\begin{equation}
			\mathbb{P}(|S_{l,ij}-\Sigma_{l,ij}^{0}|>\delta)\leq 4\ exp\bigg(-\frac{n_l \delta^2}{128(1+12K_{1}^{2})^{2}\ max(\Sigma_{l,ii}^{0})^2}\bigg),
			\label{(21)}
		\end{equation}
		for all $\delta \in (0,\underset{l_1}{min}\ 8(1+12K_{1}^{2})max\ (\Sigma_{l_{1},ii}^{0})]$.
		\label{Lemma 0.1}
	\end{lemma}
	Thus for single population, Lemma \ref{Lemma 0.1} implies that the absolute difference of the coordinates of the sample covariance matrix and the population covariance matrix, i.e., $|S_{ij}-\Sigma_{ij}^{0}|$ satisfies (\ref{(21)}) for all $\delta \in (0, 8(1+12K_{1}^{2})max\ (\Sigma_{ii}^{0})]$, with $n$ as the sample size. Then we have the inverse mappings from its definitions as
	\begin{equation}
		\delta_{f}(n,r)=8\sqrt{2}(1+12K_1^2)max\ (\Sigma_{ii}^{0})\sqrt{\frac{\log(4r)}{n}},
		\label{(22)}
	\end{equation}
	and,
	\begin{equation}
		n_{f}(\delta,r)=128(1+12K_{1}^{2})^{2}max\ (\Sigma_{ii}^{0})^{2}\ \frac{\log(4r)}{\delta^2}.
		\label{(23)}
	\end{equation}
	As a consequence of Lemma \ref{Lemma 0.1} for multiple populations, if $\delta_{l}=\delta_{f}(n_l,p^{\gamma})$ from (\ref{(22)}) with $\gamma>2$, and $n_l$ is such that $\delta_{l}\leq \underset{l_1}{min}\ 8(1+12K_{1}^{2})max\ (\Sigma_{l_{1},ii}^{0})$, for all $l=1,2,\ldots, L$, then as a consequence of Lemma \ref{Lemma 0.1}, we have 
	\begin{equation}
		\mathbb{P}(||\boldsymbol{S}_{l}-\boldsymbol{\Sigma}_{l}^{0}||_{\infty}>\delta)\leq \mathbb{P}(||\boldsymbol{S}_{l}-\boldsymbol{\Sigma}_{l}^{0}||_{\infty}>\delta_{l})\leq \frac{1}{p^{\gamma-2}},
		\label{(24)}
	\end{equation}
	where $\delta=\underset{l}{max}\ \delta_{l}$. Then,
	\begin{equation}
		\delta=\underset{l}{max}\ 8\sqrt{2}(1+12K_1^2)max\ (\Sigma_{l,ii}^{0})\sqrt{\frac{\log(4p^{\gamma})}{n_l}},
		\label{(25)}
	\end{equation}
	which implies $\delta\asymp \sqrt{\frac{\log(p)}{n}}$, where $n_l\asymp n,l=1,2,\ldots,L$ and $n=\underset{l}{min}\ n_l$.
	
	\subsection{Local strict convexity}\label{Section 3.3}
	Note that the optimization problem in (\ref{(2)}) can also be rewritten as
	\begin{equation}\label{(27)}
		\{\boldsymbol{\Omega}_l,l=1,2,\ldots,L\}=\argmin_{\boldsymbol{\Omega}=\{\boldsymbol{\Omega}_{l},l=1,2,\ldots,L\}} Q(\boldsymbol{\Omega}),
	\end{equation}
	where,
	\begin{equation}\label{(28)}
		Q(\boldsymbol{\Omega})= \sum_{l=1}^{L}-\log\ \det(\boldsymbol{\Omega}_{l})+tr(\boldsymbol{S}_l\boldsymbol{\Omega}_{l})+\lambda \bigg|\bigg|\bigg(\frac{1}{L}\sum_{l=1}^{L}\boldsymbol{\Omega}_{l}\bigg)\bigg|\bigg|_{1}+\mu \sum_{l=1}^{L}\sum_{(i,j)}\frac{(\Omega_{l,ij}-\overline{\Omega}_{ij})^2}{(\overline{\Omega}_{ij})^2}\mathbb{I}\bigg(\overline{\Omega}_{ij}\neq 0\bigg).
	\end{equation}
	\begin{Remark}
		It is important to observe that when $L=1$, the contribution of the $\mu$-dependent penalty term in (\ref{(2)}) and (\ref{(28)}) reduces to zero. This reduction aligns the optimization with the conventional graphical lasso, a framework whose theoretical guarantees have been extensively analyzed (e.g., \cite{10.1214/08-EJS176}, \cite{10.1214/11-EJS631}). However, because this penalty term remains active when $L>1$, the resulting optimization diverges significantly from the $L=1$ case. Our ensuing analysis is therefore dedicated to this multi-population setting, i.e., when $L>1$. 
	\end{Remark}
	We denote $E_{l}=\{(i,j):\boldsymbol{\Omega}^{0}_{l,ij}\neq0\}$ as the set of all coordinates with non-zero elements in $\boldsymbol{\Omega}^{0}_{l}$ and $e_{l}=|E_{l}|$ as the cardinality of $E_{l}$. Under the assumption that the sparsity pattern of the precision matrices $\boldsymbol{\Omega}_{l}^{0}$ remains constant across the models, we have $E_l=E$, and $e_l=e$, for all $l=1,2,\ldots, L$. Define $d_{l}=\underset{i=1,2,\ldots,p}{max}|\{j\in V:\Omega_{l,ij}^{0}\neq 0\}|$ as the maximum row cardinality of model $l$ defined as the maximum number of non-zero coordinates in any row of $\boldsymbol{\Omega}_{l}^{0}$. If we consider Gaussian graphical models with precision matrices $\boldsymbol{\Omega}_{l}^{0}$, then $d_{l}$ represents the highest number of edges any vertex shares with the others, including the self-loop of each vertex. Since the sets $E_l$ are the same across the models, we have $d_{l}=d$, for all $l=1,2,\ldots, L$. Define $\kappa_{\boldsymbol{\Sigma}_{l}^{0}}=|||\boldsymbol{\Sigma}_{l}^{0}|||_{\infty}$ for $l=1,2,\ldots,L$ which measures the size of the entries of the covariance matrix for each model $l$. The Hessian of the negative log-likelihood function is given by 
	$$
	\boldsymbol{\eta}(\boldsymbol{\Omega}_{l})=\boldsymbol{\Sigma}_{l}^{0}\otimes \boldsymbol{\Sigma}_{l}^{0}. 
	$$ The Hessian at the true parameter $\boldsymbol{\Omega}^{0}_{l}$ is denoted by $\boldsymbol{\eta}_{l}^{0}$. Let us define $\boldsymbol{\eta}^{0}_{lT_{l}T_{l}'}$ with rows and columns of $\boldsymbol{\eta}_{l}^{0}$ indexed by $T_{l}$ and $T_{l}'$, where $T_{l}$ and $T_{l}'$ are subsets of $V$. Define  $$\boldsymbol{\eta}^{0}_{lEE}=[\boldsymbol{\Sigma}_{l}^{0}\otimes \boldsymbol{\Sigma}_{l}^{0}]_{EE},$$ and $\kappa_{\boldsymbol{\eta}_{l}^{0}}=|||(\boldsymbol{\eta}^{0}_{lEE})^{-1}|||_{\infty}$. Any regulation over $\kappa_{\boldsymbol{\eta}_{l}^{0}}$ implies regulation of sparsity in $\boldsymbol{\Omega}_{l}^{0}$.
	
	We now present the assumptions under which the theoretical properties of our estimator will be studied.
	\begin{assumption}\label{Assumption 3}
		(Bounded eigenvalues) There exists a constant $L_1$
		such that 
		\begin{equation*}
			0<L_{1}<\Lambda_{min}(\boldsymbol{\Omega}_{l}^{0})\leq\Lambda_{max}(\boldsymbol{\Omega}_{l}^{0})<1/L_{1}<\infty,
		\end{equation*}
		where $\Lambda_{min}$ and $\Lambda_{max}$ denote the minimum and maximum eigenvalues, for all $l=1,2,\ldots,L$.
	\end{assumption}
	\begin{assumption}\label{Assumption 4}
		(Sample size ratio and dimension) Let $n=\underset{l}{min}\ n_l$, where $n_l$ is the number of samples drawn from population $l$. Then the sample sizes are all of the same order as $n$, i.e.,  $n_l\asymp n,l=1,2,\ldots,L$. Also, the dimension $p$ is comparable to $n$
	\end{assumption}
	\begin{assumption}\label{Assumption 5}
		(Irrepresentability condition) We have an $\alpha\in(0,1]$, such that
		\begin{equation}
			\label{(29)}
			\underset{e \in E^{c}}{max}||\boldsymbol{\eta}_{leE}^{0}(\boldsymbol{\eta}_{lEE}^{0})^{-1}||_{1}\leq 1-\alpha,
		\end{equation}
		for every $l=1,2,\ldots,L$.
	\end{assumption}
	
	Assumption \ref{Assumption 3} and \ref{Assumption 4} are quite mild and are frequently adopted in the literature. Assumption \ref{Assumption 5} is imposed to ensure that the influence of non-edge elements in $E^c$ remains limited relative to those belonging to the edge set $E$ across all population models simultaneously. Consider defining the zero-mean edge variables for model $l$ as
	$$
	y_{l,i_1i_2}=X_{li_1}X_{li_2}-\mathbb{E}(X_{li_1}X_{li_2}).
	$$
	The covariance structure of these variables is encoded in the entries of $\boldsymbol{\eta}_{l}^{0}$. Thus, if an edge-type variable $y_{l,i_1i_2}$ corresponds to a pair outside the edge set $E$, Assumption \ref{Assumption 5} guarantees that it is not strongly correlated with variables associated with actual edges in $E$. The parameter $\alpha$ quantifies this level of dependence, and larger values of $\alpha$ (closer to 1) indicate weaker correlations, approaching zero as $\alpha\rightarrow 1$.
	
	The algorithm begins by taking the initial estimators for each population from the graphical lasso solutions, as outlined in Remark (\ref{Remark 3}). Define $\hat{E}_{l}=\{(i,j):\Omega_{l,ij}^{g}\neq 0\}$, where $\Omega_{l,ij}^{g}$ is the $(i,j)^{th}$ entry of the graphical lasso estimate of the precision matrix corresponding to the population model $l=1,2,\ldots, L$. With this definition in place, we obtain the following lemma:
	\begin{lemma}
		\label{Lemma 0.1.1}
		Suppose $\boldsymbol{x}_{il}\in \mathbb{R}^p$ are independent and identically distributed samples with mean $\boldsymbol{0}$ and covariance matrix $\boldsymbol{\Sigma}_{l}^{0},l=1,2,\ldots,L$ satisfying Assumption \ref{Assumption 5} for some $\alpha\in (0,1]$ and the tail conditions (\ref{(17)}) with $\mathcal{T}(f,v_{l0})$ for all $l=1,2,\ldots,L$. If $\boldsymbol{\Omega}_{l}^{g}$ is the unique solution to the optimization problem
		\begin{equation}
			\label{(30)}
			\underset{\boldsymbol{\Omega}_{l}}{arg\ min}\ -\log\ \det(\boldsymbol{\Omega}_{l})+tr(S_l\Omega_l)+\lambda_{1}||\boldsymbol{\Omega}_{l}||_{1},
		\end{equation}
		for all $l=1,2,\ldots,L$, with $\lambda_{1}=\frac{8}{\alpha}\delta$, where $\delta=\underset{l}{max}\ \delta_{f}(n_l,p^{\gamma})$, for some $\gamma>2$, and if the sample size $n_l$ is bounded by 
		\begin{equation}
			\label{(31)}
			n_l>n_{f}\bigg(min\bigg\{\frac{1}{6d(1+(8/\alpha))\kappa_{\boldsymbol{\eta}_{l}^{0}}}\underset{l_1}{min}\bigg\{min\bigg\{\frac{1}{\kappa_{\boldsymbol{\Sigma}_{l_1}^{0}}},\frac{1}{\kappa_{\boldsymbol{\Sigma}_{l_1}^{0}}^{3}\kappa_{\boldsymbol{\eta}_{l_1}^{0}}}\bigg\}\bigg\},\underset{l}{min}\bigg\{\frac{1}{v_{l0}}\bigg\}\bigg\},p^{\gamma}\bigg),
		\end{equation}
		for all $l=1,2,\ldots,L$, then we have $\mathbb{P}(\cup_{l=1}^{L}\hat{E}_{l}\subseteq E)\geq 1-\frac{1}{p^{\gamma-2}}$.
	\end{lemma}
	\begin{proof}
		From Theorem 1 (\cite{10.1214/11-EJS631}), we have $\mathbb{P}(\hat{E_l}\subseteq E)\geq 1-\frac{1}{p^{\gamma-2}}$. Then, $\mathbb{P}(\cup_{l=1}^{L}\hat{E}_{l}\subseteq E)\geq \mathbb{P}(\hat{E_l}\subseteq E)\geq 1-\frac{1}{p^{\gamma-2}}\rightarrow 1$.
	\end{proof}
	
	Thus, the edge set $E$ can be identified with high probability before estimation. Conditional on the event that the edge set has been identified, we examine the difference of our estimator from the true precision matrices in the element-wise norm.
	
	Note that the penalty term involved with $\mu$ is non-convex. This raises a question about the convexity of the optimization function $Q(\boldsymbol{\Omega})$ in (\ref{(28)}) and whether the minimizer is unique. However, if we can prove that $Q(\boldsymbol{\Omega})$ is locally convex within an $\epsilon$-neighborhood of the true parameter under some assumptions, then we have a unique solution of (\ref{(27)}). The following theorem establishes that $Q(\boldsymbol{\Omega})$ from (\ref{(27)}) is a strictly convex function around a neighborhood of the true parameter.
	\begin{theorem}
		\label{Theorem 1}
		(Local strict convexity) Let $\{\boldsymbol{\Omega}_{l}^{0}\}$ be the true precision matrices and $\overline{\boldsymbol{\Omega}}^{0}$ be the average of the true precision matrices. Define the supnorm ball around the truth
		\begin{equation}
			\label{((32))}
			\mathcal{B}_{\infty}(\epsilon)=\{\{\boldsymbol{\Omega}_{l}\}:\underset{l}{max}\ ||\boldsymbol{\Omega}_{l}-\boldsymbol{\Omega}_{l}^{0}||_{\infty}\leq \epsilon\}.
		\end{equation}
		Suppose the true edge set $E$ has been identified consistently. Then, under Assumption \ref{Assumption 3}, there exists a constant $K(\boldsymbol{\Omega}^{0},L_1,\epsilon,d)$ such that whenever $0< \mu< K(\boldsymbol{\Omega}^{0},L_1,\epsilon,d)$, the function $Q(\boldsymbol{\Omega})$ is strictly convex in $\mathcal{B}_{\infty}(\epsilon)$, if $\underset{(i,j)\in E}{min}|\overline{\Omega}_{ij}^{0}|>\epsilon$, where,
		$$
		K(\boldsymbol{\Omega}^{0},L_1,\epsilon,d)=\frac{(\underset{(i,j)\in E}{min}|\overline{\Omega}_{ij}^{0}|-\epsilon)^{4}}{16\ \underset{(i,j)}{max}(\underset{l}{max}\ |\Omega_{l,ij}^{0}|+\epsilon)^{2}}\bigg(\frac{1}{L_1}+d\epsilon\bigg)^{-2}.
		$$
	\end{theorem}
	\begin{Remark}
		The neighborhood $\mathcal{B}_{\infty}(\epsilon)$ considered in Theorem \ref{Theorem 1} is defined in terms of the element-wise supnorm around the true precision matrices. Since each row of $\boldsymbol{\Omega}_{l}^{0}$ contains at most $d$ non-zero elements, the sum of absolute deviations in any row is bounded by $d\epsilon$. By the Ger\v{s}gorin circle theorem, this implies $\underset{i}{max}|\Lambda_{i}(\boldsymbol{\Omega}_{l}-\boldsymbol{\Omega}_{l}^{0})|\leq d\epsilon$, for any $\{\boldsymbol{\Omega}_{l}\}$ in $\mathcal{B}_{\infty}(\epsilon)$. In particular, choosing $\epsilon=O(1/d)$ ensures that the perturbation $d\epsilon$ remains bounded within the neighborhood. This choice is sufficient for establishing the consistency properties of the proposed estimator in the subsequent analysis. 
	\end{Remark}
	\subsection{Main result}\label{Section 3.4}
	We have proved that, conditioned on the event that $E$ has been identified consistently, $Q(\boldsymbol{\Omega})$ is a strictly convex function in $\mathcal{B}_{\infty}(\epsilon)$, and, hence, has a unique minimizer in it. We now need to show that if we restrict ourselves to  $\mathcal{B}_{\infty}(\epsilon)$ and obtain the estimator by solving (\ref{(27)}), we can derive an upper bound for the element-wise norm of the difference between the minimizer and the true precision matrices. Our estimator in $\mathcal{B}_{\infty}(\epsilon)$ is obtained after optimizing 
	\begin{equation}\label{(46)}
		\{\hat{\boldsymbol{\Omega}}_l,l=1,2,\ldots,L\}=\underset{\{\boldsymbol{\Omega}_{l},||\boldsymbol{\Omega}_{l}-\boldsymbol{\Omega}_{l}^{0}||_{\infty}\leq \epsilon,l=1,2,\ldots,L\}}{arg\ min} Q(\boldsymbol{\Omega}).
	\end{equation}
	Consider the following assumption under which the theoretical analysis is carried out.
	
	\begin{assumption}\label{Assumption 6}
		(Well-behaved precision matrices) The precision matrix $\boldsymbol{\Omega}_{l}^{0}$ for population $l$ satisfies the following restrictions:
		\begin{enumerate}
			\item \label{Assumption 6.1} 
			Define: 
			$$
			C_1= \underset{l_1}{min}\left\{ min\left\{\frac{1}{3\kappa_{\boldsymbol{\Sigma}_{l_1}^{0}}},\frac{1}{3\kappa^{3}_{\boldsymbol{\Sigma}_{l_1}^{0}}\kappa_{\boldsymbol{\eta}_{l_1}^{0}}}\right\}\right\}.
			$$ Then, for all $(i,j)\in E$, we have
			\begin{equation}
				\label{(47)}
				|\overline{\Omega}_{ij}^{0}|>\frac{C_1}{d}+\bigg(\frac{C_1}{d}\bigg)^{1/4},
			\end{equation}
			\item \label{Assumption 6.2} Also, for some positive bounded constant $c$,
			\begin{equation}
				\underset{l}{max}\ \underset{(i,j)\in E}{max}\frac{|\Omega_{l,ij}^{0}-\overline{\Omega}_{ij}^{0}|}{|\overline{\Omega}_{ij}^{0}|}\leq \frac{-1+\sqrt{1+2c\delta^{1/4}}}{2},
				\label{(48)}
			\end{equation}
			i.e., for small $\delta$,
			\begin{equation*}
				\underset{l}{max}\ \underset{(i,j)\in E}{max}\frac{|\Omega_{l,ij}^{0}-\overline{\Omega}_{ij}^{0}|}{|\overline{\Omega}_{ij}^{0}|}=O(\delta^{1/4}).
			\end{equation*}
		\end{enumerate}
	\end{assumption}
	The following theorem bounds the maximum element-wise difference between the estimated matrices obtained within  $\mathcal{B}_{\infty}(\epsilon)$ and the true precision matrices jointly across all populations.
	
	\begin{theorem}
		\label{Theorem 2}
		Let $\boldsymbol{x}_{il}\in \mathbb{R}^{p}$ denote independent and identically distributed observations from $\boldsymbol{x}_{l}$, such that $\mathbb{E}(\boldsymbol{x}_{l})=\boldsymbol{0}$ and $cov(\boldsymbol{x}_{l})=\boldsymbol{\Sigma}_{l}^{0},l=1,2,\ldots,L$. Assume that the underlying distributions satisfy the irrepresentability conditions (Assumption \ref{Assumption 5}) for an $\alpha\in (0,1]$ and the tail conditions (\ref{(17)}) with $\mathcal{T}(f,v_{l0})$ for all $l=1,2,\ldots, L$. Furthermore, suppose Assumptions \ref{Assumption 3} and \ref{Assumption 6} hold for the precision matrices $\boldsymbol{\Omega}_{l}^{0}$. Let the penalty parameters $\lambda$ and $\mu$ satisfy the restrictions that $\lambda>0$,
		\begin{equation}
			\label{(49)}
			0< \mu< min\Bigg\{\frac{\alpha\lambda}{(2-\alpha)L},\frac{\bigg(\frac{1}{L_1}+C_1\bigg)^{-2}\frac{C_1}{d}}{16\ \underset{(i,j)}{max}(\underset{l}{max}\ |\Omega_{l,ij}^{0}|+\frac{C_1}{d})^{2}}\Bigg\},
		\end{equation}
		and $(\mu+(\lambda/L))=\frac{8}{\alpha}\delta$, where $\delta=\underset{l}{max}\ \delta_{f}(n_l,p^{\gamma})$ with $\delta=o(1)$, for some $\gamma>2$ and the sample size $n_l$ satisfies (\ref{(31)}). In addition, suppose the true edge set $E$ has been identified consistently. Then, in $\mathcal{B}_{\infty}(C_1/d)$, there exists a unique local minimizer to the optimization problem (\ref{(27)}) such that, for finite $L$, we have,
		\begin{enumerate}
			\item $||\hat{\boldsymbol{\Omega}}_{l}-\boldsymbol{\Omega}_{l}^{0}||_{\infty}\leq 2\kappa_{\boldsymbol{\eta}_{l}^{0}}\bigg(1+\frac{8}{\alpha}\bigg)\delta$, for all $l=1,2,\ldots,L$ , with probability at least $1-\frac{L}{p^{\gamma-2}}\rightarrow 1$,
			\item $E(\hat{\boldsymbol{\Omega}}_{l})\subset E$ including all edges $(i,j)$ such that $|\Omega_{l,ij}^{0}|>2\kappa_{\boldsymbol{\eta}_{l}^{0}}\bigg(1+\frac{8}{\alpha}\bigg)\delta$, for all $l=1,2,\ldots,L$, with probability $1$.
		\end{enumerate}
	\end{theorem}
	The validity of the above theorem requires consistent identification of the true edge set $E$ prior to the estimation step. Lemma \ref{Lemma 0.1.1} guarantees consistent recovery of $E$ since the initial values are obtained from the graphical lasso. Therefore, combining Theorem \ref{Theorem 2} with Lemma \ref{Lemma 0.1.1}, we arrive at the next theorem. 
	\begin{theorem}
		Suppose that the assumptions of Theorem \ref{Theorem 2} hold and that the true edge set $E$ is consistently recovered with probability at least $1-\frac{1}{p^{\gamma-2}}\rightarrow 1$. Then, for any fixed $L$, the following statements hold unconditionally:
		\begin{enumerate}
			\item For all $l=1,2,\ldots,L$, $$||\hat{\boldsymbol{\Omega}}_{l}-\boldsymbol{\Omega}_{l}^{0}||_{\infty}\leq 2\kappa_{\boldsymbol{\eta}_{l}^{0}}\bigg(1+\frac{8}{\alpha}\bigg)\delta,$$with probability at least $1-\frac{L+1}{p^{\gamma-2}}\rightarrow 1$,
			\item For all $l=1,2,\ldots,L$, $E(\hat{\boldsymbol{\Omega}}_{l})\subset E$, with probability at least $1-\frac{1}{p^{\gamma-2}}\rightarrow 1$.
		\end{enumerate}
	\end{theorem}
	The proof follows directly from the law of total probability. Hence, with high probability, the local unique solution, $\hat{\boldsymbol{\Omega}}_{l}$ satisfies the restrictions:
	\begin{enumerate}
		\item If $\hat{\boldsymbol{\Theta}}=\hat{\overline{\boldsymbol{\Omega}}}=\frac{1}{L}\sum_{l=1}^L\hat{\boldsymbol{\Omega}}_{l}$, then $\hat{\Theta}_{ij}=0$ implies $\hat{\Omega}_{l,ij}=0$, for all $l=1,2,\ldots, L$,
		\item If $\hat{\Gamma}_{l,ij}=\frac{\hat{\Omega}_{l,ij}}{\hat{\Theta}_{ij}}$, when $\hat{\Theta}_{ij}\neq 0$, then $\hat{\overline{\Gamma}}_{ij}=1$.
	\end{enumerate}
	Thus, we can conclude that the local solutions of the precision matrices from solving (\ref{(27)}) are equal to a solution of (\ref{(1)}) with high probability if we use the graphical lasso estimator for the initial estimates of the parameters.

	
	\subsubsection{Sub-Gaussian distributions}
	We now discuss the consequences of Theorem \ref{Theorem 1} for distributions satisfying sub-Gaussian tail bounds (Definition \ref{Definition 3}).
	\begin{corollary}
		\label{Corollary 1.1}
		Suppose the samples satisfy the sub-Gaussian condition (Definition \ref{Definition 3}) with a common parameter $K_1>0$. Under the same conditions as Theorem \ref{Theorem 1}, if the sample sizes satisfy the condition
		\begin{equation}
			\label{(84)}
			n_l>C_{1l}d^{2}\bigg(1+\frac{8}{\alpha}\bigg)^{2}\ \log(4p^{\gamma}),
		\end{equation}
		where $C_{1l}=(48\sqrt{2}\boldsymbol{\kappa}_{\eta_{l}^{0}}(1+12K_1^2)\ max(\Sigma_{l,ii}^{0})\ \underset{l_1}{max}\ max\{\boldsymbol{\kappa}_{\Sigma_{l_1}^{0}},\boldsymbol{\kappa}_{\Sigma_{l_1}^{0}}^{3}\boldsymbol{\kappa}_{\eta_{l_1}^{0}}\})^{2}$, and if $\frac{\log(p)}{n}=o(1)$, then with probability at least $1-\frac{L}{p^{\gamma-2}}$,
		\begin{equation}
			\label{(85)}
			||\hat{\boldsymbol{\Omega}}_{l}-\boldsymbol{\Omega}_{l}^{0}||_{\infty}\leq 16\sqrt{2}\boldsymbol{\kappa}_{\eta_{l}^{0}}\bigg(1+\frac{8}{\alpha}\bigg)\underset{l}{max}(1+12K_1^2)\ max(\Sigma_{l,ii}^{0})\sqrt{\frac{\log(4p^{\gamma})}{n_l}},
		\end{equation}
		for $l=1,2,\ldots,L$.
	\end{corollary}
	If $\boldsymbol{\kappa}_{\eta_{l}^{0}},\boldsymbol{\kappa}_{\Sigma_{l}^{0}}$ and $\alpha$ remain constant as a function of $n,p$ and $d$, for all $l=1,2,\ldots,L$, then from Corollary \ref{Corollary 1.1}, we have $||\hat{\boldsymbol{\Omega}}_{l}-\boldsymbol{\Omega}_{l}^{0}||_{\infty}=O_{p}(\sqrt{\log(p)/n})$, by taking sample sizes $n_l$ such that $n_l=\phi(d^2\ \log(4p^{\gamma})$ and $\delta=O(\sqrt{\log(p)/n})$, since $\delta_{f}(n_l,p^{\gamma})=O(\sqrt{\log(p)/n})$.
	\begin{Remark}
		Since $\delta=O(\sqrt{\frac{\log(p)}{n}})$, Assumption 6.\ref{Assumption 6.2} gives us
		\begin{equation}
			\label{(86)}
			\underset{l}{max}\ \underset{(i,j)\in E}{max}\frac{|\Omega_{l,ij}^{0}-\overline{\Omega}_{ij}^{0}|}{|\overline{\Omega}_{ij}^{0}|}=O\bigg(\bigg(\frac{\log(p)}{n}\bigg)^{1/8}\bigg),       
		\end{equation}
		if $\frac{\log(p)}{n}=o(1)$. Moreover, from the restriction of the sample size $n_l$ in (\ref{(31)}) and interpretation of the inverse function, we also have
		\begin{eqnarray}
			\notag 
			\delta_{f}(n_l,p^{\gamma})&\leq& \frac{1}{6d\bigg(1+\frac{8}{\alpha}\bigg)\boldsymbol{\kappa}_{\eta_{l}^{0}}}C_1\text{, for all }l=1,2,\ldots,L\\
			\label{(87)}
			\implies \frac{6}{C_1} \bigg(1+\frac{8}{\alpha}\bigg)\delta & \leq & \frac{1}{\boldsymbol{\kappa}_{\eta_{l}^{0}}d}
		\end{eqnarray}
		Hence, when $\boldsymbol{\kappa}_{\eta_{l}^{0}},\boldsymbol{\kappa}_{\Sigma_{l}^{0}}$ and $\alpha$ remain constant as a function of $n,p$ and $d$, for all $l=1,2,\ldots,L$, from (\ref{(25)}) we have
		\begin{equation}
			\label{(88)}
			\frac{1}{d}=\phi(\delta)=\phi(\sqrt{\frac{\log(p)}{n}}).   
		\end{equation}
		Hence, if $\frac{\log(p)}{n}=o(1)$, then Assumption 6.\ref{Assumption 6.1} simplifies to $$|\overline{\Omega}_{ij}^{0}|=\phi\bigg(\bigg(\frac{\log(p)}{n}\bigg)^{1/8}\bigg),$$
		for all $(i,j)\in E$.
	\end{Remark}
	\section{Applications}\label{Section 4}
	\subsection{Application on synthetic datasets}\label{Section 4.1}
	To evaluate the proposed methodology, we consider simulated datasets generated from chain and star graph structures. These graph configurations are chosen because they satisfy the irrepresentability conditions required for the theoretical results, thereby allowing us to empirically investigate model selection consistency and supremum norm differences. Before presenting the simulation results, we first describe the procedure used for selecting the penalty parameters.
	
	According to our theoretical framework, these parameters should be $\lambda=LC_1\sqrt{\frac{\log\ p}{n}}$, and $\mu=C_2\sqrt{\frac{\log\ p}{n}}$. In practice, however, the tuning parameters $(C_1,C_2)$ are selected using the Extended Bayesian Information Criterion (EBIC). The optimal pair $(C_1,C_2)$ is determined through a grid-search procedure by selecting the values that minimize the EBIC with $\gamma=0.5$.
	
	For the chain graph setting, the population precision matrices follow a tridiagonal structure tridiag($\rho_{l},1,\rho_{l})$ with $\rho_{1}=0.2$, and $\rho_{2}=0.35$ for populations 1 and 2, respectively. For the star graph setting, a hub node $X_{u}$ is randomly selected and connected to $d_{max}$ randomly chosen coordinates, which remain mutually disconnected. Let $b$ denote the index set of these randomly chosen coordinates. The precision matrices are then defined by
	\begin{eqnarray*}
		\Omega_{l,ij} & = & \begin{cases}
			c_{l} & ,i=j\\
			\delta_{l} & ,i=u,j\in b\text{ or }i\in b,j=u\\
			0 & ,o.w.,
		\end{cases},
	\end{eqnarray*}
	where $(c_1,\delta_1)=(2,0.3)$ and $(c_2,\delta_2)=(2.5,0.45)$. In both settings, observations are generated from multivariate Gaussian distributions with zero mean and covariance matrix $\left(\boldsymbol{\Omega}_{l}\right)^{-1}$, for $l=1,2$, across dimensions $p=25,50,75,100$, and sample sizes $n$ ranging from $50$ to $600$, with the entire simulation procedure repeated $B=100$ times.
	\begin{figure}[t]
		\centering{}\caption{Star-shaped graph structures used in the simulation study. Left: Population 1 with hub-edge weight 0.3 and diagonal entries 2. Right: Population 2 with hub-edge weight 0.45 and diagonal entries 2.5. Both networks share the same hub-and-spoke topology with $p=25$ nodes and maximum degree $d_{\max} = 20$, differing only in the magnitude of the non-zero precision matrix entries.}
		\includegraphics[scale=0.29]{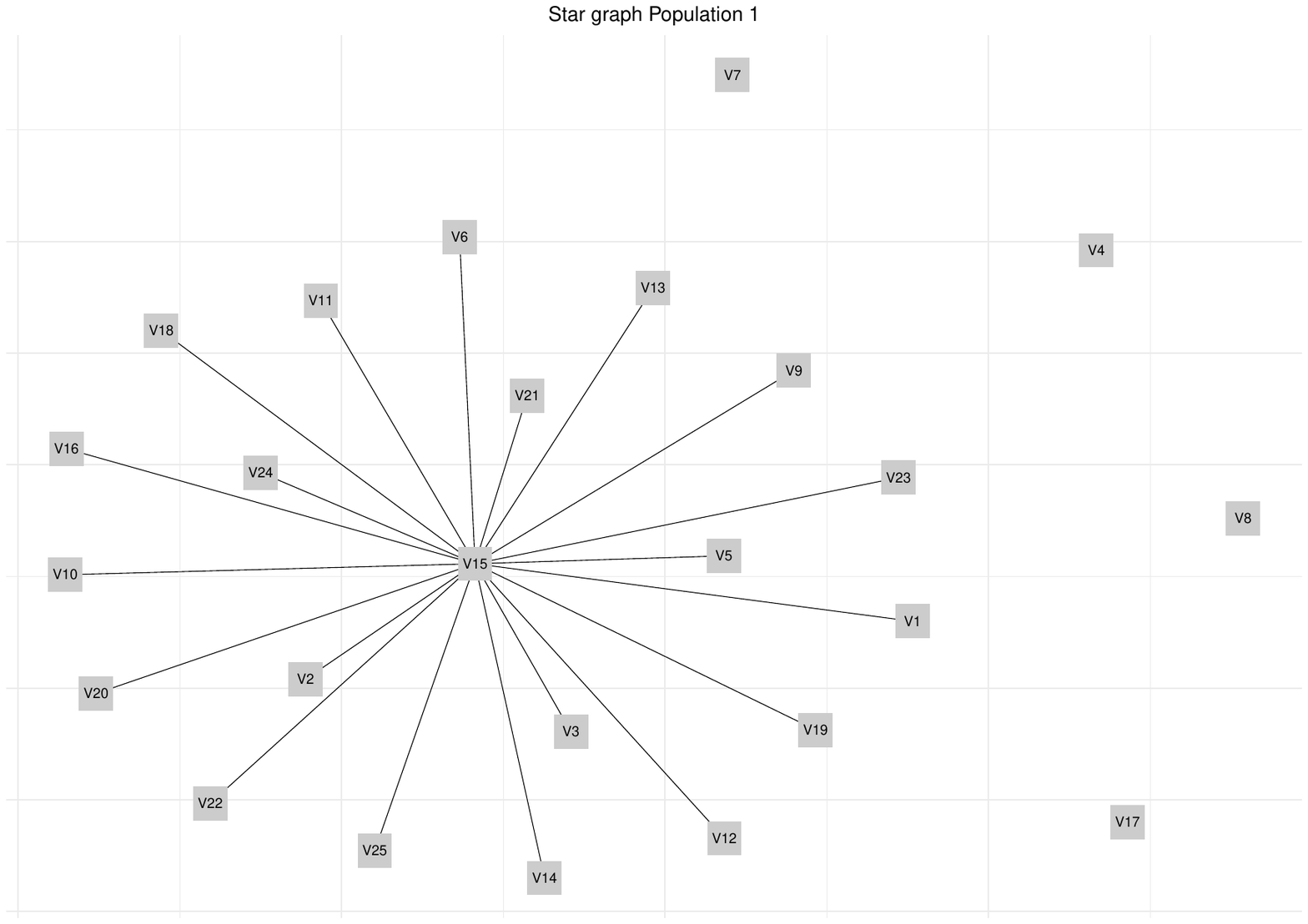}\includegraphics[scale=0.29]{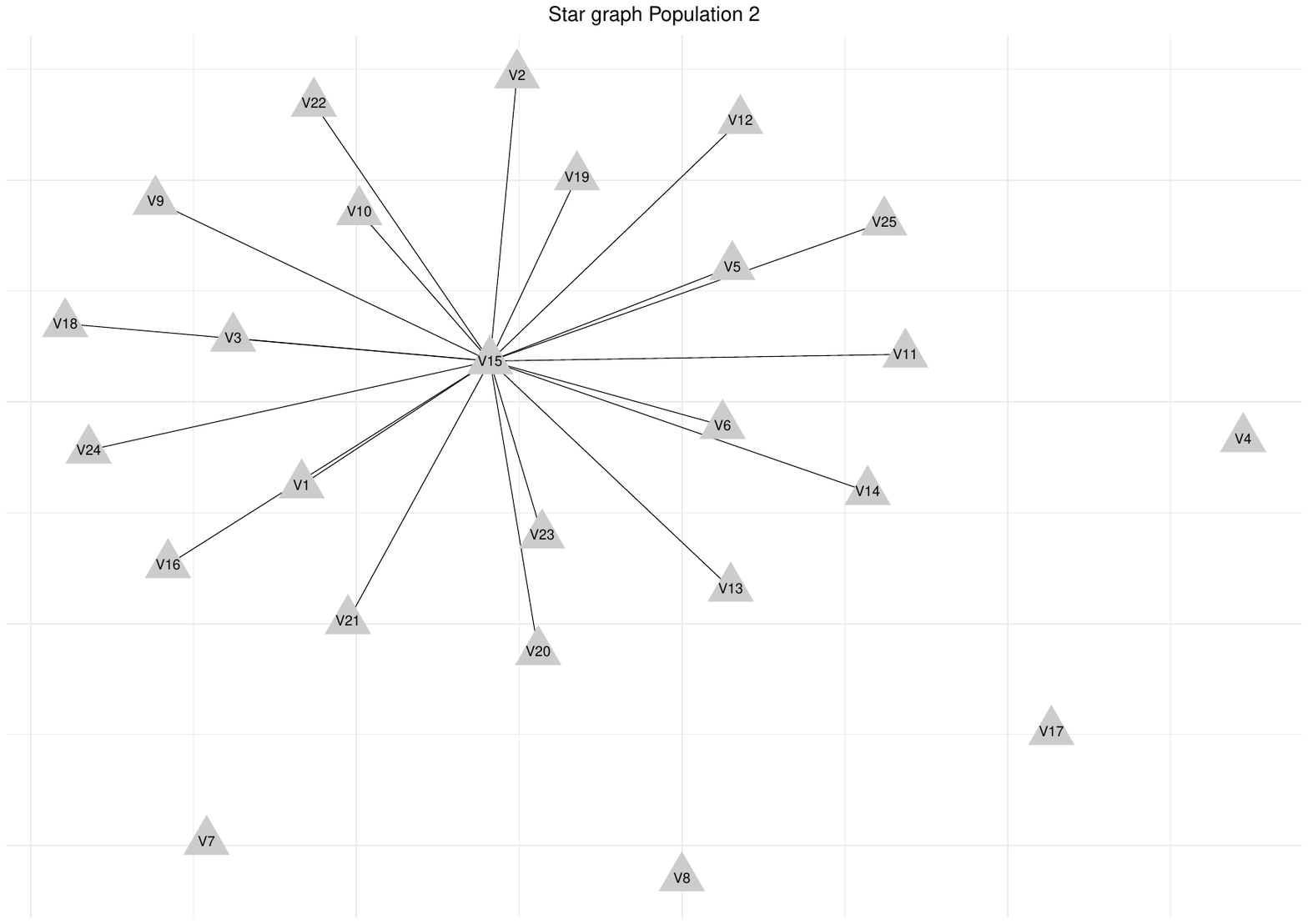}
		\label{Figure 1}
	\end{figure}
	
	\subsubsection{Model consistency against sample size}
	The results presented here illustrate the probability of correctly recovering the signed edge-set for both population models. For a fixed sample size $n$ and dimension $p$, we report the proportion of samples that successfully identify the edge set whose signs match those of the true precision matrices across both population models. For the star graphs, a maximum degree of $d_{max} = 20$ is assumed.
	
	As sample size grows, the proportion of correctly recovered signed edge sets increases steadily toward one for both chain and star graph structures, confirming the model selection consistency. Higher-dimensional settings require larger sample sizes to achieve the similar level of consistency reflecting the expected dependence of convergence rates on dimensionality (Figure \ref{Figure 2}).
	\begin{figure}[t]
		\caption{Proportion of simulation replications ($B = 100$) in which the Mglasso correctly recovers the signed edge set for both population models simultaneously, plotted against sample size for chain (left) and star (right) graph structures across dimensions $p \in \{25, 50, 75, 100\}$. Curves shift rightward with increasing dimension, reflecting the slower convergence to model consistency in higher-dimensional settings.}
		
		\centering{}\includegraphics[scale=0.29]{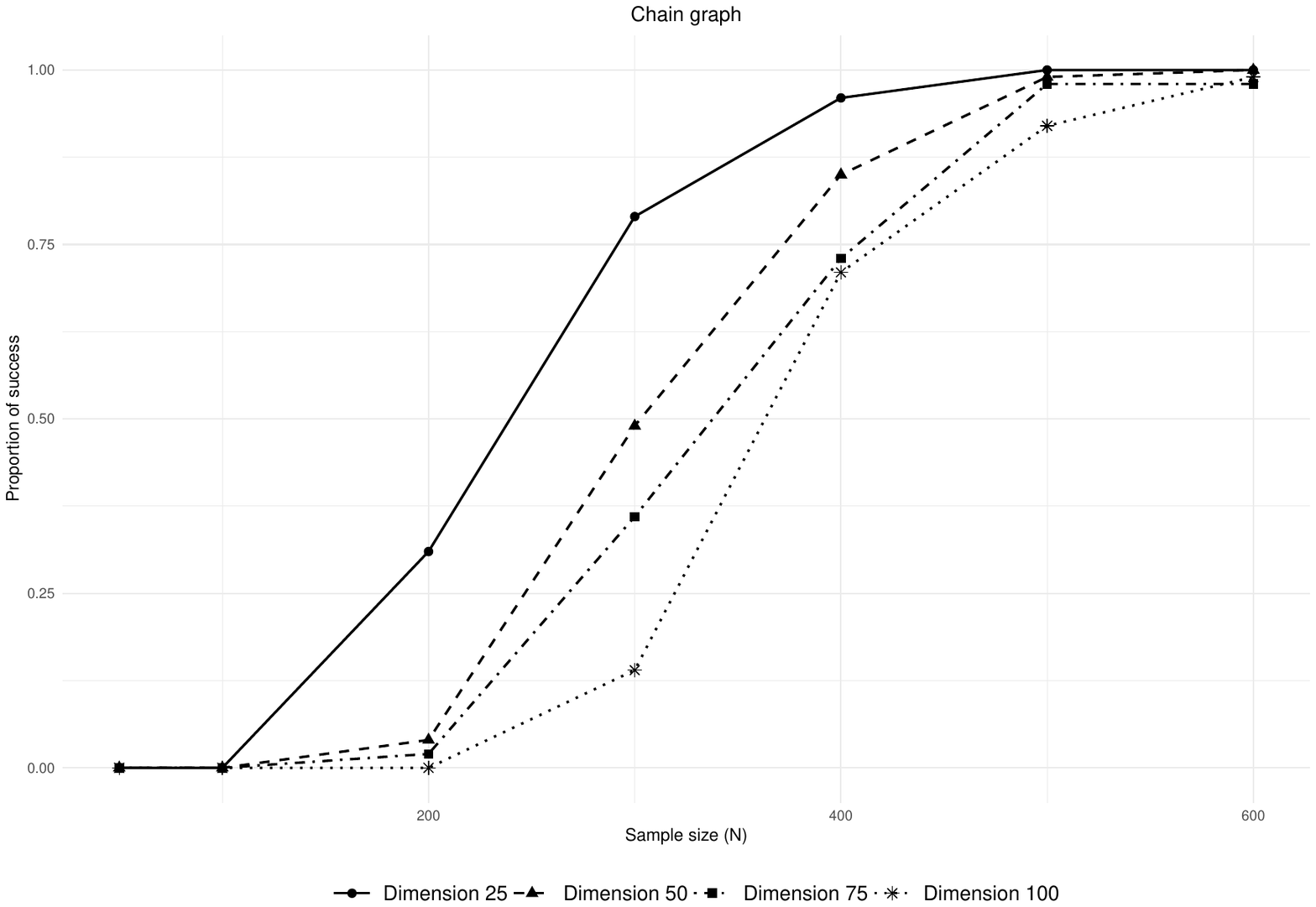}\includegraphics[scale=0.29]{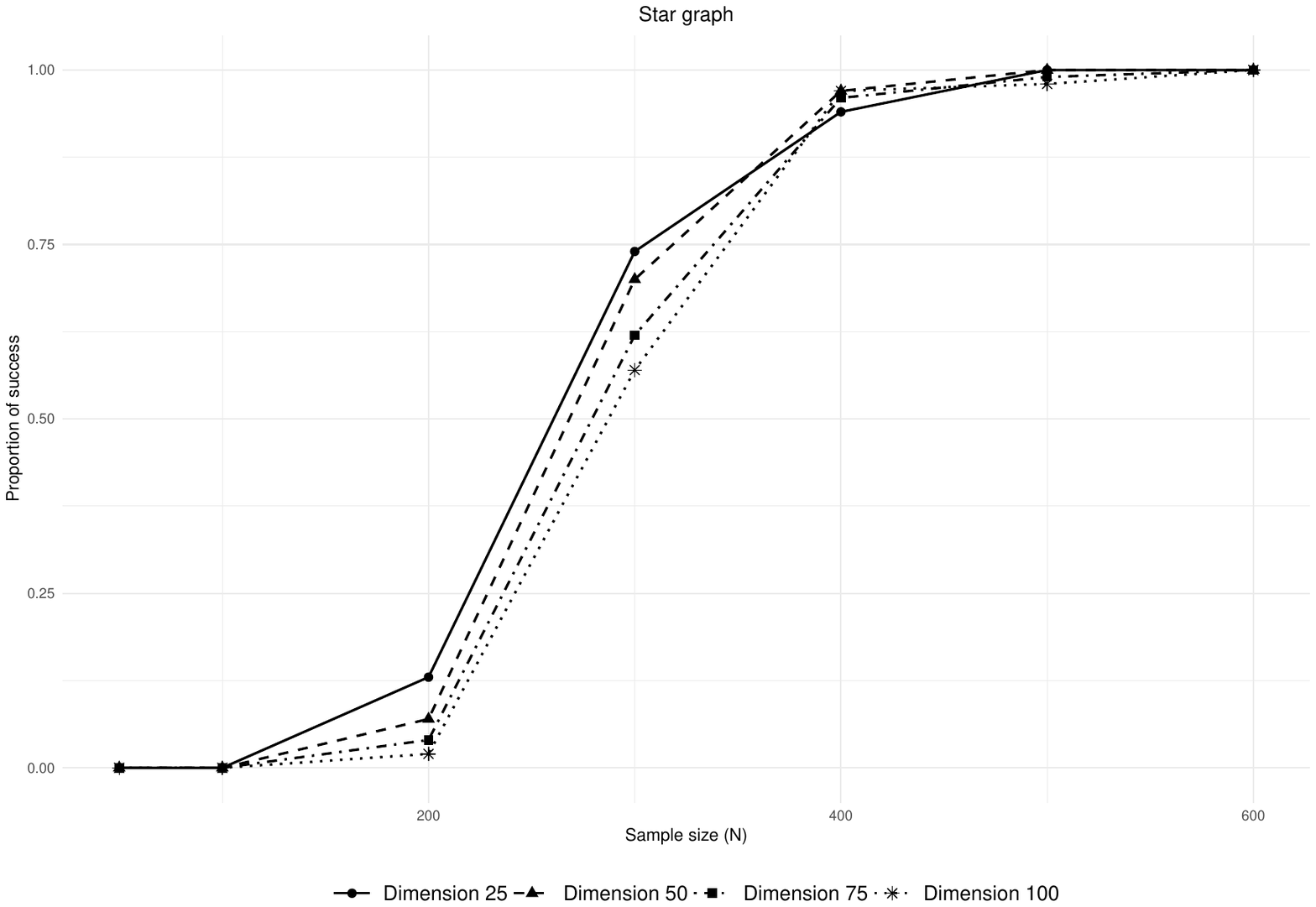}
		\label{Figure 2}
	\end{figure}
	
	\subsubsection{Supnorm convergence against sample size}
	We present the average supremum norm distance between the predicted and true precision matrices, calculated across both populations for a fixed sample size and dimension. The average supremum norm distance between the estimated and true precision matrices decreases monotonically with increasing sample size across both graph structures and all dimensions considered. Higher-dimensional settings exhibit slower convergence, with the gap between dimensions narrowing at larger sample sizes (Figures \ref{Figure 3} and \ref{Figure 4}).
	\begin{figure}[t]
		\caption{Average supremum norm distance between the Mglasso estimates and the true precision matrices for Population 1 (left) and Population 2 (right) in the chain graph setting, plotted against sample size for $p \in \{25, 50, 75, 100\}$. All curves decrease monotonically with sample size, with higher dimensions exhibiting slower convergence.}
		
		\centering{}\includegraphics[scale=0.29]{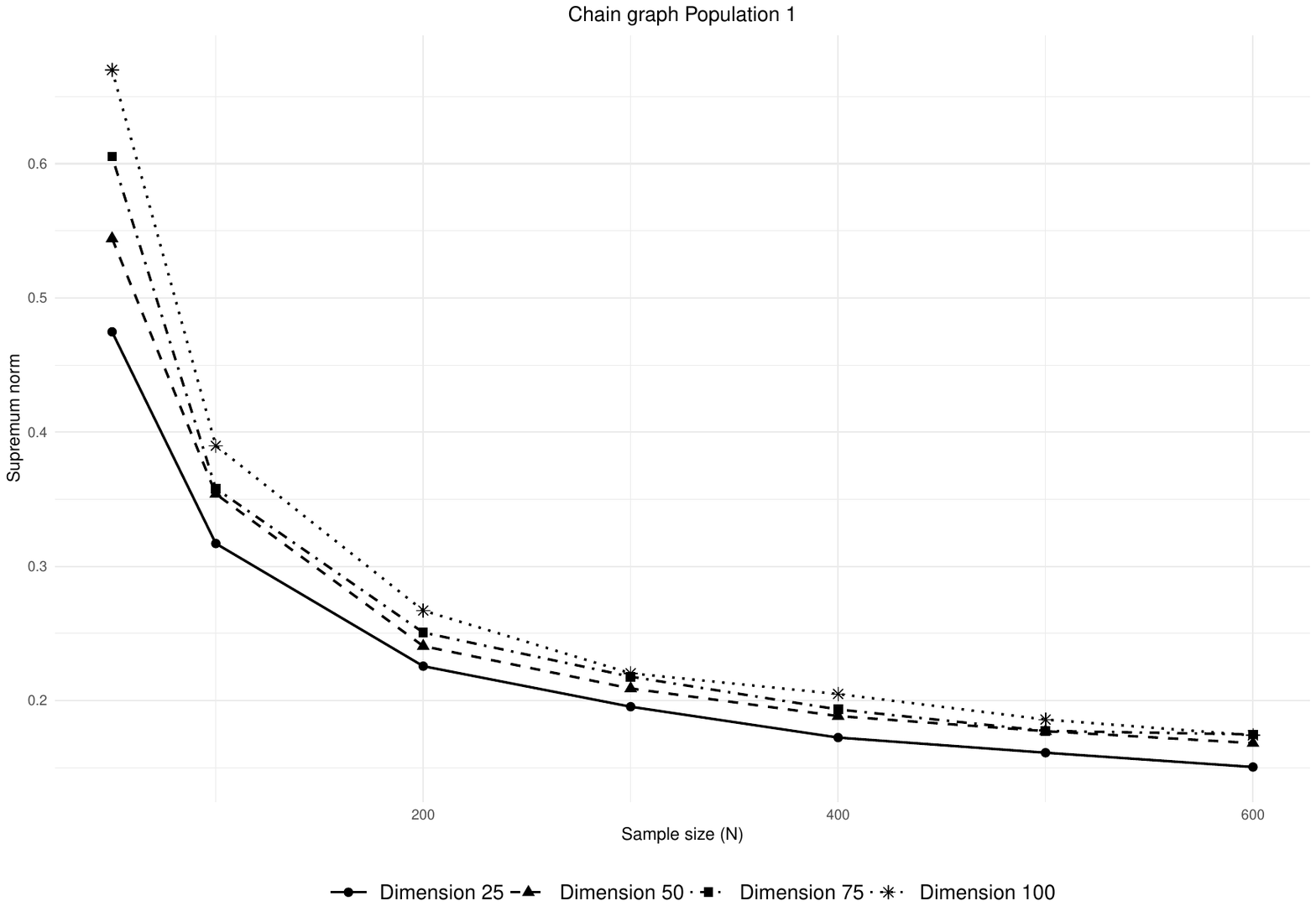}\includegraphics[scale=0.29]{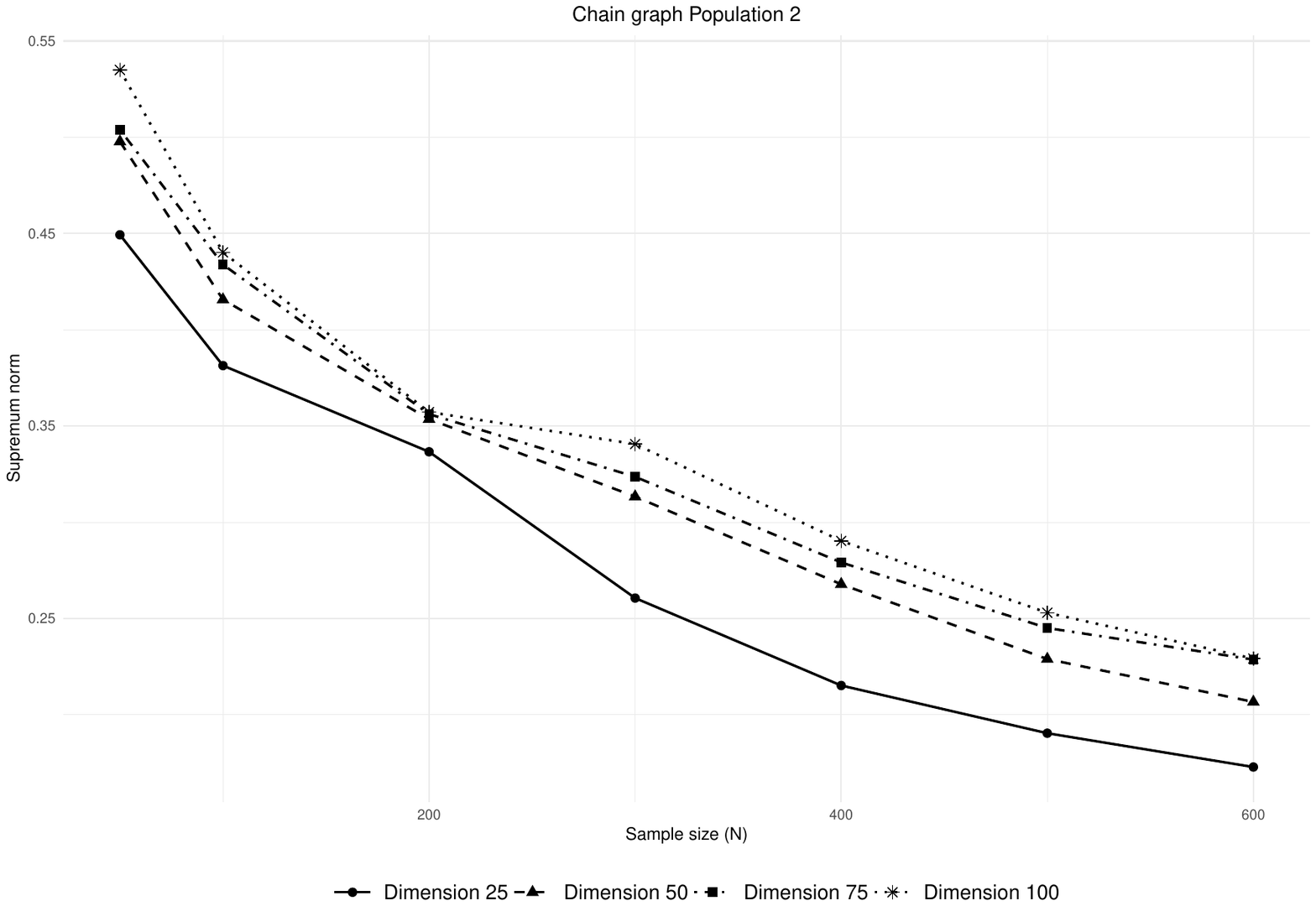}
		\label{Figure 3}
	\end{figure}
	\begin{figure}[t]
		\caption{Average supremum norm distance between the Mglasso estimates and the true precision matrices for Population 1 (left) and Population 2 (right) in the star graph setting, plotted against sample size for $p \in \{25, 50, 75, 100\}$. The star graph yields larger absolute estimation errors than the chain graph, consistent with its higher maximum degree, but the same monotone decreasing trend with sample size is observed.}
		
		\centering{}\includegraphics[scale=0.29]{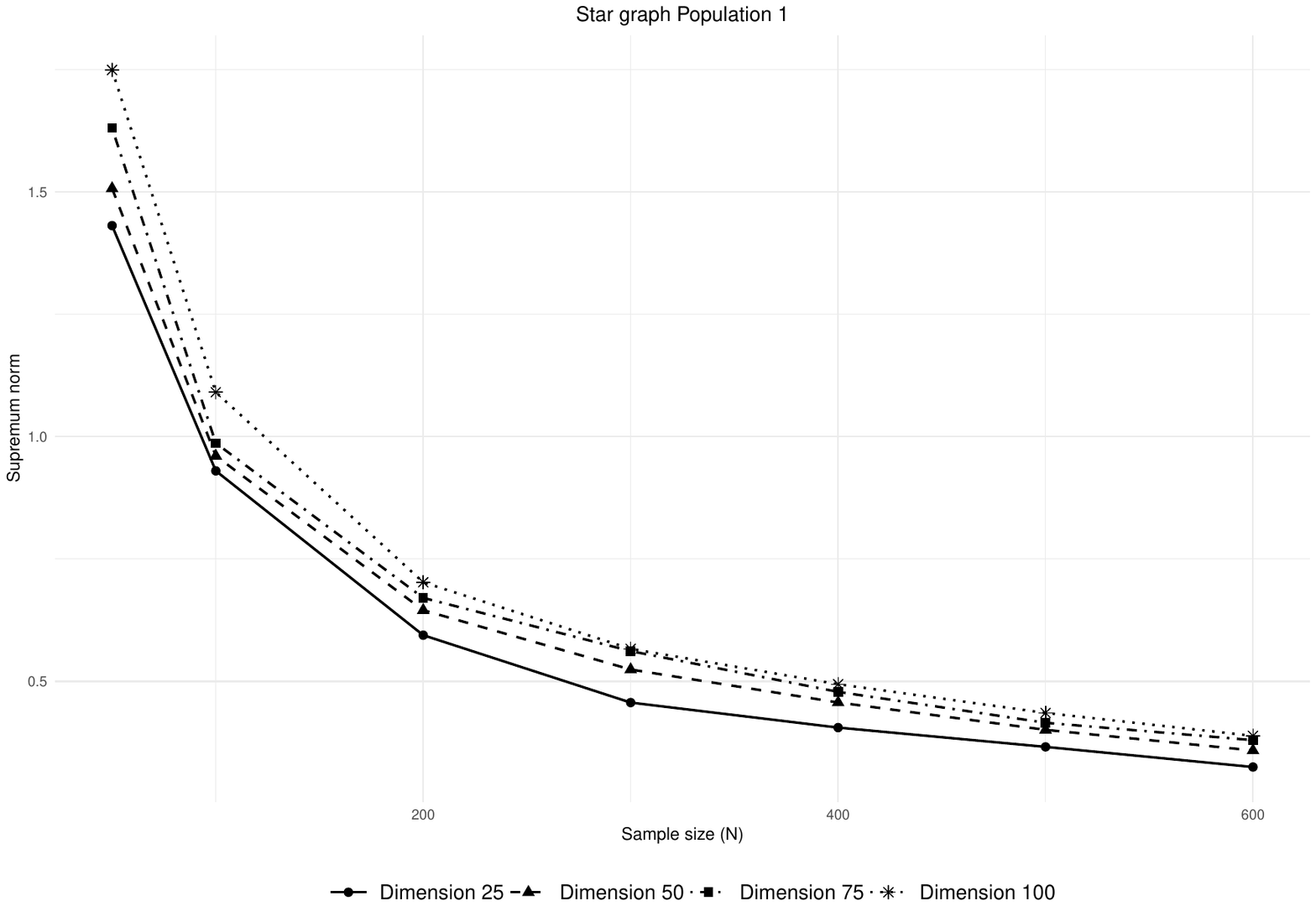}\includegraphics[scale=0.29]{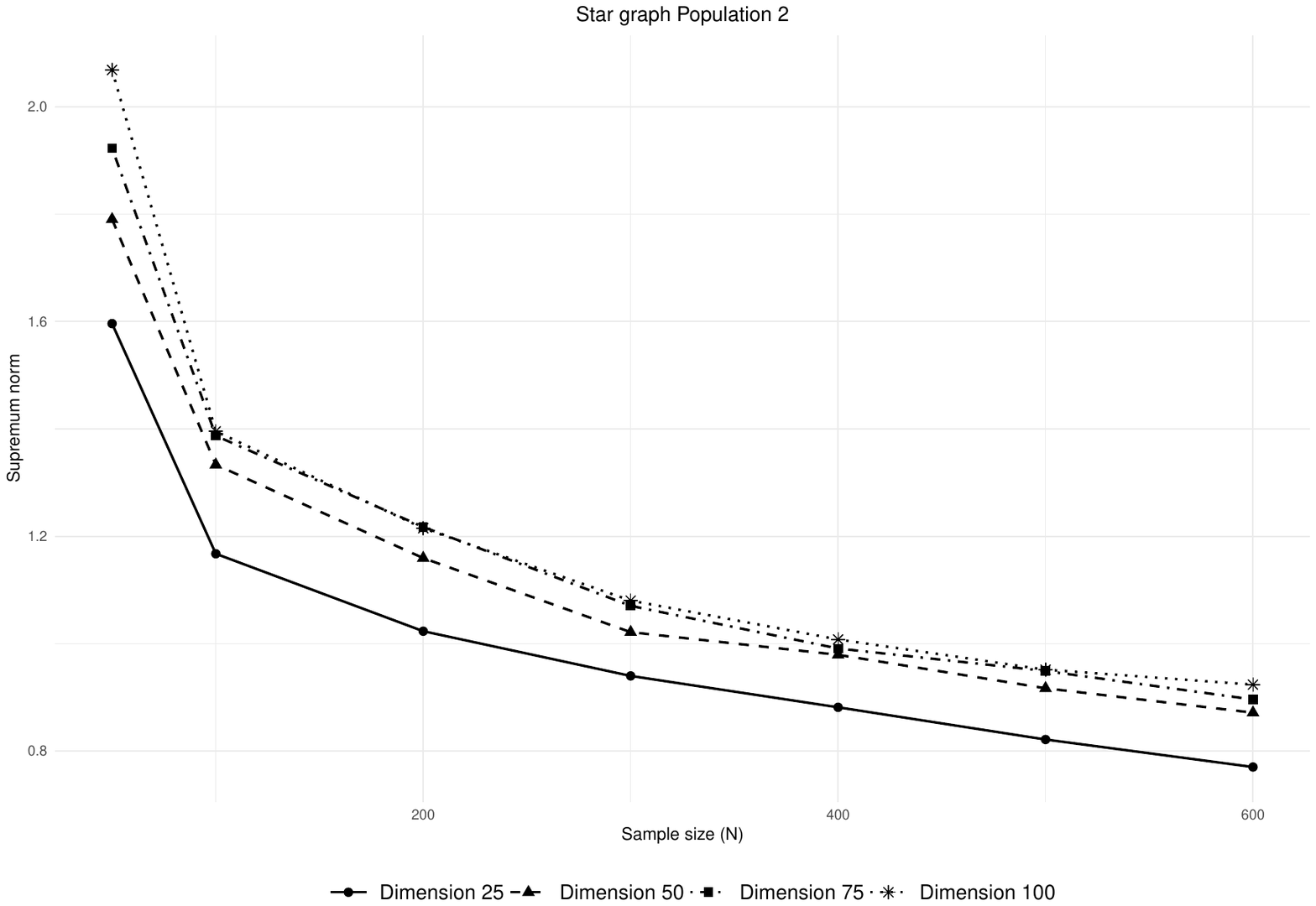}
		\label{Figure 4}
	\end{figure}
	\subsubsection{True positives and false positives against sample size}
	The number of correctly identified edges increases with sample size for both graph types, while false positives remain small and does not exhibit a systematic upward trend, indicating accurate support recovery without growing false discovery burden. Figures \ref{Figure 5} and \ref{Figure 6} illustrate these findings in terms of true and false positive edges.
	\begin{figure}[t]
		\caption{Number of true positive edges (top row) and false positive edges (bottom row) identified by the Mglasso estimator for Population 1 (left column) and Population 2 (right column) in the chain graph setting, plotted against sample size for $p \in \{25, 50, 75, 100\}$. True positive counts increase toward their theoretical maximum with sample size while false positives remain near zero, confirming support recovery consistency.}
		
		\centering{}\includegraphics[scale=0.47]{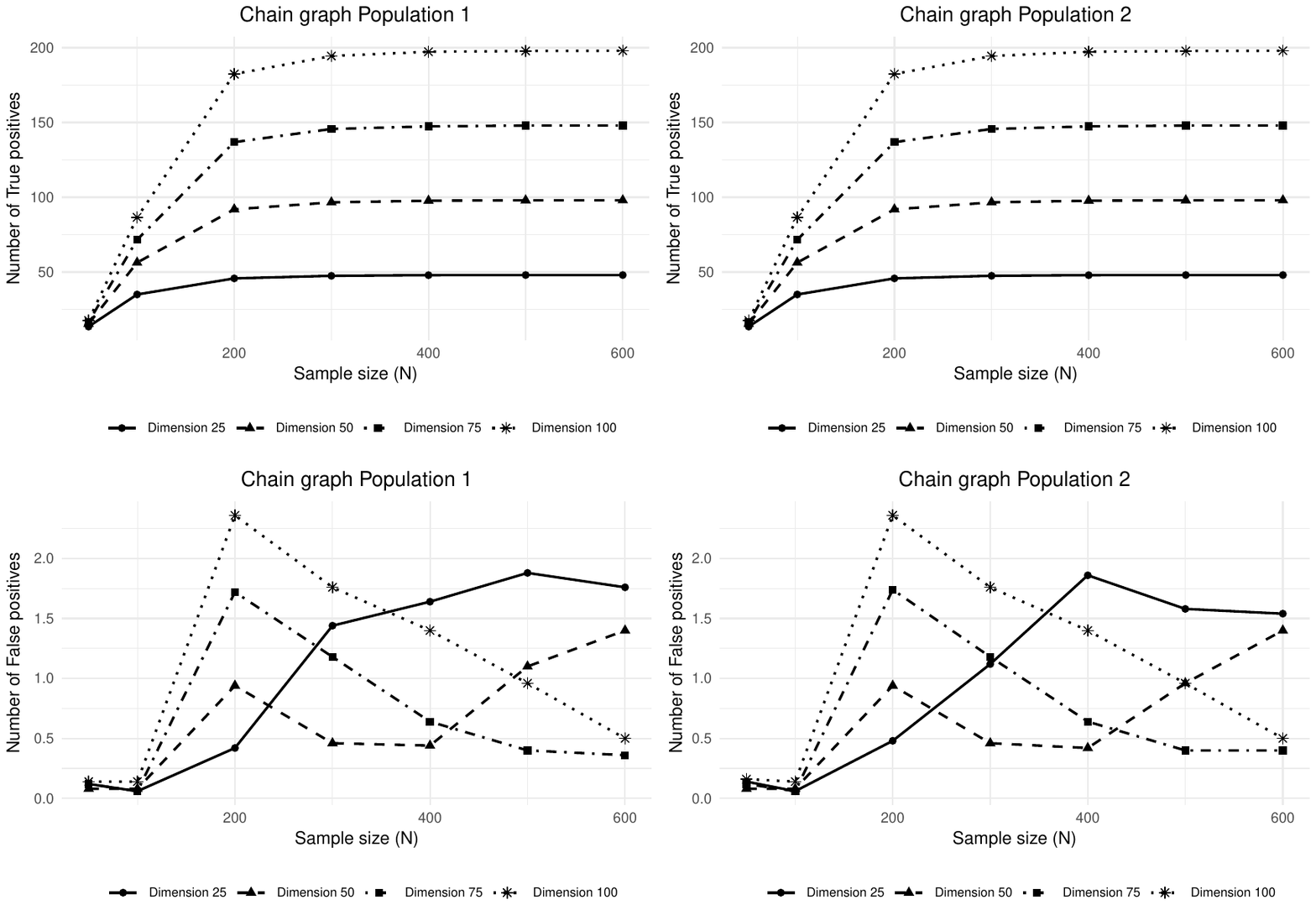}
		\label{Figure 5}
	\end{figure}
	\begin{figure}[t]
		\caption{. Number of true positive edges (top row) and false positive edges (bottom row) identified by the Mglasso estimator for Population 1 (left column) and Population 2 (right column) in the star graph setting, plotted against sample size for $p \in \{25, 50, 75, 100\}$. The pattern mirrors Figure \ref{Figure 5}, with true positives increasing and false positives remaining negligible across all dimensions.}
		
		\centering{}\includegraphics[scale=0.47]{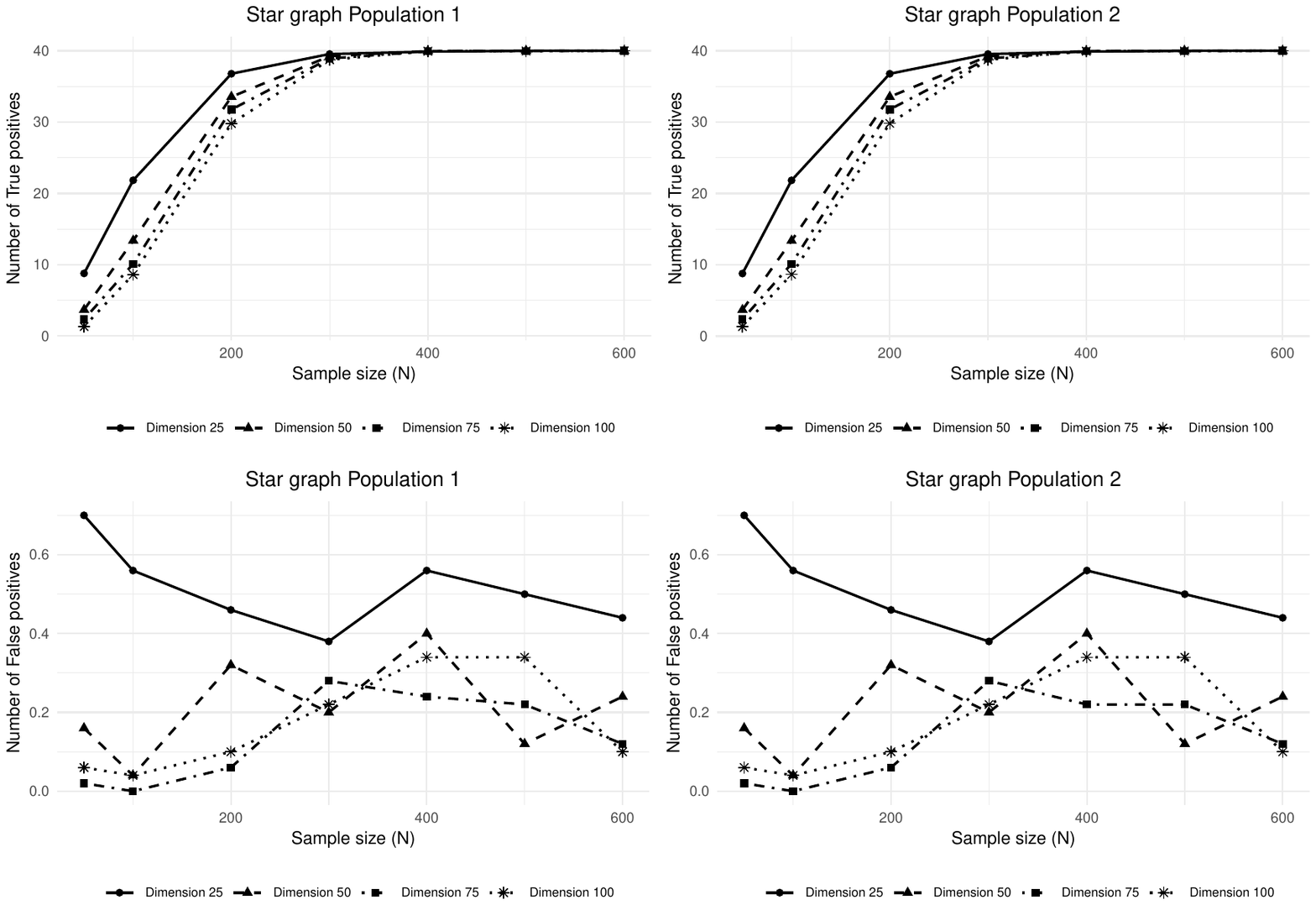}
		\label{Figure 6}
	\end{figure}
	\subsubsection{Comparison against Group Graphical Lasso (GGL)}
	We compare the performance of the Mglasso with GGL, proposed by \cite{danaher2014joint}. The Mglasso achieves higher model selection consistency than the GGL across all dimensions and both graph structures, attaining the correct signed edge set with substantially smaller sample sizes (Figures \ref{Figure 7}–\ref{Figure 8}). In terms of estimation accuracy, Mglasso yields similar performance to GGL, with both methods demonstrating comparable average supremum norm distances from the true precision matrices across most dimension–sample size combinations for chain graphs; however, it is worth noting that this estimation error is noticeably more prominent in star graph topologies (Figures \ref{Figure 9}-\ref{Figure 10}). Ultimately, this demonstrates that while the proposed common-sparsity decomposition achieves faster and more parsimonious structural recovery, it maintains highly competitive estimation accuracy.
	\begin{figure}[t]
		\caption{Model selection consistency comparison between the Mglasso and GGL on chain graph simulations, for $p \in \{25, 50, 75, 100\}$. Mglasso achieves the correct signed edge set at substantially smaller sample sizes than GGL across all dimensions, with the gap widening in higher-dimensional settings.}
		
		\centering{}\includegraphics[scale=0.47]{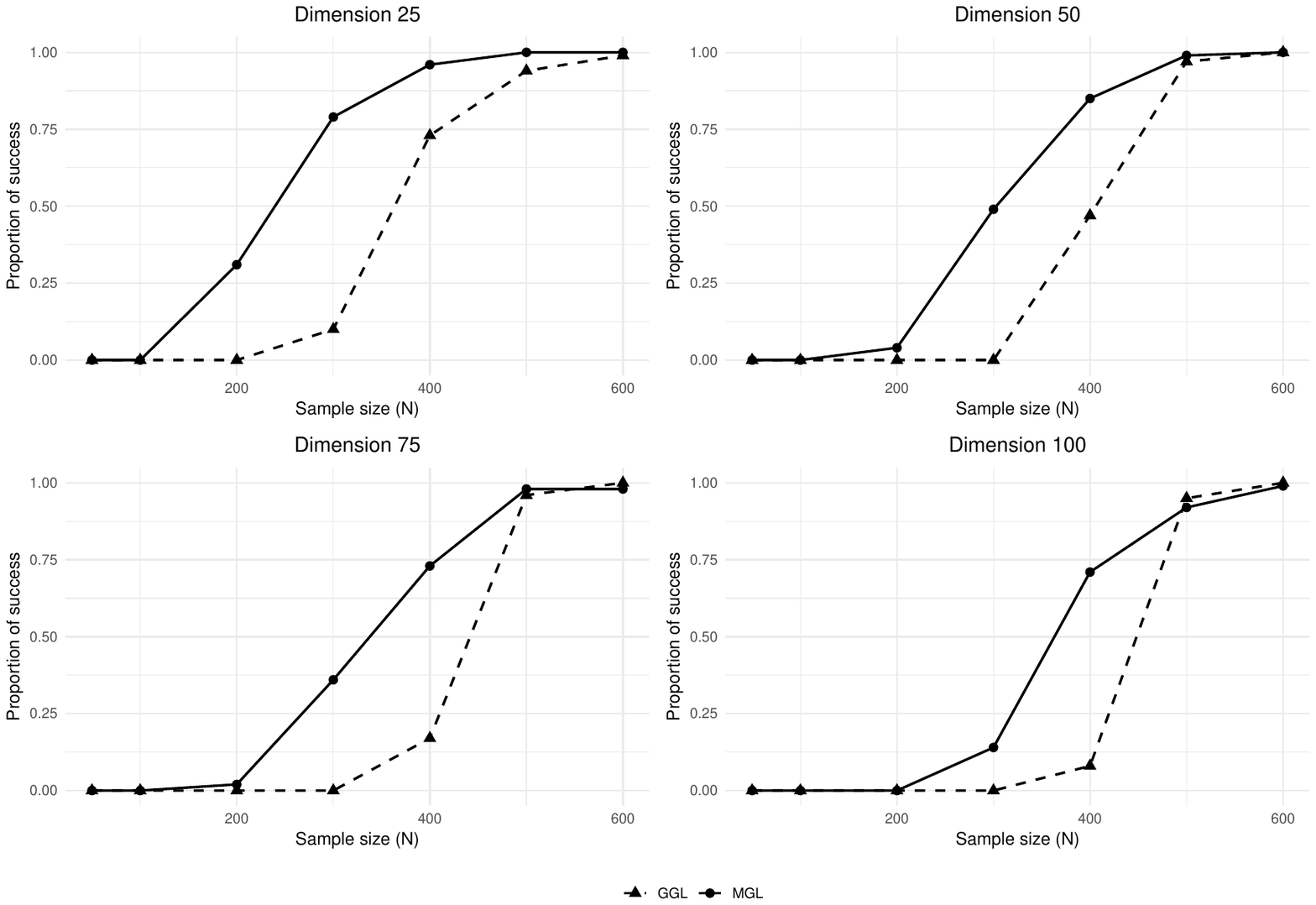}
		\label{Figure 7}
	\end{figure}
	\begin{figure}[t]
		\caption{Model selection consistency comparison between the Mglasso and GGL on star graph simulations, for $p \in \{25, 50, 75, 100\}$. Mglasso consistently reaches full consistency at smaller sample sizes, and the advantage over GGL is more pronounced for the star graph than for the chain graph, reflecting the greater benefit of enforcing a common sparsity structure in denser network settings.}
		
		\centering{}\includegraphics[scale=0.47]{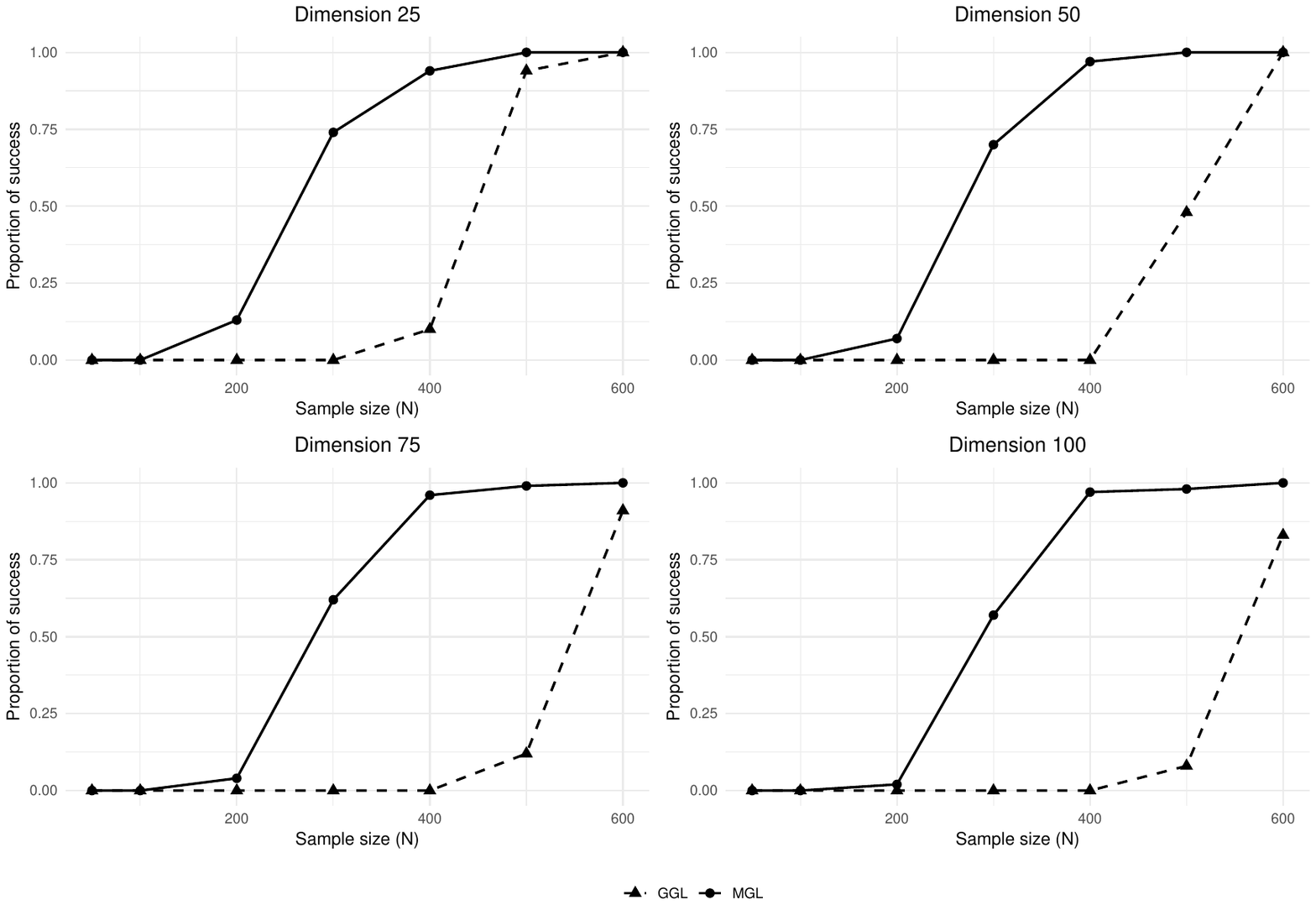}
		\label{Figure 8}
	\end{figure}
	
	\begin{figure}
		\centering
		\caption{ Average supremum norm distance from the true precision matrices, comparing Mglasso and GGL for Population 1 (panel a) and Population 2 (panel b) in the chain graph setting, across $p \in \{25, 50, 75, 100\}$. Mglasso achieves uniformly lower estimation error than GGL at all sample sizes and dimensions, with the gap diminishing as sample size grows.}
		
		\subfloat[Population 1]{
			
			\includegraphics[scale=0.47]{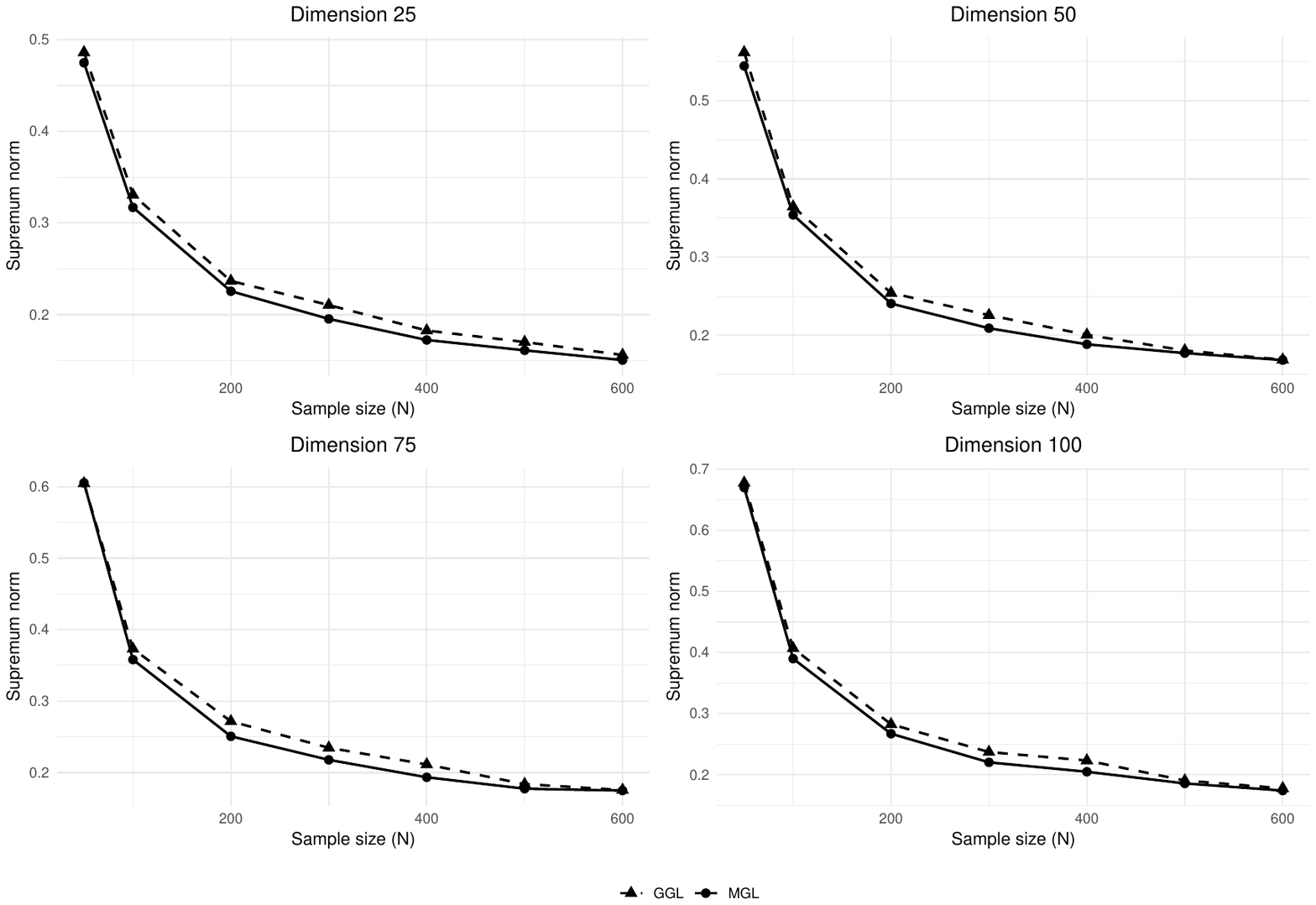}}
		
		\subfloat[Population 2]{
			
			\includegraphics[scale=0.47]{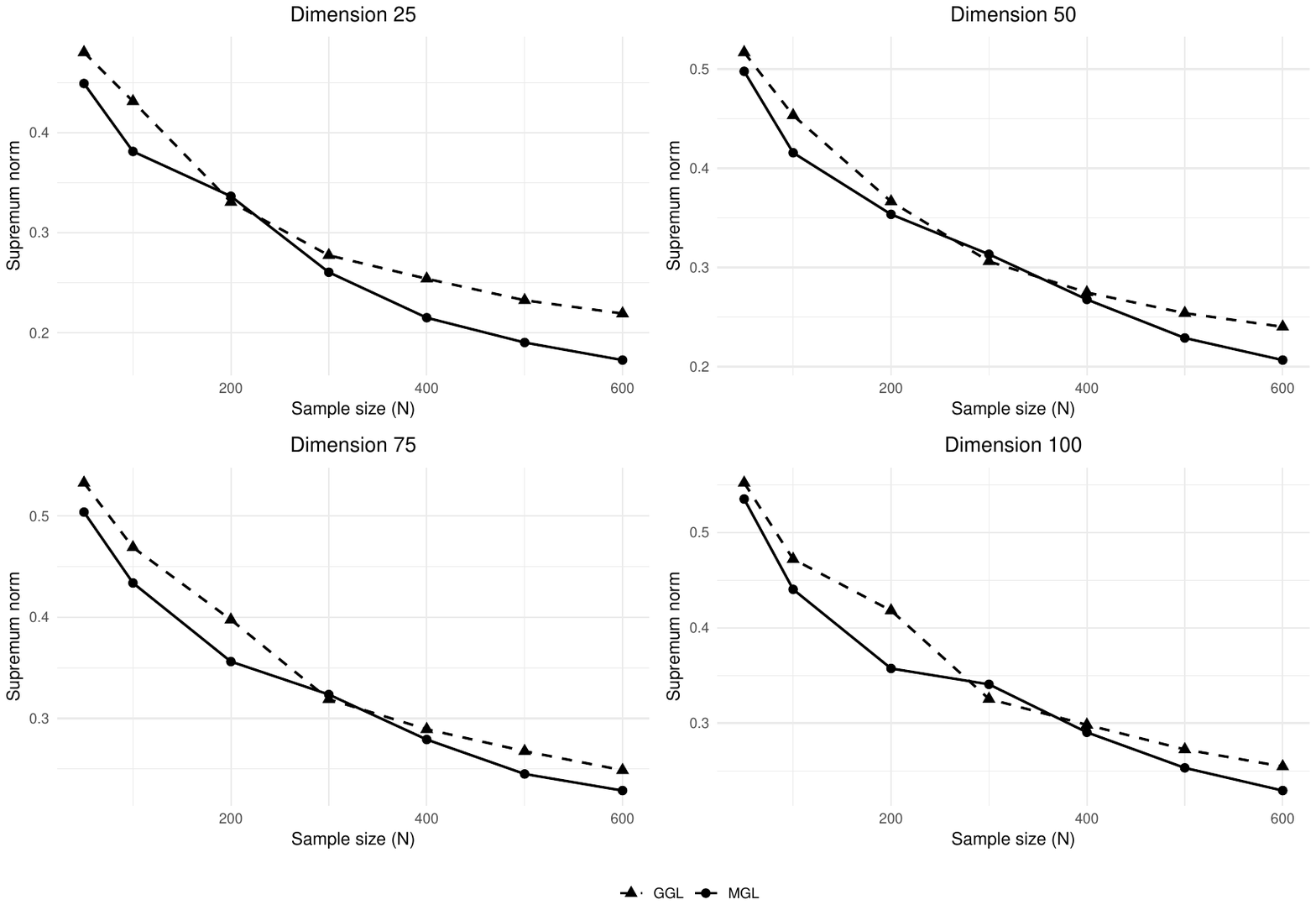}}
		\label{Figure 9}
	\end{figure}
	
	\begin{figure}
		\centering
		\caption{Average supremum norm distance from the true precision matrices, comparing Mglasso and GGL for Population 1 (panel a) and Population 2 (panel b) in the star graph setting, across $p \in \{25, 50, 75, 100\}$. Mglasso yields substantially lower supremum norm errors than GGL throughout, and the advantage is more pronounced in the star graph than in the chain graph setting, consistent with the higher connectivity of the star network amplifying the benefit of shared-sparsity estimation}
		
		\subfloat[Population 1]{
			
			\includegraphics[scale=0.47]{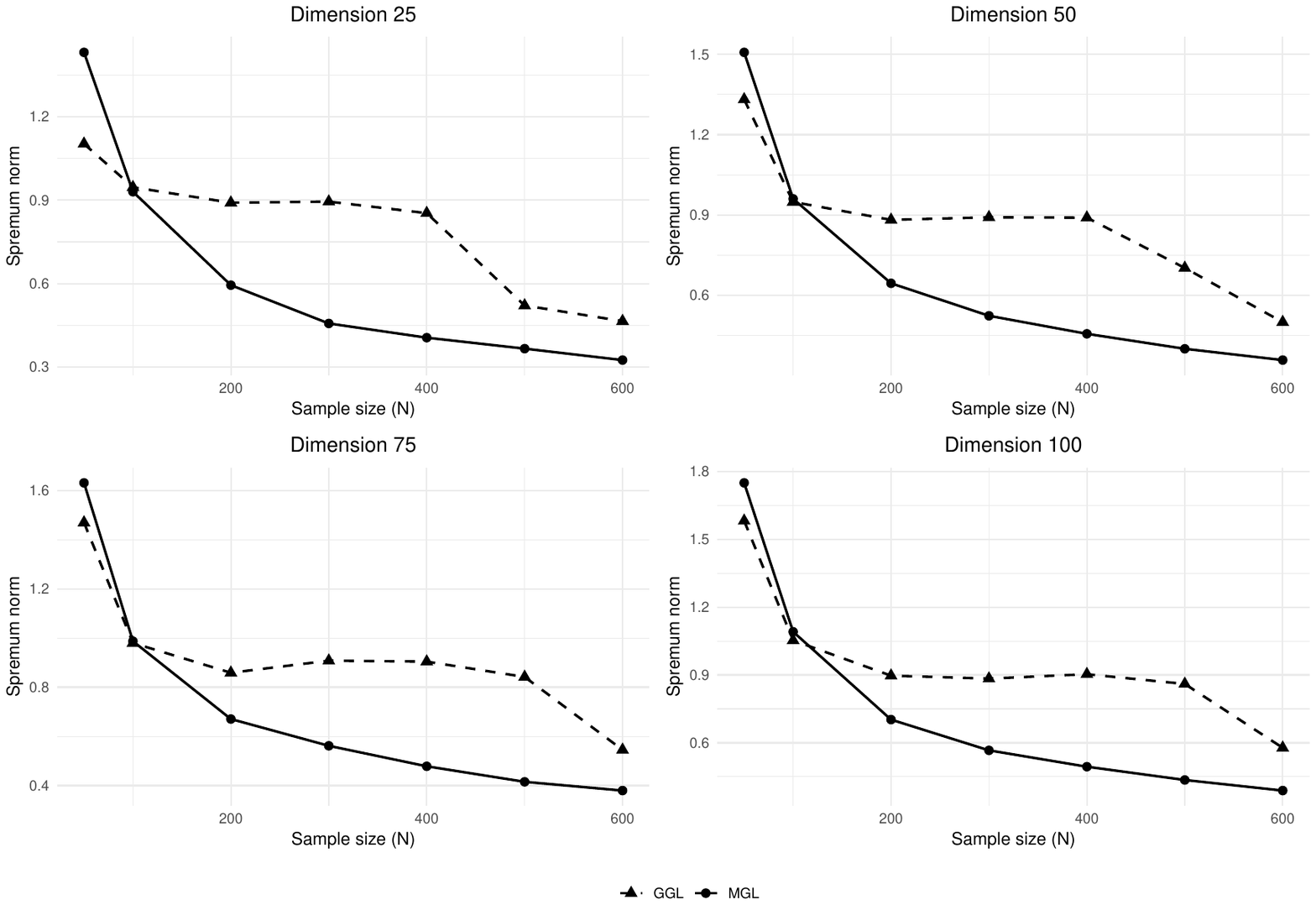}}
		
		\subfloat[Population 2]{
			
			\includegraphics[scale=0.47]{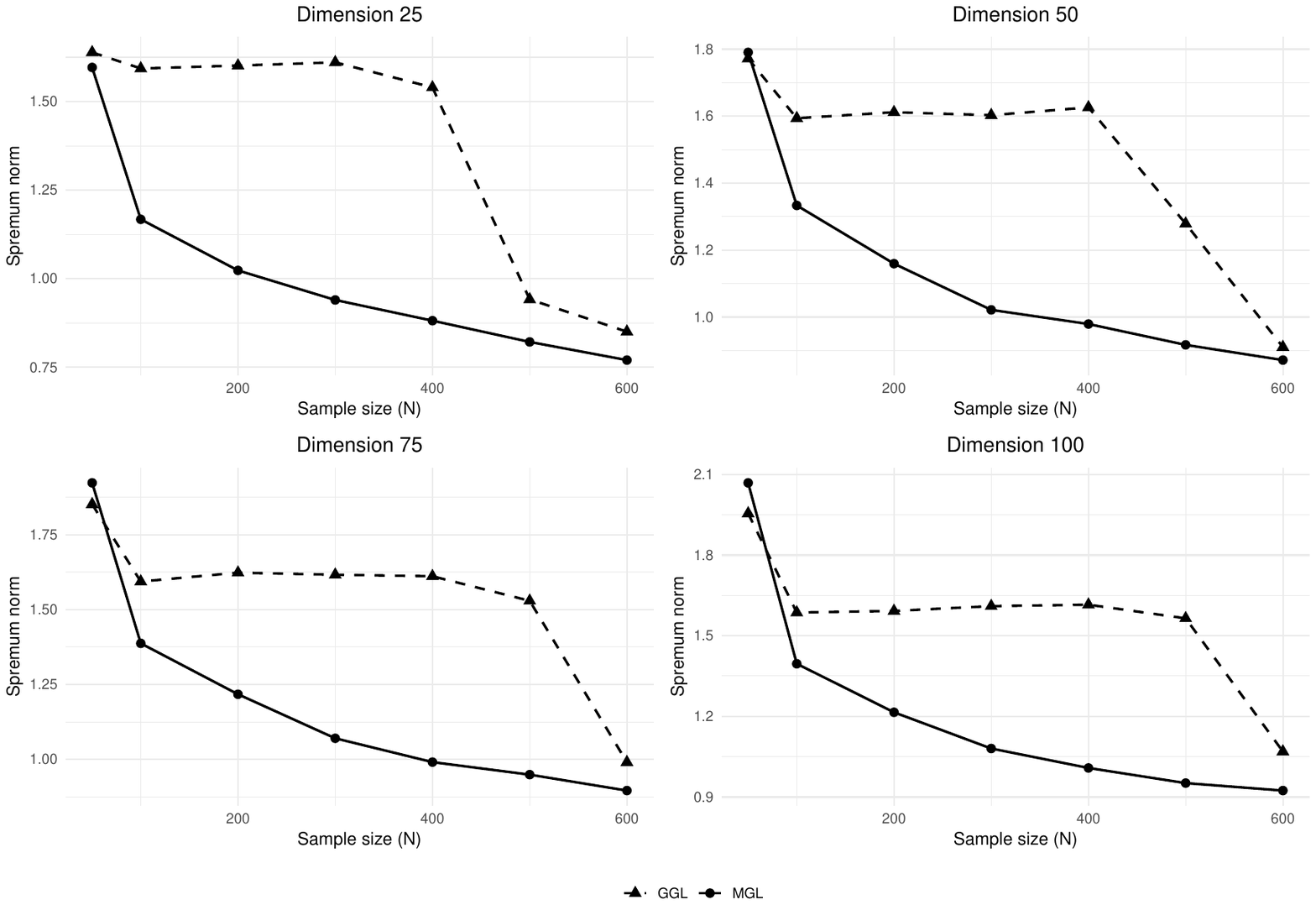}}
		\label{Figure 10}
	\end{figure}
	\subsection{A Heterogeneous Sparsity Example}\label{Section 4.2}
	In this section, we analyze the performance of Mglasso in recovering both common and population-specific graph structures on a synthetic dataset, where the precision matrices have dissimilar sparsity pattern. We consider a simulation setup with $p=25$ variables across $L=3$ populations. The true graph for each population is based on a chain graph backbone. The edge set is partitioned into common edges, present in all three populations,
	\[
	E_{common}=\{(1,2),(3,4),(6,7),(10,11),(14,15),(17,18),(19,20),(20,21),(23,24)\},
	\]
	and uncommon edges, absent in at least one population:
	\[
	E_a=E_1\cap E_2\cap E_3^c=\{(8,9),(12,13),(22,23)\}, \quad
	E_b=E_1\cap E_2^c\cap E_3=\{(5,6),(11,12),(15,16)\},
	\]
	\[
	E_c=E_1^c\cap E_2\cap E_3^c=\{(13,14)\}, \quad
	E_d=E_1^c\cap E_2^c\cap E_3=\{(9,10)\}.
	\]
	Then the uncommon edge set $E_{uncommon}= E_a \cup E_b \cup E_c \cup E_d$. Define the population-specific edge sets as unions of the pieces above: 
	$$E_{1}=E_{common}\cup E_a\cup E_b, \qquad E_2 = E_{common}\cup E_a\cup E_c,\qquad E_3 = E_{common}\cup E_b\cup E_d.$$
	The binary adjacency matrices $\boldsymbol{U}_{1},\boldsymbol{U}_{2},\boldsymbol{U}_{3}\in\{0,1\}^{p\times p}$
	encode the edge structure of each population as
	$$U_{k,ij}=\mathbb{I}\{(i,j)\in E_k\},\qquad k=1,2,3,$$
	where each edge $(i,j)\in E_{k}$ is an ordered pair with $i<j$, so $U_k$ is upper-triangular. The precision matrices are then defined as: $\boldsymbol{\Omega}_{1}=\boldsymbol{I}_{p}+0.2*(\boldsymbol{U}_{1}+\boldsymbol{U}_{1}^{T})$,
	$\boldsymbol{\Omega}_{2}=\boldsymbol{I}_{p}+0.35*(\boldsymbol{U}_{2}+\boldsymbol{U}_{2}^{T})$,
	and $\boldsymbol{\Omega}_{3}=\boldsymbol{I}_{p}+0.25*(\boldsymbol{U}_{3}+\boldsymbol{U}_{3}^{T})$. The resulting network structures are visualized in Figure \ref{Figure 8.1}.
	\begin{figure}
		
		\centering{
			\includegraphics[scale=0.47]{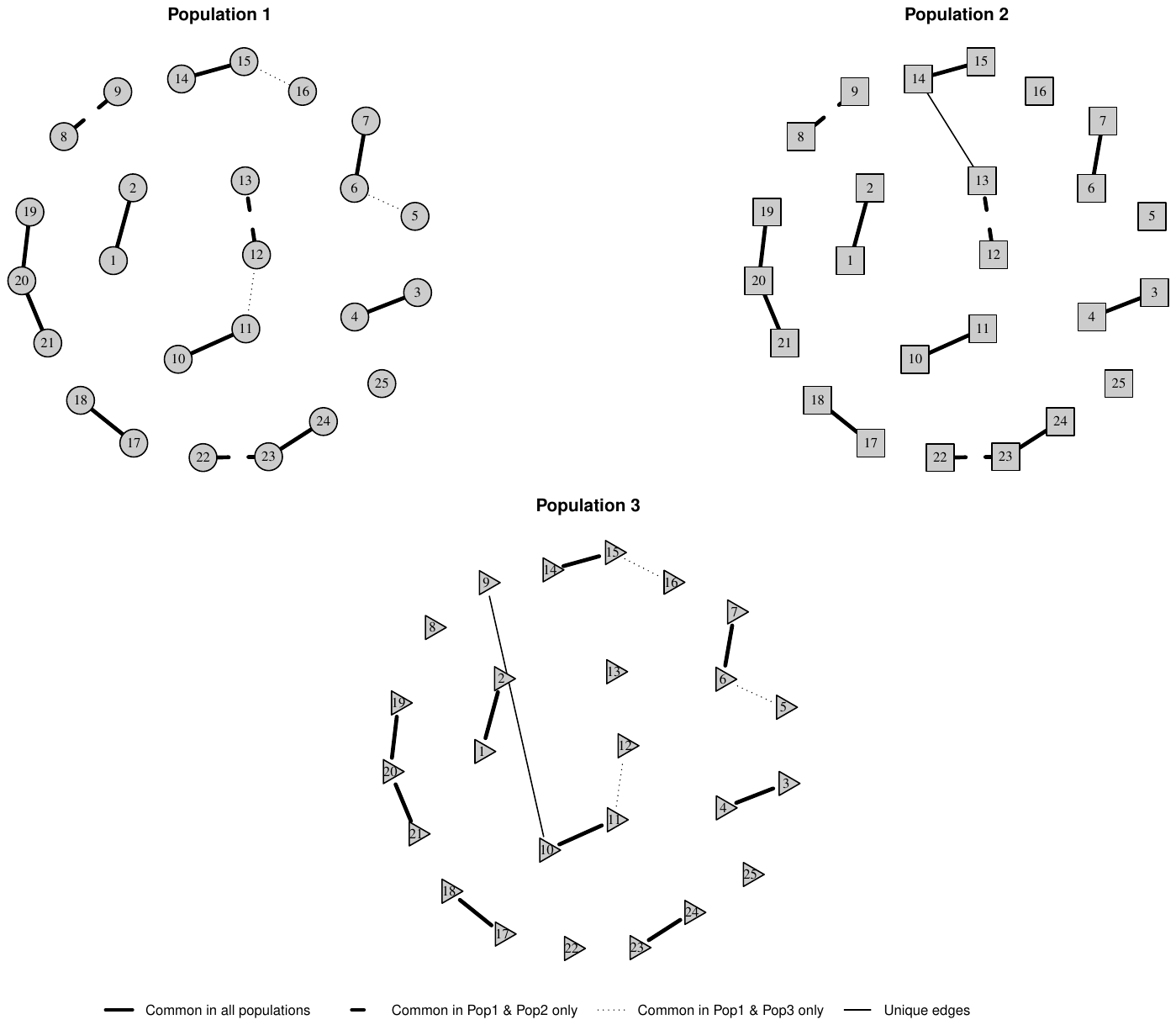}		
		}
		\caption{\label{Figure 8.1}Graphical network structures for different populations. (a) Population
			1 with precision matrix $\boldsymbol{\Omega}_{1}=2\boldsymbol{I}_{p}+0.5\boldsymbol{Z}+\boldsymbol{U_{1}}$;
			(b) Population 2 with precision matrix $\boldsymbol{\Omega}_{2}=\boldsymbol{I}_{p}+0.35*(\boldsymbol{U}_{2}+\boldsymbol{U}_{2}^{T})$;
			(c) Population 3 with precision matrix $\boldsymbol{\Omega}_{3}=\boldsymbol{I}_{p}+0.25*(\boldsymbol{U}_{3}+\boldsymbol{U}_{3}^{T})$}
		
	\end{figure}
	For each population, we draw $n$ samples ranging from $100$ to $600$, and repeat the process for $B=100$ times. Precision matrices are estimated via Mglasso with penalty parameters chosen as before using EBIC. 
	
	We define the union edge set $E = E_{common} \cup E_{uncommon}$ as the collection of coordinates with a non-zero entry in at least one population, and let $\hat{E}_{i}$, $i=1,2,3$, denote the estimated edge set for population $i$. Restricting attention to the estimated edges that fall within the true edge set, we find that these subsets agree closely across populations, with $\hat{E}_{i} = \hat{E}$ holding in most of the $B=100$ replications.
	
	Since the edges in $E_{common}$ have non-zero entries in every population's precision matrix, a reliable method should detect them consistently. Indeed, these common edges are recovered in the majority of replications across the full range of sample sizes considered ($n \in [100,600]$), and detection frequency tends to improve as $n$ grows, consistent with the estimator's expected asymptotic behavior.
	
	For the uncommon edges, the results show a clear two-tier pattern. Edges shared between two populations---such as $(8,9)$ and $(5,6)$---are detected at moderate rates, around 30--50\%, while edges specific to only one population---$(13,14)$ and $(9,10)$---are detected far less consistently (between 4 and 23 times out of 100). This gap reflects the weaker statistical signal underlying population-specific structure. The uncommon edges essentially act as model bias terms, and hence their estimates do not improve with sample sizes, as expected.
	
	One thing that needs to be noted is that the estimated precision matrices have similar sparsity pattern that is shared across all the populations even in the presence of heterogeneous edge strengths, a key property of the Mglasso estimator. Furthermore, among the 283 coordinate pairs in $E^{c}$, on average only 5--6 were falsely flagged as edges, supporting the method's low false-positive rate.
	
	Finally, Figure~\ref{Figure 8.2} shows the supremum-norm distance between the estimated and true matrices, computed over the diagonal entries and the coordinates in $E_{common}$. This distance decreases as sample size increases across all populations. Full tables corresponding to this analysis are provided in the supplementary material.
	
	\begin{figure}
		
		\centering{
			\includegraphics[scale=0.47]{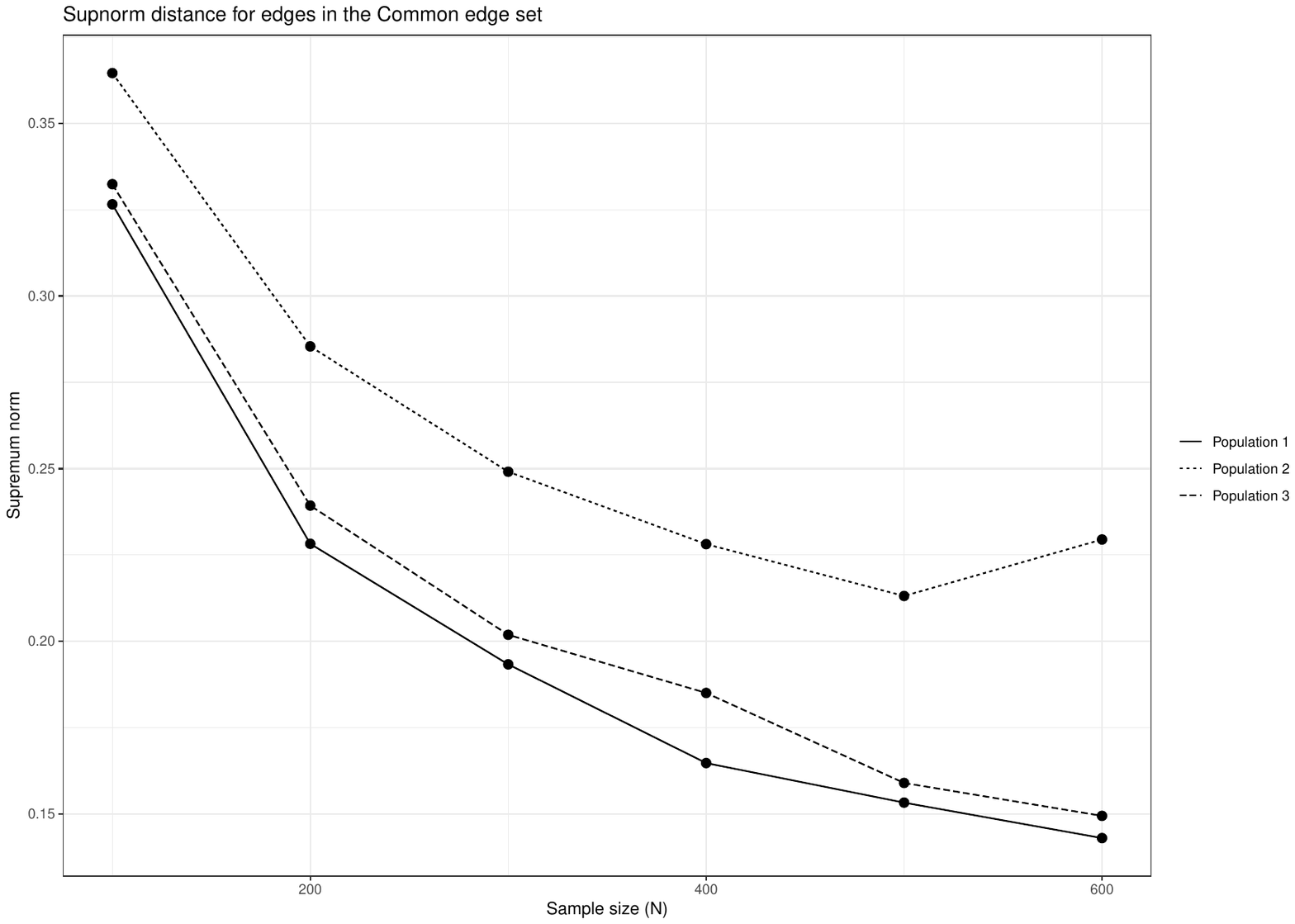}		
		}
		\caption{\label{Figure 8.2}Supnorm distance of the estimated precision matrices for each population
			for coordinates in $E_{common}$}
		
	\end{figure}

	\subsection{Application on Real-world datasets}\label{Section 4.3}
	\subsubsection{Dataset GSE25066}
	We consider the publicly available microarray gene expression dataset GSE25066 (\cite{hatzis2011genomic}), accessed from the NCBI GEO database. The data was generated using the Affymetrix Human Genome U133A microarray platform, which measures the expression	levels of over 22,000 probe sets corresponding to approximately 12,000 unique genes. The dataset comprises 508 breast cancer patients who received taxane-anthracycline based neoadjuvant chemotherapy prior to surgery. For each patient, Estrogen Receptor (ER) status is available and categorized as ER-positive or ER-negative determined by immunohistochemistry. This clinical classification provides two naturally and biologically motivated partition of the patient cohort into $L=2$ populations.
	
	The Mglasso requires that the precision matrices $\Omega_{l}=\Sigma_{l}^{-1},l=1,2$ share a	common sparsity structure encoded by the matrix $\Theta$ in the decomposition $\Omega_{l}=\Theta\odot\Gamma_{l}$, meaning the conditional dependency pattern is preserved across populations while interaction magnitudes may vary. To ensure this assumption is biologically defensible rather than merely imposed, we restrict variables to those in the KEGG Breast Cancer signalling pathway (hsa05224). This curated molecular pathway encodes the experimentally established molecular interactions that collectively govern the three hallmark processes driving breast cancer progression, namely uncontrolled cell proliferation mediated through the PI3K/AKT and RAS/MAPK signalling cascades involving genes such as ERBB2, EGFR, PIK3CA, and KRAS; evasion of apoptosis regulated through the TP53-BCL2-BAX axis; and tumour invasion and metastasis facilitated by CDH1, matrix metalloproteinases, and VEGF signalling. It also encompasses ER and Wnt/Notch signalling modules, making it well-suited for ER-stratified analysis. The Affymetrix probe identifiers are mapped to Entrez Gene IDs via hgu133a.db, intersected with the hsa05224 gene list retrieved through KEGGREST, and multi-probe genes are averaged using limma's avereps function, yielding a single representative measurement per gene. This results in a final expression matrix whose columns correspond to a set of $p=50$ genes, all of which participate in a single, well-characterised molecular network whose topology is determined by human biology rather than by the statistical properties of the observed data.
	
	The biological rationale for the shared sparsity assumption is as follows. The physical and biochemical mechanisms by which genes within the KEGG Breast Cancer pathway interact --- which transcription factors bind to which promoter regions, which kinases phosphorylate which substrates, and which proteins form regulatory complexes --- are determined by the molecular architecture of human cells and are invariant to the clinical subgroup to which a patient belongs. Whether a patient is ER-positive or ER-negative, does not alter the fundamental wiring of the pathway. What differs across subgroups is the intensity of specific interactions, reflecting the presence or absence of estrogen receptor signalling. This distinction maps precisely onto the structure of the model: $\Theta$ encodes the common network topology, while the population-specific $\Gamma_{l}$ capture the differential amplification or suppression of individual pathway edges in each subgroup. Consequently, the pathway-filtered dataset provides a principled and biologically interpretable setting in which to demonstrate the methodology, where the estimated $\hat{\Gamma}_{l}$ directly quantify which pathway interactions are upregulated or downregulated in each clinical subtype relative to the population average.
	
	We have $N_{1}=297$ and $N_{2}=205$ observations from Population 1 (ER-positive) and Population 2 (ER-negative), respectively. After normalizing the data, we estimate the precision matrices using the Mglasso with the penalty parameters $\lambda\propto L\sqrt{\log(p)/n}$ and $\mu\propto \sqrt{\log(p)/n}$. The predicted network is shown in Figure \ref{Figure 13}(left). The details of selected genes and a network summary are provided in the supplementary material.
	
	Because the true precision matrices are unknown, we use the whole-network estimates as proxies. To evaluate performance across varying sample sizes ($n=50,75,100,125,150,175,190$), we randomly select sub samples of $n$ observations from each population without replacement for $B=100$ replications. For each replication, we normalize the samples, estimate the precision matrices and measure the proportion of common edges with correctly detected signs and the average supremum norm distance from the proxy matrices. We then report the average of $B=100$ replications and report the result in  Table \ref{Table 1}. As expected, increasing the sample size leads to a higher proportion of correctly detected edge signs and a lower average supremum norm distance. A sample of the recovered common network with $N=190$ is shown in Figure \ref{Figure 13}(right).
	
	\begin{figure}
		
		\centering{
			\includegraphics[scale=0.3]{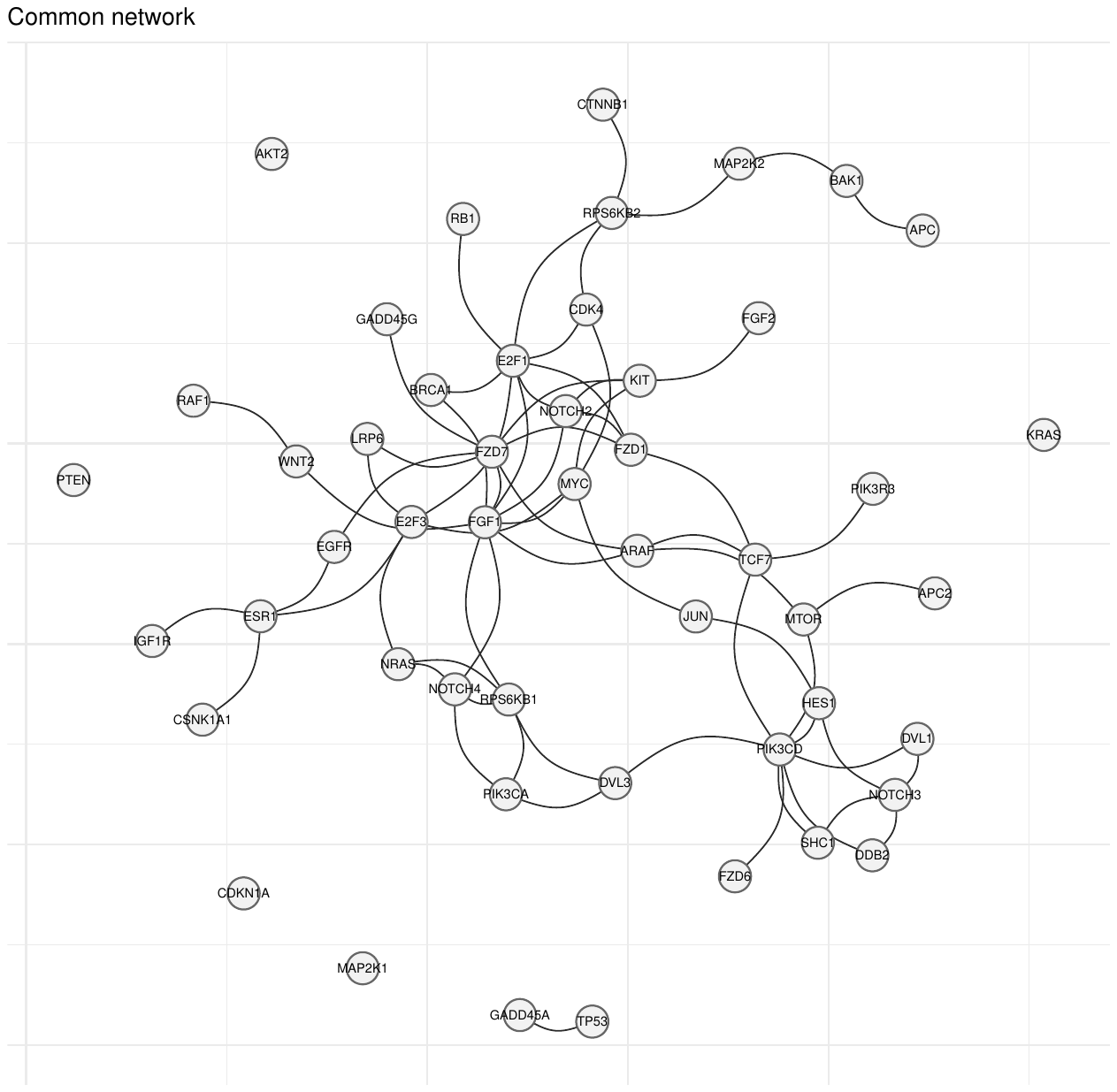}\includegraphics[scale=0.3]{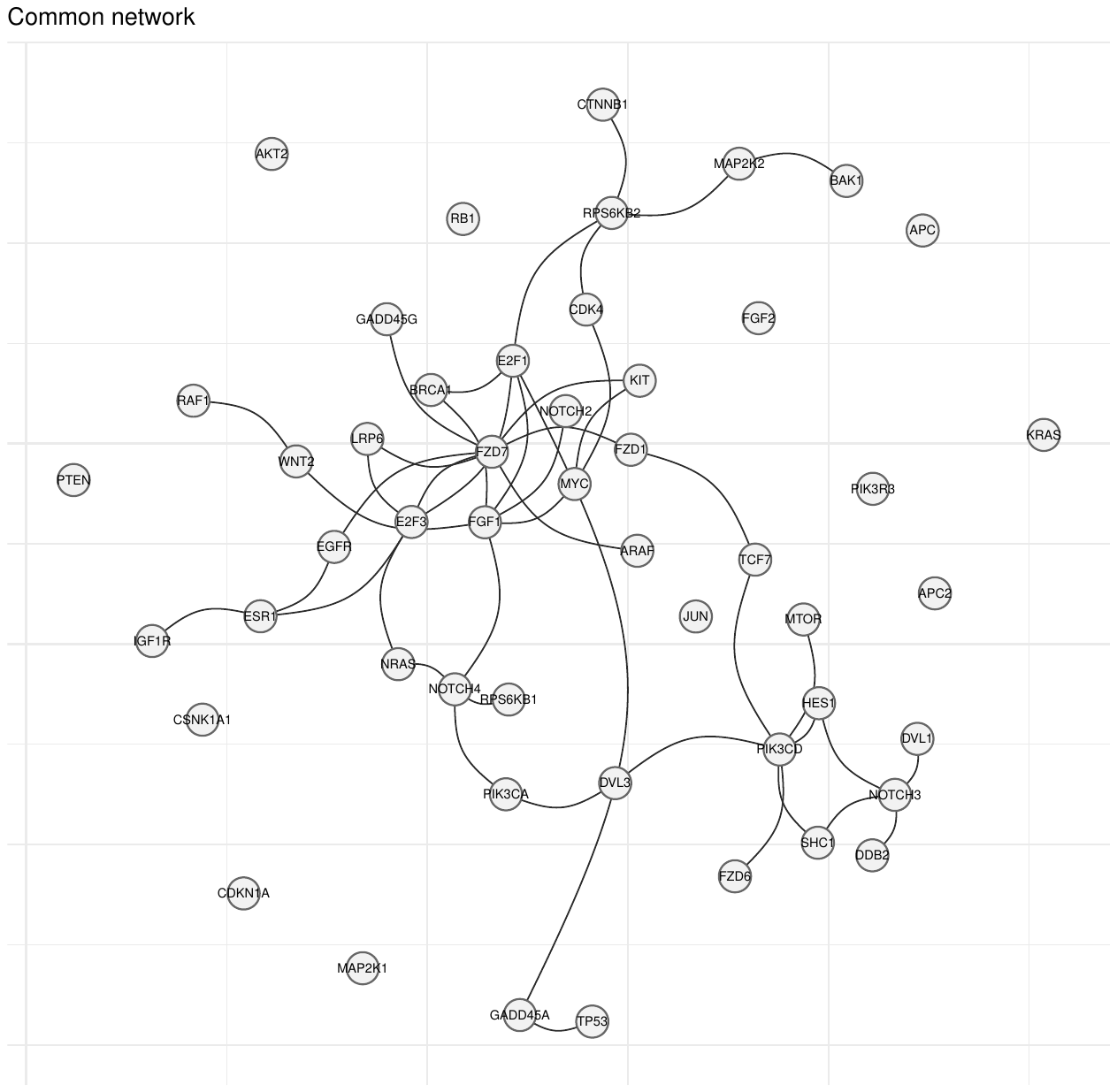}		
		}
		\caption{\label{Figure 13}Estimated common gene interaction network (left) recovered by the Mglasso across ER-positive and ER-negative breast cancer patient
			cohorts. Nodes represent the $p=50$ KEGG Breast Cancer pathway genes;
			edges encode shared conditional dependencies present in both subtypes. Recovered common network structure (Right) estimated from a random sub sample
			of $n=190$ observations per population.}
		
	\end{figure}
	\begin{table}[htbp]
		\centering
		\caption{Sign recovery rate and average supremum-norm error (separately for ER-positive and ER-negative precision matrices) for the Mglasso estimator as a function of subsample size $n$, based on $B=100$ random subsamples. All metrics are evaluated relative to full-data Mglasso estimates as proxy ground truth.}
		\label{Table 1}
		
		\begin{tabular}{lccccccc}
			\toprule
			Subsample size $(n)$
			& 50 & 75 & 100 & 125 & 150 & 175 & 190 \\
			\midrule
			
			Correctly signed estimated edges (\%)
			& 13.1 & 22.3 & 33.1 & 47.05 & 60.1 & 72.4 & 79.1 \\
			
			Average supnorm distance (ER-positive)
			& 0.345 & 0.283 & 0.239 & 0.199 & 0.168 & 0.137 & 0.117 \\
			
			Average supnorm distance (ER-negative)
			& 0.330 & 0.273 & 0.223 & 0.188 & 0.144 & 0.110 & 0.090 \\
			
			\bottomrule
		\end{tabular}
	\end{table}
	\subsubsection{WebKb dataset}
	For our empirical analysis, we utilize a well-known corpus of academic webpages collected from the computer science departments of four institutions: Cornell University, the University of Washington, the University of	Wisconsin, and the University of Texas. Initial text normalization	involved filtering out terms shorter than three characters, removing standard stopwords, and applying \href{https://tartarus.org/martin/PorterStemmer/}{the Porter algorithms} to reduce the remaining words to their base stems (further preprocessing details are given in \cite{2007:phd-Ana-Cardoso-Cachopo}).
	
	While the original corpus categorizes webpages into seven distinct groups-'other' (3764 pages), 'student' (1641), 'faculty' (1124), 'course' (930), 'project' (504), 'department' (182), and 'staff' (137)- our study isolates the two most heavily populated academic roles: students and faculty. We extracted a working subset totaling 1396 observations, specifically including $N_{1}=544$ student pages and $N_{2}=374$ faculty pages.
	
	To construct our quantitative feature space, we first identify the raw frequency of the $j^{th}$ term within the $i^{th}$ webpage, $f_{ij}$,$i=1,2,\ldots,n;j=1,2,\ldots,p$. The final data matrix is defined as $X=(x_{ij})$, applying a log-entropy transformation	such that $x_{ij}=e_{j}\log(1+f_{ij})$. In this context, $e_{j}$ represents the log-entropy weight of the $j^{th}$ term, computed as $e_{j}=1+\sum_{i=1}^{n}p_{ij}\log(p_{ij})/\log(n)$, where the term probability is $p_{ij}=f_{ij}/\sum_{i=1}^{n}f_{ij}$ (\cite{dumais1991improving}). Building upon the framework established in \cite{guo2011joint}, which modeled network structures using the 100 features with the highest log-entropy weights, we further select a subset of 50 terms for our	analysis. 
	
	Figure \ref{Figure 14} illustrates the estimated shared network structure after estimation of the precision matrices from the entire $N_{1}=544$ and $N_{2}=374$ observations. Under the absence of the true precision matrices, we used the predicted matrices from the entire dataset as a proxy to the true parameters. We focussed our investigation on evaluating	the fraction of shared edges with correctly identified signs across both populations and the average supremum norm distance of the estimated matrices from the proxy ones. We drew samples of sizes $n_{i}=\{N_{i}/2,2N_{i}/3,3N_{i}/4$ $,4N_{i}/5,5N_{i}/6,9N_{i}/10\},i=1,2$ without replacement, repeating this procedure for $B=100$ times. The result is reported in Table \ref{Table 2}.
	\begin{figure}
		
		\centering{
			\includegraphics[scale=0.47]{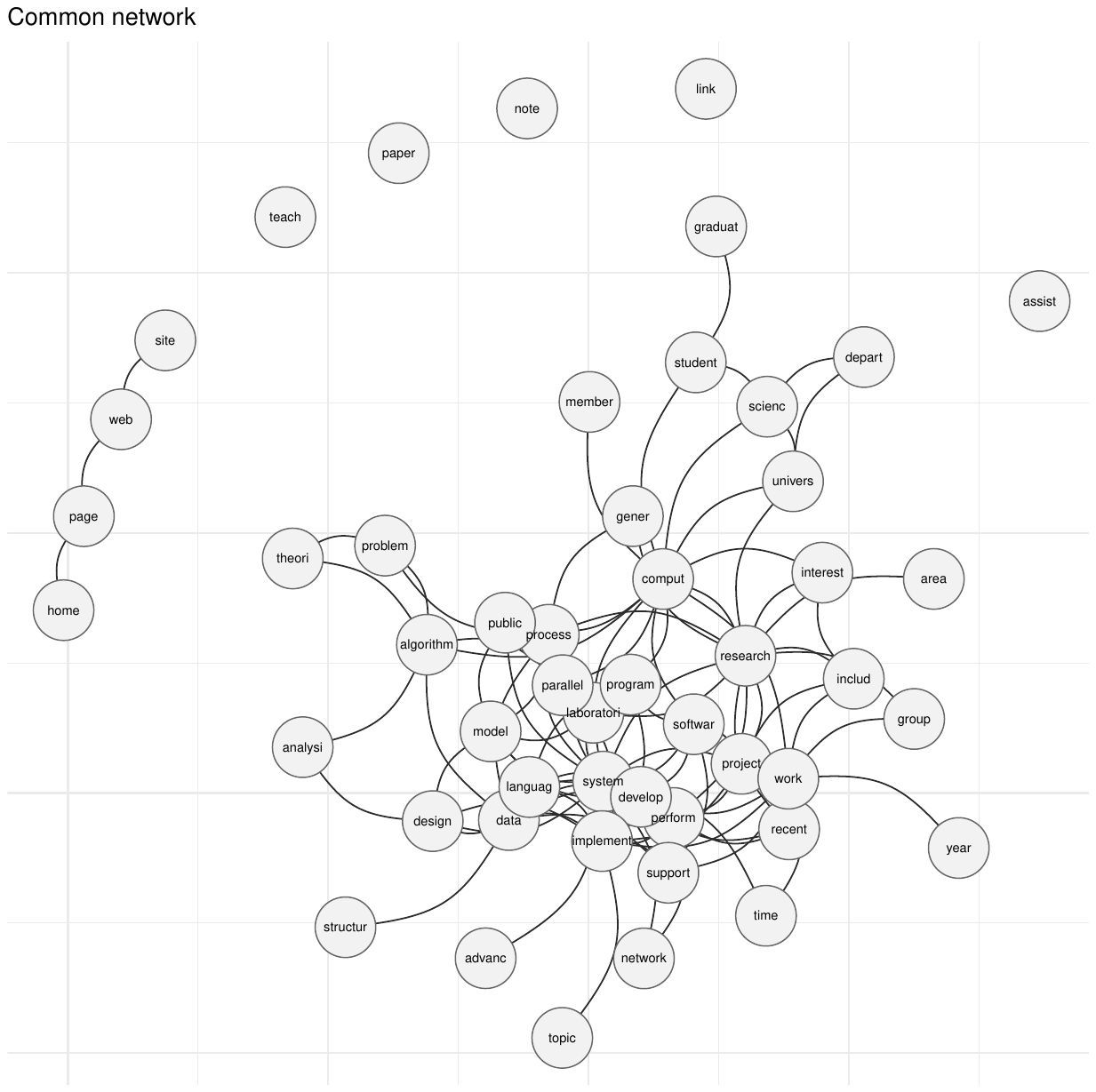}		
		}
		\caption{\label{Figure 14}Estimated common interaction network recovered by the Mglasso across the two populations, student and faculty. Nodes represent
			the $p=50$ terms}
		
	\end{figure}
	\begin{table}[htbp]
		\centering
		\caption{Sign recovery rate and average supremum-norm estimation error for
			the Mglasso estimator as a function of subsample
			size $n$, evaluated over $B=100$ random subsamples. Results are
			reported separately for Student and Faculty precision matrix estimates,
			with the full-data estimates used as the proxy ground truth.}
		\label{Table 2}
		
		\begin{tabular}{lccccccc}
			\toprule
			Subsample size $(n)$
			& $N_{i}/2$ & $2N_{i}/3$ & $3N_{i}/4$ & $4N_{i}/5$ & $5N_{i}/6$ & $9N_{i}/10$ \\
			\midrule
			
			Correctly signed estimated edges (\%)
			& 52 & 67.4 & 75.5 & 79.4 & 82.1 & 88.8 \\
			
			Average supnorm distance (Student) & 0.68 & 0.58 & 0.55 & 0.53 & 0.52 & 0.48 \\
			
			Average supnorm distance (Faculty) & 0.4 & 0.38 & 0.36 & 0.35 & 0.33 & 0.31 \\
			
			\bottomrule
		\end{tabular}
	\end{table}
	\section{Conclusion}\label{Section 5}
	In this paper, we introduced the Multiplicative graphical lasso (Mglasso), a novel regularization-based approach designed to jointly estimate multiple high-dimensional precision matrices that share a common sparsity structure but exhibit population- specific edge strengths. By decomposing each precision matrix into the Schur-Haramard product of a shared structural matrix $\boldsymbol{\Theta}$, and a population-specific matrix $\boldsymbol{\Gamma}_{l}$, our method effectively isolates common conditional independence graphs while capturing distinct interaction variation across different populations.
	
	The optimization problem is solved via an ADMM algorithm integrated with gradient descent, for which we provided a detailed derivation and demonstrate empirical convergence. Tuning parameter selection is carried out via the Extended Bayesian Information Criterion, and the algorithm is initialized from individual graphical lasso estimates.
	
	Theoretically, we established rigorous high-dimensional consistency guarantees for the proposed estimator. By demonstrating the local strict convexity of the objective function, we derived supremum norm error bounds and exact support recovery under sub-Gaussian tail conditions, confirming the method's statistical reliability in identifying true network topologies.
	
	Simulation studies across chain and star graph structures confirm the theoretical predictions and demonstrate that the Mglasso achieves model consistency at substantially smaller sample sizes than the GGL across all dimensions considered. In terms of estimation accuracy, Mglasso maintains competitive supremum norm error to GGL on chain graphs and outperforms it markedly on star graphs. The practical utility of the method was further corroborated through real-world applications on the GSE25066 breast cancer and WebKb datasets. In the genomic analysis, the method successfully recovered biologically interpretable shared conditional dependencies across ER-positive and ER-negative subtypes, quantifying population-specific regulatory intensifications without forcibly altering the underlying, biologically grounded network topology.
	
	A noted limitation of the current algorithm is its reliance on the graphical lasso estimates; if these initialization estimates are sparser than the true network, the coordinate-wise updates may fail to recover the missing edges. Future research could explore alternative initialization strategies to mitigate this issue, as well as extend the framework to accommodate more flexible, non-Gaussian distributions and directed acyclic graphs.
	
	%
	%
	%
	%
	\phantomsection\label{supplementary-material}
	\bigskip
	
	\begin{center}
		
		{\large\bf SUPPLEMENTARY MATERIAL}
		
	\end{center}
	
	\subsection{Detailed derivation of Mglasso algorithm}\label{Appendix A}
	We discuss the detailed steps of the derivation of the Mglasso algorithm. The gradient of $\mathcal{L}_{\lambda,\mu,\delta}(\boldsymbol{\Theta},\boldsymbol{\mathcal{G}},\boldsymbol{R},\boldsymbol{B})$ with respect to $\boldsymbol{\Theta}$ is
	\begin{eqnarray}\label{(8)}
		\notag
		\nabla_{\boldsymbol{\Theta}}\mathcal{L}_{\lambda,\mu,\delta}(\boldsymbol{\Theta},\boldsymbol{\mathcal{G}},\boldsymbol{R},\boldsymbol{B}) &=& -\sum_{l=1}^{L}\boldsymbol{\Gamma}_{l}\odot(\boldsymbol{\Theta}\odot\boldsymbol{\Gamma}_{l})^{-1}+\sum_{l=1}^{L}\boldsymbol{S}_{l}\odot\boldsymbol{\Gamma}_{l}+\boldsymbol{B}+\frac{1}{\delta}(\boldsymbol{\Theta}-\boldsymbol{R})\\
		& & -2\mu\sum_{l=1}^{L}(\boldsymbol{\Gamma}_{l}-\overline{\boldsymbol{\Gamma}}(\boldsymbol{\Theta}))\odot \nabla_{\boldsymbol{\Theta}}\overline{\boldsymbol{\Gamma}}(\boldsymbol{\Theta}).
	\end{eqnarray}
	Considering $\overline{\boldsymbol{\Gamma}}(\boldsymbol{\Theta})$ as a function of $\boldsymbol{\Theta}$, we can approximate $(\overline{\boldsymbol{\Gamma}}(\boldsymbol{\Theta}))_{ij}=\mathbb{I}(|\Theta_{ij}|>0)\approx 1-exp(-c\Theta_{ij}^{2})$, for some $c>0$. Then
	$$
	\frac{d}{d\Theta_{ij}}(\overline{\boldsymbol{\Gamma}}(\boldsymbol{\Theta}))_{ij}\approx 2\ c\ \Theta_{ij}\ exp(-c\Theta_{ij}^{2}).
	$$
	Thus, we approximate $\nabla_{\boldsymbol{\Theta}}\overline{\boldsymbol{\Gamma}}(\boldsymbol{\Theta})\approx 2c\ \boldsymbol{\Theta}\odot f(\boldsymbol{\Theta})$, where $(f(\boldsymbol{\Theta}))_{ij}=exp(-c\Theta_{ij}^{2})$.
	\begin{Remark}
		Since we are approximating a discontinuous function $\mathbb{I}(|\Theta_{ij}|>0)$ with a continuous function $1-exp(-c\Theta_{ij}^{2})$, we do not expect $\overline{\boldsymbol{\Gamma}}(\boldsymbol{\Theta})$ to be too much volatile within each gradient descent step and fix $\overline{\boldsymbol{\Gamma}}(\boldsymbol{\Theta})$ to be at $\boldsymbol{\Theta}=\boldsymbol{\Theta}^{k-1}$, i.e., the value of $\boldsymbol{\Theta}$ at the previous iterated ADMM step.
	\end{Remark}
	This gradient equation shall be used in the Gradient descent method for obtaining $\boldsymbol{\Theta}$. For that, we update the value of $\boldsymbol{\Theta}$ as follows:
	\begin{eqnarray}\label{(9)}
		\notag
		\boldsymbol{\Theta}_{m}&=&\boldsymbol{\Theta}_{m-1}-t_m\nabla_{\boldsymbol{\Theta}_{m-1}}\mathcal{L}_{\lambda,\mu,\delta}(\boldsymbol{\Theta},\mathcal{G}^{k-1},\boldsymbol{R}^{k-1},\boldsymbol{B}^{k-1})\\
		\notag
		&=& \boldsymbol{\Theta}_{m-1}-t_m\bigg\{-\sum_{l=1}^{L}\boldsymbol{\Gamma}_{l}^{k-1}\odot(\boldsymbol{\Theta}_{m-1}\odot\boldsymbol{\Gamma}_{l}^{k-1})^{-1}+\sum_{l=1}^{L}\boldsymbol{S}_{l}\odot\boldsymbol{\Gamma}_{l}^{k-1}+\boldsymbol{B}^{k-1}\\
		& &+\frac{1}{\delta}(\boldsymbol{\Theta}_{m-1}-\boldsymbol{R}^{k-1})-2\mu \sum_{l=1}^{L}(\boldsymbol{\Gamma}_{l}^{k-1}-\overline{\boldsymbol{\Gamma}}^{k-1})\odot (2c\ \boldsymbol{\Theta}_{m-1}\odot f(\boldsymbol{\Theta}_{m-1}))\bigg\}. 
	\end{eqnarray}
	We stop at some $m$ and hence, $\boldsymbol{\Theta}^k=\boldsymbol{\Theta}_m$. We obtain $\boldsymbol{\overline{\Gamma}}^k$ from $\boldsymbol{\Theta}^k$, since $$((\boldsymbol{\overline{\Gamma}}^k))_{ij}=\begin{cases}
		1, & \text{if }\boldsymbol{\Theta}^{k}_{ij} \neq 0 \\
		0, & \text{if }\boldsymbol{\Theta}^{k}_{ij}= 0.
	\end{cases}$$
	The constraint $\frac{1}{L}\sum_{l=1}^{L}\boldsymbol{\Gamma}_l=\overline{\boldsymbol{\Gamma}}(\boldsymbol{\Theta})$ restricts us to solve the problem in parallel for $l=1,2,\ldots,L$. We define
	\begin{eqnarray}\label{(10)}
		\notag
		G_{l} &=& \nabla_{\boldsymbol{\Gamma}_{l}}\mathcal{L}_{\lambda,\mu,\delta}(\boldsymbol{\Theta},\boldsymbol{\mathcal{G}},\boldsymbol{R},\boldsymbol{B})\\
		&=& \boldsymbol{S}_{l}\odot\boldsymbol{\Theta}-\boldsymbol{\Theta}\odot (\boldsymbol{\Theta}\odot\boldsymbol{\Gamma}_{l})^{-1}+2\mu(\boldsymbol{\Gamma}_{l}-\overline{\boldsymbol{\Gamma}}),
	\end{eqnarray}
	and update
	\begin{equation}\label{(11)}
		\Tilde{\boldsymbol{\Gamma}}_{lm}=\Tilde{\boldsymbol{\Gamma}}_{l(m-1)}-s_mG_{l(m-1)},
	\end{equation}
	where $G_{l(m-1)}=\boldsymbol{S}_{l}\odot\boldsymbol{\Theta}^{k}-\boldsymbol{\Theta}^{k}\odot (\boldsymbol{\Theta}^{k}\odot\Tilde{\boldsymbol{\Gamma}}_{l(m-1)})^{-1}+2\mu(\Tilde{\boldsymbol{\Gamma}}_{l(m-1)}-\overline{\boldsymbol{\Gamma}}^{k})$. After some gradient descent steps, we set $\Tilde{\boldsymbol{\Gamma}}_{l}^{k}=\Tilde{\boldsymbol{\Gamma}}_{lm}$. However, $\Tilde{\boldsymbol{\Gamma}}_{l}^{k}$ does not satisfy the constraint $\frac{1}{L}\sum_{l=1}^{L}\Tilde{\boldsymbol{\Gamma}}_{l}^{k}=\overline{\boldsymbol{\Gamma}}^{k}$. For that, we define $\boldsymbol{\zeta}=\sum_{l=1}^{L}\Tilde{\boldsymbol{\Gamma}}_{l}^{k}-L\overline{\boldsymbol{\Gamma}}^{k}$, and then set $\boldsymbol{\Gamma}_{l}^{k}=\Tilde{\boldsymbol{\Gamma}}_{l}^{k}-\frac{1}{L}\boldsymbol{\zeta}$. Due to this update, the last term in (\ref{(9)}) does not contribute. Hence, rewriting the gradient descent equation for $\boldsymbol{\Theta}$, the gradient descent equation form (\ref{(9)}) simplifies to
	\begin{eqnarray}\label{(12)}
		\notag
		\boldsymbol{\Theta}_{m}&=&\boldsymbol{\Theta}_{m-1}-t_m\bigg\{-\sum_{l=1}^{L}\boldsymbol{\Gamma}_{l}^{k-1}\odot(\boldsymbol{\Theta}_{m-1}\odot\boldsymbol{\Gamma}_{l}^{k-1})^{-1}+\sum_{l=1}^{L}\boldsymbol{S}_{l}\odot\boldsymbol{\Gamma}_{l}^{k-1}\\
		& &+\boldsymbol{B}^{k-1}+\frac{1}{\delta}(\boldsymbol{\Theta}_{m-1}-\boldsymbol{R}^{k-1})\bigg\}.
	\end{eqnarray}
	For (\ref{(6)}), the gradient equation of $\mathcal{L}_{\lambda,\mu,\delta}(\boldsymbol{\Theta},\boldsymbol{\mathcal{G}},\boldsymbol{R},\boldsymbol{B})$ is
	\begin{eqnarray}\label{(13)}
		\notag
		\nabla_{\boldsymbol{R}}\mathcal{L}_{\lambda,\mu,\delta}(\boldsymbol{\Theta},\boldsymbol{\mathcal{G}},\boldsymbol{R},\boldsymbol{B})=\frac{1}{\delta}(\boldsymbol{R}-\boldsymbol{\Theta})-\boldsymbol{B}+\lambda\nabla_{\boldsymbol{R}}l_{1}(\boldsymbol{R})&=&0\\
		\implies \boldsymbol{R}-\boldsymbol{\Theta}-\delta\boldsymbol{B}+\lambda\nabla_{\boldsymbol{R}}l_{1}(\boldsymbol{R})&=&0
	\end{eqnarray}
	A closed form solution of (\ref{(13)}) is given by 
	\begin{equation}\label{(14)}
		\boldsymbol{R}_{ij}^k=\begin{cases}
			\boldsymbol{\Theta}_{ii}+\delta \boldsymbol{B}_{ii} &, i=j\\
			\boldsymbol{\Theta}_{ij}+\delta\boldsymbol{B}_{ij}-\lambda \delta &, \text{if } \boldsymbol{\Theta}_{ij}+\delta\boldsymbol{B}_{ij}>\lambda \delta, i\neq j \\
			\boldsymbol{\Theta}_{ij}+\delta\boldsymbol{B}_{ij}+\lambda \delta &, \text{if } \boldsymbol{\Theta}_{ij}+\delta\boldsymbol{B}_{ij}<-\lambda \delta, i\neq j \\
			0 &, \text{if } |\boldsymbol{\Theta}_{ij}+\delta \boldsymbol{B}_{ij}|\leq \lambda \delta, i\neq j
		\end{cases}.
	\end{equation}
	Hence, we have a closed-form solution of (\ref{(6)}) as
	\begin{equation}\label{(15)}
		\boldsymbol{R}_{ij}^k=\begin{cases}
			\boldsymbol{\Theta}^{k}_{ii}+\delta \boldsymbol{B}^{k-1}_{ii} &, i=j\\
			\boldsymbol{\Theta}^{k}_{ij}+\delta\boldsymbol{B}^{k-1}_{ij}-\lambda \delta &,\text{if }\boldsymbol{\Theta}^{k}_{ij}+\delta\boldsymbol{B}^{k-1}_{ij}>\lambda \delta, i\neq j \\
			\boldsymbol{\Theta}^{k}_{ij}+\delta\boldsymbol{B}^{k-1}_{ij}+\lambda \delta &, \text{if }\boldsymbol{\Theta}^{k}_{ij}+\delta\boldsymbol{B}^{k-1}_{ij}<-\lambda \delta, i\neq j \\
			0 &, \text{if } |\boldsymbol{\Theta}^{k}_{ij}+\delta \boldsymbol{B}^{k-1}_{ij}|\leq \lambda \delta, i\neq j
		\end{cases}.
	\end{equation}
	The proposed Mglasso estimator is computed using the following ADMM
	procedure.
	
	\bigskip
	
	\noindent
	\textbf{Algorithm 1. ADMM algorithm for Mglasso}
	
	\begin{enumerate}
		\item Initialize symmetric matrices $\boldsymbol{\Theta};\boldsymbol{\Gamma_l},l=1,2,\ldots,L;\boldsymbol{R};\boldsymbol{B}.$
		
		\item For $k=1,2,\ldots$, repeat until convergence:
		
		\begin{enumerate}
			
			\item Set
			\[
			\boldsymbol{\Theta}_{0}=\boldsymbol{\Theta}^{k-1}.
			\]
			
			Use backtracking search for $t_m$. Update
			\[
			\boldsymbol{\Theta}_{m}
			=
			\boldsymbol{\Theta}_{m-1}-t_m
			\Bigg[
			\sum_{l=1}^{L}
			\Big\{
			\boldsymbol{\Gamma}_{l}^{k-1}\odot \boldsymbol{S}_{l}-\boldsymbol{\Gamma}_{l}^{k-1}\odot ((\boldsymbol{\Theta}_{m-1}\odot \boldsymbol{\Gamma}_{l}^{k-1})^{-1})
			\Big\}
			+\boldsymbol{B}^{k-1}+\frac{1}{\delta}(\boldsymbol{\Theta}_{m-1}-\boldsymbol{R}^{k-1})
			\Bigg],
			\]
			Repeat for $m=1,2,3,\ldots$. Stop at some point. For symmetricity, set
			\[
			\boldsymbol{\Theta}^k=0.5*(\boldsymbol{\Theta}^m+(\boldsymbol{\Theta}^m)^{T}).
			\]
			\item From $\boldsymbol{\Theta}^k$, we get $\boldsymbol{\overline{\Gamma}}^{k}$, such that $\boldsymbol{\overline{\Gamma}}^{k}_{ij}=\mathbb{I}(|\boldsymbol{\Theta}^{k}_{ij}|>0)$.
			\item
			For each $l=1,\ldots,L$, initialize
			\[
			\Tilde{\boldsymbol{\Gamma}}_{l0}=\boldsymbol{\Gamma}_{l}^{k-1}
			\]
			Use backtracking search for $s_m$.Update
			\[
			\Tilde{\boldsymbol{\Gamma}}_{lm}
			=
			\Tilde{\boldsymbol{\Gamma}}_{l(m-1)}-s_m
			\Big[
			\boldsymbol{S}_{l}\odot\boldsymbol{\Theta}^{k}-\boldsymbol{\Theta}^{k}\odot ((\boldsymbol{\Theta}^{k}\odot \Tilde{\boldsymbol{\Gamma}}_{l(m-1)})^{-1})+2\mu (\Tilde{\boldsymbol{\Gamma}}_{l(m-1)}-\overline{\boldsymbol{\Gamma}}^{k})
			\Big],
			\]
			Repeat for $m=1,2,3,\ldots$. Stop at some point. For symmetricity, set
			\[
			\Tilde{\boldsymbol{\Gamma}}_{l}^{k}=0.5*(\Tilde{\boldsymbol{\Gamma}}_{lm}+\Tilde{\boldsymbol{\Gamma}}_{lm}^T)).
			\]
			
			\item Define $\boldsymbol{\zeta}=\sum_{l=1}^{L}\Tilde{\boldsymbol{\Gamma}}^{k}_{l}-L\overline{\boldsymbol{\Gamma}}^{k}$. Update
			\[
			\boldsymbol{\Gamma}_{l}^{k}=\Tilde{\boldsymbol{\Gamma}}^{k}_{l}-\frac{1}{L}\boldsymbol{\zeta},l=1,2,\ldots,L.
			\]
			
			Set $\boldsymbol{\Gamma}_{l,ij}^{k}=0$, whenever $\boldsymbol{\Theta}^{k}_{ij}=0,l=1,2,\ldots,L$.
			
			\item 	Update $\boldsymbol{R}^k$ via soft-thresholding:
			\begin{equation*}
				\boldsymbol{R}_{ij}^k=\begin{cases}
					\boldsymbol{\Theta}^{k}_{ii}+\delta \boldsymbol{B}^{k-1}_{ii} &, i=j\\
					\boldsymbol{\Theta}^{k}_{ij}+\delta\boldsymbol{B}^{k-1}_{ij}-\lambda \delta &, \text{if } \boldsymbol{\Theta}^{k}_{ij}+\delta\boldsymbol{B}^{k-1}_{ij}>\lambda \delta, i\neq j \\
					\boldsymbol{\Theta}^{k}_{ij}+\delta\boldsymbol{B}^{k-1}_{ij}+\lambda \delta &, \text{if }  \boldsymbol{\Theta}^{k}_{ij}+\delta\boldsymbol{B}^{k-1}_{ij}<-\lambda \delta, i\neq j \\
					0 &, \text{if } |\boldsymbol{\Theta}^{k}_{ij}+\delta \boldsymbol{B}^{k-1}_{ij}|\leq \lambda \delta, i\neq j
				\end{cases}
			\end{equation*}
			
			Symmetrize:
			\[
			\boldsymbol{R}^k=0.5*(\boldsymbol{R}^k+(\boldsymbol{R}^k)^{T}).
			\]
			
			\item Update the dual variable
			\[
			\boldsymbol{B}^k = \boldsymbol{B}^{k-1}+\frac{1}{\delta}(\boldsymbol{\Theta}^k-\boldsymbol{R}^k),
			\]
			and symmetrize
			\[
			\boldsymbol{B}^k=0.5*(\boldsymbol{B}^k+(\boldsymbol{B}^k)^{T}).
			\]
			
			\item Terminate the ADMM loop if
			$|\mathcal{L}_{\lambda,\mu,\delta}(\boldsymbol{\Theta}^{k},\boldsymbol{\mathcal{G}}^{k},\boldsymbol{R}^{k},\boldsymbol{B}^{k})-\mathcal{L}_{\lambda,\mu,\delta}(\boldsymbol{\Theta}^{k-1},\boldsymbol{\mathcal{G}}^{k-1},\boldsymbol{R}^{k-1},\boldsymbol{B}^{k-1})|\leq \varepsilon_{1}$, where $\varepsilon_{1}$ is some tolerance limit.
			
		\end{enumerate}
		
	\end{enumerate}
	\subsection{Proof of the theorems and relevant lemmas}
	\subsubsection{Proof of theorems}
	\begin{proof}[Proof of Theorem \ref{Theorem 1}]
		We want a uniform positive lower bound of the quadratic form of the Hessian of $Q$ at any point in $\mathcal{B}_{\infty}(\epsilon)$. Suppose we have an arbitrary perturbation set $\boldsymbol{U}=(\boldsymbol{U}_1,\boldsymbol{U}_2,\ldots,\boldsymbol{U}_{L})$, where $\boldsymbol{U}_l$ is of the same shape as $\boldsymbol{\Omega}_{l}$. The second directional derivative at any point $\{\boldsymbol{\Omega}_{l}\}$ in $\mathcal{B}_{\infty}(\epsilon)$ in the direction $\boldsymbol{U}_{l}$ decomposes as
		\begin{eqnarray}
			\label{(33)}
			\notag
			Q_{1}(\boldsymbol{U})&=&\sum_{l=1}^{L}tr(\boldsymbol{\Omega}_{l}^{-1}\boldsymbol{U}_{l}\boldsymbol{\Omega}_{l}^{-1}\boldsymbol{U}_{l})+\text{contribution from the penalty involving }\mu\\
			&=&T_1+T_2
		\end{eqnarray}
		\begin{enumerate}
			\item Curvature of $T_1$: For each $l=1,2,\ldots,L$,
			\begin{eqnarray*}
				tr(\boldsymbol{\Omega}_{l}^{-1}\boldsymbol{U}_{l}\boldsymbol{\Omega}_{l}^{-1}\boldsymbol{U}_{l})&=& ||\boldsymbol{\Omega}_{l}^{-1}\boldsymbol{U}_{l}||_{F}^{2}\\
				&=& ||vec(\boldsymbol{\Omega}_{l}^{-1}\boldsymbol{U}_{l})||_{2}^{2}\\
				&=&||(I\otimes\boldsymbol{\Omega}_{l}^{-1})vec(\boldsymbol{U}_{l})||_{2}^{2}\\
				&=&||\boldsymbol{Av}||_{2}^{2},
			\end{eqnarray*}
			where $\boldsymbol{A}=I\otimes\boldsymbol{\Omega}_{l}^{-1}$ and $\boldsymbol{v}=vec(\boldsymbol{U}_{l})$. Let the SVD of $\boldsymbol{A}$ be $\boldsymbol{V_1}\boldsymbol{D}\boldsymbol{V_2}^{T}$. Then,
			\begin{eqnarray*}
				||\boldsymbol{Av}||_{2}&=&\sqrt{(\boldsymbol{D}\boldsymbol{V_2}^{T}\boldsymbol{v})^{T}\boldsymbol{V_1}^{T}\boldsymbol{V_1}(\boldsymbol{D}\boldsymbol{V_2}^{T}\boldsymbol{v})}\\
				&=&||\boldsymbol{D}\boldsymbol{V_2}^{T}\boldsymbol{v}||_{2}\\
				&=&\sqrt{\sum_{i=1}^{p^2}D_{i}^{2}w_{i}^{2}},
			\end{eqnarray*}
			where $\boldsymbol{w}=\boldsymbol{V_2}^{T}\boldsymbol{v}$. Then, $||\boldsymbol{Av}||_{2}^{2}\geq \Lambda^{2}_{min}(\boldsymbol{A})||\boldsymbol{w}||_{2}^{2}=\Lambda^{2}_{min}(\boldsymbol{A})||\boldsymbol{v}||_{2}^{2}$. Hence,
			\begin{eqnarray*}
				tr(\boldsymbol{\Omega}_{l}^{-1}\boldsymbol{U}_{l}\boldsymbol{\Omega}_{l}^{-1}\boldsymbol{U}_{l})&\geq& \Lambda^{2}_{min}(I\otimes\boldsymbol{\Omega}_{l}^{-1})||\boldsymbol{U}_{l}||_{F}^{2}\\
				&=&  \Lambda^{2}_{min}(\boldsymbol{\Omega}_{l}^{-1})||\boldsymbol{U}_{l}||_{F}^{2}\\
				&=& \Lambda^{-2}_{max}(\boldsymbol{\Omega}_{l})||\boldsymbol{U}_{l}||_{F}^{2}.
			\end{eqnarray*}
			Now, from Ger\v{s}gorin Circle Theorem (Theorem 6.1.1 from \cite{horn2012matrix}) and (\ref{((32))}), we have
			\begin{equation}
				\label{(34)}
				\underset{i}{max}|\Lambda_{i}(\boldsymbol{\Omega}_{l}-\boldsymbol{\Omega}_{l}^{0})|\leq |||\boldsymbol{\Omega}_{l}-\boldsymbol{\Omega}_{l}^{0}|||_{\infty}\leq d\epsilon.
			\end{equation}
			Also, from Weyl's inequality (Theorem 4.3.1 from \cite{horn2012matrix}), we have
			\begin{equation}
				\label{(35)}
				\Lambda_{max}(\boldsymbol{\Omega}_{l})\leq  |||\boldsymbol{\Omega}_{l}-\boldsymbol{\Omega}_{l}^{0}|||_{\infty}+max(|\Lambda_{min}(\boldsymbol{\Omega}_{l}^{0})|,|\Lambda_{max}(\boldsymbol{\Omega}_{l}^{0})|).
			\end{equation}
			Then from Assumption \ref{Assumption 3}, we have 
			\begin{eqnarray}
				\notag
				\Lambda_{max}(\boldsymbol{\Omega}_{l}) &\leq& d\epsilon+\frac{1}{L_1}\\
				\label{(36)}
				\implies \Lambda^{-2}_{max}(\boldsymbol{\Omega}_{l}) &\geq &\bigg(\frac{1}{L_1}+d\epsilon\bigg)^{-2}\\
				\label{(37)}
				\implies tr(\boldsymbol{\Omega}_{l}^{-1}\boldsymbol{U}_{l}\boldsymbol{\Omega}_{l}^{-1}\boldsymbol{U}_{l})&\geq& \bigg(\frac{1}{L_1}+d\epsilon\bigg)^{-2}||\boldsymbol{U}_{l}||_{F}^{2}.
			\end{eqnarray}
			Hence, we have
			\begin{equation}
				\sum_{l=1}^{L} tr(\boldsymbol{\Omega}_{l}^{-1}\boldsymbol{U}_{l}\boldsymbol{\Omega}_{l}^{-1}\boldsymbol{U}_{l})\geq \bigg(\frac{1}{L_1}+d\epsilon\bigg)^{-2} \sum_{l=1}^{L}||\boldsymbol{U}_{l}||_{F}^{2}.
				\label{(38)}
			\end{equation}
			\item Contribution from $\mu$ penalty: Fix any $(i,j)\in E$. Let $\boldsymbol{m}_{ij}=(m_{1,ij},m_{2,ij},\ldots,m_{L,ij})$, where $m_{l,ij}=\Omega_{l,ij}$ and $\overline{m}_{ij}=\frac{1}{L}\sum_{l=1}^{L}m_{l,ij}=\overline{\Omega}_{ij}$. The per-entry penalty is
			\begin{equation}
				p_{ij}(\boldsymbol{m}_{ij})=\mu L^{2} \frac{\boldsymbol{m}^{T}_{ij}\boldsymbol{m}_{ij}}{(\boldsymbol{1}^{T}\boldsymbol{m}_{ij})^2}-\mu L.
				\label{(39)}
			\end{equation}
			This suffices to differentiate the scalar function $g(\boldsymbol{m}_{ij})=\frac{\boldsymbol{m}^{T}_{ij}\boldsymbol{m}_{ij}}{(\boldsymbol{1}^{T}\boldsymbol{m}_{ij})^2}$. With a small perturbation direction $\boldsymbol{u}_{ij}=(U_{1,ij},U_{2,ij},\ldots,U_{L,ij})'\in \mathbb{R}^{L}$, corresponding to the entries of $\boldsymbol{U}$. Consider $\boldsymbol{m}_{ij}(t)=\boldsymbol{m}_{ij}+t\boldsymbol{u}_{ij}$ and $h_{ij}(t)=g(\boldsymbol{m}_{ij}(t))$. Define $a_{ij}(t)=\boldsymbol{m}^{T}_{ij}(t)\boldsymbol{m}_{ij}(t)=a_{ij}+2t\boldsymbol{m}^{T}_{ij}\boldsymbol{u}_{ij}+t^2\boldsymbol{u}^{T}_{ij}\boldsymbol{u}_{ij}$, where $a_{ij}=\boldsymbol{m}^{T}_{ij}\boldsymbol{m}_{ij}$. Note that $a_{ij}(0)=a_{ij},\ a'_{ij}(0)=2\boldsymbol{m}^{T}_{ij}\boldsymbol{u}_{ij},\ a''_{ij}(0)=\boldsymbol{u}^{T}_{ij}\boldsymbol{u}_{ij}$. Also, let $s_{ij}=\boldsymbol{1}^{T}\boldsymbol{m}_{ij}$. Then, $s_{ij}(t)=\boldsymbol{1}^{T}\boldsymbol{m}_{ij}(t)=s_{ij}+tr_{ij}$, with $r_{ij}=\boldsymbol{1}^{T}\boldsymbol{u}_{ij}$. Then $s_{ij}(0)=s,\ s'_{ij}(0)=r_{ij},\ s''_{ij}(0)=0$. Thus,
			\begin{eqnarray}
				\label{(40)}
				\notag
				h''(0) &=& 2\frac{\boldsymbol{u}^{T}_{ij}\boldsymbol{u}_{ij}}{s_{ij}^2}-\frac{8r_{ij}(\boldsymbol{m}^{T}_{ij}\boldsymbol{u}_{ij})}{s_{ij}^3}+\frac{6a_{ij}r_{ij}^2}{s_{ij}^4}\\
				\notag
				\implies \frac{d^2}{dt^2}p_{ij}(\boldsymbol{m}_{ij}+t\boldsymbol{u}_{ij})\bigg|_{t=0}&=& \mu L^2 h''(0)\\
				&=& \mu L^2 \bigg(2\frac{\boldsymbol{u}^{T}_{ij}\boldsymbol{u}_{ij}}{s_{ij}^2}-\frac{8r_{ij}(\boldsymbol{m}^{T}_{ij}\boldsymbol{u}_{ij})}{s_{ij}^3}+\frac{6a_{ij}r_{ij}^2}{s_{ij}^4}\bigg).
			\end{eqnarray}
			We form an upper bound proportional to $||\boldsymbol{u}_{ij}||_{2}^{2}=\sum_{l=1}^{L}u_{l,ij}^{2}$. From the Cauchy-Schwarz inequality, we have
			\begin{itemize}
				\item $|r_{ij}|=|\boldsymbol{1}^{T}\boldsymbol{u}_{ij}|\leq \sqrt{L}||\boldsymbol{u}_{ij}||_{2}$,
				\item $|\boldsymbol{m}^{T}_{ij}\boldsymbol{u}_{ij}|\leq ||\boldsymbol{m}_{ij}||_{2}||\boldsymbol{u}_{ij}||_{2}$,
				\item $a_{ij}=\boldsymbol{m}^{T}_{ij}\boldsymbol{m}_{ij}=||\boldsymbol{m}_{ij}||_{2}^{2}$.
			\end{itemize}
			From these inequalities, we have,
			\begin{equation}
				\label{(41)}
				\bigg| \frac{d^2}{dt^2}p_{ij}(\boldsymbol{m}_{ij}+t\boldsymbol{u}_{ij})\bigg|_{t=0}\bigg|\leq \mu \frac{2\overline{m}_{ij}^{2}+6\frac{1}{L}||\boldsymbol{m}_{ij}||_{2}^{2}+8\frac{1}{\sqrt{L}}||\boldsymbol{m}_{ij}||_{2}|\overline{m}_{ij}|}{(\overline{m}_{ij})^{4}}||\boldsymbol{u}_{ij}||_{2}^{2}.
			\end{equation}
			In $\mathcal{B}_{\infty}(\epsilon)$, we have $|\Omega_{l,ij}-\Omega_{l,ij}^{0}|\leq \epsilon$, for all $l=1,2,\ldots, L$, and for all $(i,j)$.
			\begin{itemize}
				\item Then,
				\begin{eqnarray*}
					|\overline{m}_{ij}| &=& \bigg|\frac{1}{L}\sum_{l=1}^{L}m_{l,ij}\bigg|\\
					&\leq& \frac{1}{L}\sum_{l=1}^{L}|m_{l,ij}|\\
					&\leq& \frac{1}{L}\sum_{l=1}^{L}(|\Omega^{0}_{l,ij}|+\epsilon)\\
					&\leq & \underset{l}{max}\ |\Omega_{l,ij}^{0}|+\epsilon\\
					\implies (\overline{m}_{ij})^2 &\leq& (\underset{l}{max}\ |\Omega_{l,ij}^{0}|+\epsilon)^2.
				\end{eqnarray*}
				\item Also,
				\begin{eqnarray*}
					||\boldsymbol{m}_{ij}||_{2} &=&       \sqrt{\sum_{l=1}^{L}m_{l,ij}^{2}}\\
					&\leq& \sqrt{\sum_{l=1}^{L}(|\Omega_{l,ij}^{0}|+\epsilon)^{2}}\\
					\implies\frac{1}{\sqrt{L}}||\boldsymbol{m}_{ij}||_{2} &\leq & \sqrt{\frac{1}{L}\sum_{l=1}^{L}(|\Omega_{l,ij}^{0}|+\epsilon)^{2}}
				\end{eqnarray*}
				\item Hence,
				\begin{eqnarray*}
					\frac{1}{L}||\boldsymbol{m}_{ij}||_{2}|\overline{m}_{ij}|&\leq& \sqrt{\frac{1}{L}\sum_{l=1}^{L}(|\Omega_{l,ij}^{0}|+\epsilon)^{2}}(\underset{l}{max}\ |\Omega_{l,ij}^{0}|+\epsilon)\\
					&\leq& \sqrt{\underset{l}{max}(|\Omega_{l,ij}^{0}|+\epsilon)^{2}}(\underset{l}{max}\ |\Omega_{l,ij}^{0}|+\epsilon)\\
					&=& (\underset{l}{max}\ |\Omega_{l,ij}^{0}|+\epsilon)^{2}.
				\end{eqnarray*}
				\item Again,
				\begin{eqnarray*}
					\frac{1}{L}||\boldsymbol{m}_{ij}||_{2}^{2} &=& \frac{1}{L}\sum_{l=1}^{L}m_{l,ij}^{2}\\
					&\leq & \underset{l}{max}(|\Omega_{l,ij}^{0}|+\epsilon)^{2}\\
					&=& (\underset{l}{max}\ |\Omega_{l,ij}^{0}|+\epsilon)^{2}.
				\end{eqnarray*}
			\end{itemize}
			Using the above inequalities, we have
			\begin{equation}
				\bigg| \frac{d^2}{dt^2}p_{ij}(\boldsymbol{m}_{ij}+t\boldsymbol{u}_{ij})\bigg|_{t=0}\bigg| \leq \frac{16\mu (\underset{l}{max}\ |\Omega_{l,ij}^{0}|+\epsilon)^{2}}{(\overline{m}_{ij})^{4}}||\boldsymbol{u}_{ij}||_{2}^{2}.
				\label{(42)}
			\end{equation}
			From $\mathcal{B}_{\infty}(\epsilon)$ ball, if $|\overline{\Omega}_{ij}^{0}|>\epsilon$, we have $|\overline{m}_{ij}|>|\overline{\Omega}_{ij}^{0}|-\epsilon$, which implies $(\overline{m}_{ij})^{4}>(|\overline{\Omega}_{ij}^{0}|-\epsilon)^{4}\geq \underset{(i,j)\in E}{min}(|\overline{\Omega}_{ij}^{0}|-\epsilon)^{4}$. Then,
			\begin{equation}
				\label{(43)}
				\sum_{i,j} \bigg| \frac{d^2}{dt^2}p_{ij}(\boldsymbol{m}_{ij}+t\boldsymbol{u}_{ij})\bigg|_{t=0}\bigg| \leq 16\mu \frac{(\underset{l}{max}\ |\Omega_{l,ij}^{0}|+\epsilon)^{2}}{\underset{(i,j)\in E}{min}(|\overline{\Omega}_{ij}^{0}|-\epsilon)^{4}} \sum_{l=1}^{L}||\boldsymbol{U}_{l}||_{F}^{2}.
			\end{equation}
		\end{enumerate}
		Thus, for all non-zero $\boldsymbol{U}$, we have $Q_1(\boldsymbol{U})>0$ if 
		\begin{equation}
			\label{(44)}
			0< \mu <\frac{\underset{(i,j)\in E}{min}(|\overline{\Omega}_{ij}^{0}|-\epsilon)^{4}}{16\ \underset{(i,j)}{max}(\underset{l}{max}\ |\Omega_{l,ij}^{0}|+\epsilon)^{2}}\bigg(\frac{1}{L_1}+d\epsilon\bigg)^{-2}.
		\end{equation}
		Hence, under Assumption \ref{Assumption 3}, if $\underset{(i,j)\in E}{min}|\overline{\Omega}_{ij}^{0}|>\epsilon$, and
		
		\begin{equation}
			\label{(45)}
			0<\mu< \frac{(\underset{(i,j)\in E}{min}|\overline{\Omega}_{ij}^{0}|-\epsilon)^{4}}{16\ \underset{(i,j)}{max}(\underset{l}{max}\ |\Omega_{l,ij}^{0}|+\epsilon)^{2}}\bigg(\frac{1}{L_1}+d\epsilon\bigg)^{-2}=K(\boldsymbol{\Omega}^{0},L_1,\epsilon,d),
		\end{equation}
		$Q(\boldsymbol{\Omega})$ is strictly convex in $\mathcal{B}_{\infty}(\epsilon)$.
	\end{proof}
	\begin{proof}[Proof of Theorem \ref{Theorem 2}]
		We prove the theorem using the primal-dual witness method, which involves a particular order of steps (see \cite{wainwright2006sharp}). We follow the steps to construct $(\tilde{\boldsymbol{\Omega}}_{l},\tilde{\boldsymbol{Z}}_{l},\tilde{\boldsymbol{M}}_{l};l=1,2,\ldots, L)$. The solutions are expected to satisfy the optimal conditions associated with the optimization problem (\ref{(27)}) in $\mathcal{B}_{\infty}(C_1/d)$ with high probability. If the method succeeds, then $\tilde{\boldsymbol{\Omega}}_{l}$ is equal to the unique solution $\hat{\boldsymbol{\Omega}}_{l}$. 
		
		The KKT conditions corresponding to the optimization problem for (\ref{(27)}) for $(i,j)\in E$ are
		\begin{equation}\label{(50)}
			(\boldsymbol{S}_l-\boldsymbol{\Omega}_{l}^{-1})_{ij}+\frac{\lambda}{L}Z_{l,ij}+\mu M_{l,ij}=0,
		\end{equation}
		while for $(i,j) \in E^c$, the KKT conditions are
		\begin{equation}\label{(51)}
			(\boldsymbol{S}_l-\boldsymbol{\Omega}_{l}^{-1})_{ij}+\frac{\lambda}{L}Z_{l,ij}=0.
		\end{equation}
		Here 
		\begin{equation}
			\label{(52)}
			Z_{l,ij}=\begin{cases}
				sign(\overline{\Omega}_{ij})&, \text{if }\overline{\Omega}_{ij}\neq 0,i\neq j\\
				[-1,1]&, \text{if }\overline{\Omega}_{ij}= 0,i\neq j\\
				0 &, \text{if }i= j\\
			\end{cases}
		\end{equation}
		and, 
		\begin{equation}
			\label{(53)}
			M_{l,ij}=\frac{2(\Omega_{l,ij}-\overline{\Omega}_{ij})}{(\overline{\Omega}_{ij})^{2}}-\frac{2}{L}\frac{\sum_{k}(\Omega_{k,ij}-\overline{\Omega}_{ij})^2}{(\overline{\Omega}_{ij})^{3}}.
		\end{equation}
		We then construct the primal-dual witness solution as follows:
		\begin{itemize}
			\item Obtain the primal solution as
			\begin{equation}\label{(54)}
				\{\Tilde{\boldsymbol{\Omega}}_l,l=1,2,\ldots,L\}=\underset{\{\boldsymbol{\Omega}_{l}\succ 0,\boldsymbol{\Omega}_{l(E^c)}=0,||\boldsymbol{\Omega}_{l}-\boldsymbol{\Omega}_{l}^{0}||_{\infty}\leq C_1/d\}}{arg\ min}\ Q(\boldsymbol{\Omega}).
			\end{equation}
			\item For $(i,j)\in E$, we obtain $\tilde{Z}_{l,ij}$ and $\tilde{M}_{l,ij}$ from (\ref{(52)}) and (\ref{(53)}).
			\item For $(i,j)\in E^c$, define $\tilde{Z}_{l,ij}=\frac{L}{\lambda}(\tilde{\boldsymbol{\Omega}}_{l}^{-1}-\boldsymbol{S}_l)_{ij}$.
		\end{itemize}
		The primal-dual witness approach only succeeds if the strict dual feasibility conditions are met by the primal solution,i.e., we have to show $|\Tilde{Z}_{l,ij}|<1,\forall(i,j)\in E^c, l=1,2,\ldots,L$. Define the following notations:
		\begin{enumerate}
			\item $\boldsymbol{\Delta}_{l}=\Tilde{\boldsymbol{\Omega}}_{l}-\boldsymbol{\Omega}_{l}^{0}$: Difference between primal solutions and true parameter. Note that $\boldsymbol{\Delta}_{l(E^c)}=0,l=1,2,\ldots,L$.
			\item $\boldsymbol{W}_{l}=\boldsymbol{S}_{l}-\boldsymbol{\Sigma}_{l}$: Difference between true covariance matrix and sample covariance matrix.
			\item $R(\boldsymbol{\Delta}_{l})=\boldsymbol{R}_{l}= (\Tilde{\boldsymbol{\Omega}}_{l})^{-1}-(\boldsymbol{\Omega}_{l}^{0})^{-1}+(\boldsymbol{\Omega}_{l}^{0})^{-1}\ \boldsymbol{\Delta}_{l}\ (\boldsymbol{\Omega}_{l}^{0})^{-1}$: Difference between the gradient of $\log\ \det(A)$ at $A=\Tilde{\boldsymbol{\Omega}}_{l}$ and its first-order Taylor expansion around $\boldsymbol{\Omega}_{l}^{0}$.
		\end{enumerate}
		Using $\boldsymbol{\Delta}_{l},\boldsymbol{W}_{l}$ and $\boldsymbol{R}_{l}$, stacking the components into a column vector and disjointly decomposing the coordinates into $E$ and $E^c$, we have
		\begin{eqnarray}
			\label{(55)}
			\eta_{lEE}^{0}vec(\boldsymbol{\Delta}_{l(E)})+vec(\boldsymbol{W}_{l(E)})-vec(\boldsymbol{R}_{l(E)})+\frac{\lambda}{L}vec(\tilde{\boldsymbol{Z}}_{l(E)})+\mu vec(\tilde{\boldsymbol{M}}_{l(E)})&=&0\\
			\label{(56)}
			\eta_{lE^{c}E}^{0}vec(\boldsymbol{\Delta}_{l(E)})+vec(\boldsymbol{W}_{l(E^{c})})-vec(\boldsymbol{R}_{l(E^{c})})+\frac{\lambda}{L}vec(\tilde{\boldsymbol{Z}}_{l(E^{c})})&=&0.
		\end{eqnarray}
		From (\ref{(55)}), we have
		\begin{equation}\label{(57)}
			vec(\boldsymbol{\Delta}_{l(E)})=(\eta_{lEE}^{0})^{-1}\bigg(vec(\boldsymbol{R}_{l(E)})-vec(\boldsymbol{W}_{l(E)})-\frac{\lambda}{L}vec(\tilde{\boldsymbol{Z}}_{l(E)})-\mu vec(\tilde{\boldsymbol{M}}_{l(E)})\bigg).
		\end{equation}
		Substituting (\ref{(57)}) into (\ref{(56)}), we get,
		\begin{eqnarray}
			\label{(58)}
			\notag
			0&=&\eta_{lE^{c}E}^{0}\Bigg[(\eta_{lEE}^{0})^{-1}\bigg(vec(\boldsymbol{R}_{l(E)})-vec(\boldsymbol{W}_{l(E)})-\frac{\lambda}{L}vec(\tilde{\boldsymbol{Z}}_{l(E)})-\mu vec(\tilde{\boldsymbol{M}}_{l(E)})\bigg)\Bigg]\\
			& &+vec(\boldsymbol{W}_{l(E^{c})})-vec(\boldsymbol{R}_{l(E^{c})})+\frac{\lambda}{L}vec(\tilde{\boldsymbol{Z}}_{l(E^{c})})
		\end{eqnarray}
		\begin{eqnarray}
			\label{(59)}
			\notag
			\implies \frac{\lambda}{L}vec(\tilde{\boldsymbol{Z}}_{l(E^{c})}) &=& \eta_{lE^{c}E}^{0} (\eta_{lEE}^{0})^{-1} (vec(\boldsymbol{W}_{l(E)})-vec(\boldsymbol{R}_{l(E)}))-vec(\boldsymbol{W}_{l(E^{c})})+vec(\boldsymbol{R}_{l(E^{c})})\\
			& &+ \eta_{lE^{c}E}^{0} (\eta_{lEE}^{0})^{-1}\bigg(\frac{\lambda}{L}vec(\tilde{\boldsymbol{Z}}_{l(E)})+\mu vec(\tilde{\boldsymbol{M}}_{l(E)})\bigg) 
		\end{eqnarray}
		Hence, we have 
		\begin{eqnarray}\label{(60)}
			\notag
			\frac{\lambda}{L}||\tilde{\boldsymbol{Z}}_{l(E^{c})}||_{\infty}&\leq &|||\eta_{lE^{c}E}^{0} (\eta_{lEE}^{0})^{-1}|||_{\infty}(||\boldsymbol{W}_{l}||_{\infty}+||\boldsymbol{R}_{l}||_{\infty})+(||\boldsymbol{W}_{l}||_{\infty}+||\boldsymbol{R}_{l}||_{\infty})\\
			& &+|||\eta_{lE^{c}E}^{0} (\eta_{lEE}^{0})^{-1}|||_{\infty}\bigg(\frac{\lambda}{L}+\mu ||\tilde{\boldsymbol{M}}_{l(E)}||_{\infty}\bigg).
		\end{eqnarray}
		Using irrepresentability conditions (Assumption \ref{Assumption 5}), we get
		\begin{equation}
			\label{(61)}
			\frac{\lambda}{L}||\tilde{\boldsymbol{Z}}_{l(E^{c})}||_{\infty}\leq (2-\alpha)(||\boldsymbol{W}_{l}||_{\infty}+||\boldsymbol{R}_{l}||_{\infty})+(1-\alpha)\bigg(\frac{\lambda}{L}+\mu ||\tilde{\boldsymbol{M}}_{l(E)}||_{\infty}\bigg)
		\end{equation}
		If $||\tilde{\boldsymbol{M}}_{l(E)}||_{\infty}\leq 1$, for sufficiently large $n$, we have 
		\begin{equation}\label{(62)}
			\frac{\lambda}{L}||\tilde{\boldsymbol{Z}}_{l(E^{c})}||_{\infty}\leq (2-\alpha)(||\boldsymbol{W}_{l}||_{\infty}+||\boldsymbol{R}_{l}||_{\infty})+(1-\alpha)\bigg(\frac{\lambda}{L}+\mu \bigg)
		\end{equation}
		Thus, if $\mu<\frac{\alpha\lambda}{(2-\alpha)L}$ and $max\{||\boldsymbol{W}_{l}||_{\infty},||\boldsymbol{R}_{l}||_{\infty}\}\leq \frac{\alpha}{8}\bigg(\frac{\lambda}{L}+\mu \bigg)$, we have $||\tilde{\boldsymbol{Z}}_{l(E^c)}||_{\infty}<1,l=1,2,\ldots,L$ and hence the strict dual feasibility conditions are satisfied.
		Thus for strict dual feasibility conditions to be satisfied, we need to have:  
		\begin{enumerate}
			\item
			\begin{equation}
				||\tilde{\boldsymbol{M}}_{l(E)}||_{\infty}\leq 1, \text{ for sufficiently large } n
				\label{(63)}
			\end{equation}
			\item
			\begin{equation}
				||\boldsymbol{W}_{l}||_{\infty}\leq \frac{\alpha}{8}\bigg(\frac{\lambda}{L}+\mu \bigg)
				\label{(64)}
			\end{equation}
			\item 
			\begin{equation}
				||\boldsymbol{R}_{l}||_{\infty}\leq \frac{\alpha}{8}\bigg(\frac{\lambda}{L}+\mu \bigg)
				\label{(65)}
			\end{equation}
			\item 
			\begin{equation}
				\mu<\frac{\alpha\lambda}{(2-\alpha)L}
				\label{(66)}
			\end{equation}
		\end{enumerate}
		\begin{lemma}\label{Lemma 0.2}
			(Controlling $\boldsymbol{W}_{l}$) Suppose $\boldsymbol{W}_{l}=\boldsymbol{S}_{l}-\boldsymbol{\Sigma}_{l}^{0}$ be the difference between the sample covariance matrix and the true covariance matrix. Suppose the sample size $n_l$ be such that $\delta_{f}(n_l,p^{\gamma})\leq \underset{l=1,2,\ldots,L}{min}\frac{1}{v_{l0}}$, for all $l=1,2,\ldots,L$, where $\delta_{f}(n_l,p^{\gamma})$ is the inverse function from (\ref{(19)}). Then, for any $\gamma>2$, we have
			\begin{equation}
				\mathbb{P}\bigg(||W_{l}||_{\infty}\geq \delta_{f}(n_l,p^{\gamma})\bigg)\leq \frac{1}{p^{\gamma-2}}\rightarrow 0,
				\label{(67)}
			\end{equation}
			for all $l=1,2,\ldots,L$.
		\end{lemma}
		If we take $\delta=\underset{l}{max}\ \delta_{f}(n_l,p^{\gamma})$ and $\frac{\alpha}{8}\bigg(\frac{\lambda}{L}+\mu \bigg)=\delta$, then from (\ref{(67)}), we have $\mathbb{P}\bigg(||W_{l}||_{\infty}\geq \delta\bigg)\leq \frac{1}{p^{\gamma-2}}\rightarrow 0$, for all $l=1,2,\ldots,L$. Thus, we have (\ref{(67)}) with high probability for each $l=1,2,\ldots,L$, if $\frac{\alpha}{8}\bigg(\frac{\lambda}{L}+\mu \bigg)=\delta$.
		\begin{lemma}\label{Lemma 0.3}
			(Controlling $\boldsymbol{R}_{l}$) Suppose
			\begin{equation}\label{(68)}
				||\boldsymbol{\Delta}_{l}||_{\infty}\leq \frac{1}{3d\kappa_{\boldsymbol{\Sigma}_l^{0}}}, l=1,2,\ldots,L.
			\end{equation}
			Then $\boldsymbol{J}_{l}=\sum_{i=0}^{\infty}(-1)^{i}((\boldsymbol{\Omega}_{l}^{0})^{-1}\boldsymbol{\Delta}_{l})^{i}$ satisfies
			\begin{equation}\label{(69)}
				|||(\boldsymbol{J}_l)^{T}|||_{\infty}\leq \frac{3}{2},
			\end{equation}
			and $\boldsymbol{R}_l=(\boldsymbol{\Omega}_{l}^{0})^{-1}\boldsymbol{\Delta}_{l}(\boldsymbol{\Omega}_{l}^{0})^{-1}\boldsymbol{\Delta}_{l}\boldsymbol{J}_{l}(\boldsymbol{\Omega}_{l}^{0})^{-1}$ satisfies
			\begin{equation}\label{(70)}
				||\boldsymbol{R}_l||_{\infty}\leq \frac{3}{2}d ||\boldsymbol{\Delta}_{l}||_{\infty}^{2}\kappa^{3}_{\boldsymbol{\Sigma}^{0}_{l}}.
			\end{equation}
		\end{lemma}
		\begin{lemma}\label{Lemma 0.4}
			(Controlling $\tilde{\boldsymbol{M}}_{l(E)}$ and $\boldsymbol{\Delta}_{l}$) Suppose
			\begin{equation}\label{(75)}
				r_l= 2\kappa_{\boldsymbol{\eta}_{l}^{0}}\bigg(||\boldsymbol{W}_{l}||_{\infty}+(\frac{\lambda}{L}+\mu)\bigg) \leq \frac{1}{d}\ \underset{l_1}{min}\left\{ min\left\{\frac{1}{3\kappa_{\boldsymbol{\Sigma}_{l_1}^{0}}},\frac{1}{3\kappa^{3}_{\boldsymbol{\Sigma}_{l_1}^{0}}\kappa_{\boldsymbol{\eta}_{l_1}^{0}}}\right\}\right\}=\frac{C_1}{d},
			\end{equation}
			for all $l=1,2,\ldots,L$. Define $m_{0}=\underset{(i,j)\in E}{min}\ |\overline{\Omega}_{ij}^{0}|$ and $m_1=\underset{l}{max}\ \underset{(i,j)\in E}{max}\frac{|\Omega_{l,ij}^{0}-\overline{\Omega}_{ij}^{0}|}{|\overline{\Omega}_{ij}^{0}|}$. Let $r=\underset{l}{max}\ r_l$. Also suppose
			\begin{enumerate}
				\item \label{Lemma 0.4.1}
				\begin{eqnarray}
					\label{(76)}
					m_{0} & > & r+r^{1/4},
				\end{eqnarray}
				\item \label{Lemma 0.4.2}
				\begin{eqnarray}
					\label{(77)}
					m_{1} & \leq & \frac{-1+\sqrt{1+2r^{1/4}}}{2},
				\end{eqnarray}
				\item \label{Lemma 0.4.3}
				\begin{eqnarray}
					\label{(78)}
					r & \rightarrow & 0\text{ for sufficiently large }n.
				\end{eqnarray} 
			\end{enumerate}
			Then, 
			\begin{equation}\label{(79)}
				||\boldsymbol{\Delta}_{l}||_{\infty}=||\Tilde{\boldsymbol{\Omega}}_{l}-\boldsymbol{\Omega}_{l}^{0}||_{\infty}\leq r_l,
			\end{equation}
			for all $l=1,2,\ldots,L$.
		\end{lemma}
		Define $\mathcal{A}_{l}=\{||\boldsymbol{W}_{l}||_{\infty}\leq \delta_{f}(n_{l},p^{\gamma})\},l=1,2,\ldots,L$ and some $\gamma>2$. If $\delta_{f}(n_{l},p^{\gamma})\leq \underset{l}{min}\ \frac{1}{v_{l0}},\forall l=1,2,\ldots,L$, we have $\mathbb{P}(\mathcal{A}_{l})\geq 1-\frac{1}{p^{\gamma-2}}\rightarrow 1$, and hence $\mathbb{P}(\mathcal{A}_{l},l=1,2,\ldots,L)\geq 1-\frac{L}{p^{\gamma-2}}$. Also, $\delta_{f}(n_{l},p^{\gamma})\leq \underset{l}{min}\ \frac{1}{v_{l0}},\forall l=1,2,\ldots,L$ implies $n_{l}\geq n_{f}(\underset{l}{min}\ \frac{1}{v_{l0}},p^{\gamma}),\forall l=1,2,\ldots,L$. Let $\mu+(\lambda/L)=(8/\alpha)\delta$, where $\delta=\underset{l}{max}\ \delta_{f}(n_{l},p^{\gamma})$. Under $\mathcal{A}_{l}$, we have $||\boldsymbol{W}_{l}||_{\infty}\leq \delta=\frac{\alpha}{8}\bigg(\frac{\lambda}{L}+\mu\bigg)$, for all $l=1,2,\ldots,L$. From (\ref{(31)}) and the interpretation of the tail conditions in (\ref{(20)}), we have
		\begin{equation}
			\label{(82)}
			\delta<\frac{1}{6d(1+(8/\alpha))\kappa_{\boldsymbol{\eta}_{l}^{0}}}\underset{l_1}{min}\bigg\{min\bigg\{\frac{1}{\kappa_{\boldsymbol{\Sigma}_{l_1}^{0}}},\frac{1}{\kappa_{\boldsymbol{\Sigma}_{l_1}^{0}}^{3}\kappa_{\boldsymbol{\eta}_{l_1}^{0}}}\bigg\}\bigg\},
		\end{equation}
		for all $l=1,2,\ldots,L$. Thus,
		$$
		2\kappa_{\boldsymbol{\eta}_{l}^{0}}\bigg(||\boldsymbol{W}_{l}||_{\infty}+\frac{\lambda}{L}+\mu\bigg)\leq 2\kappa_{\boldsymbol{\eta}_{l}^{0}}\bigg(1+\frac{8}{\alpha}\bigg)\delta \leq \frac{1}{d} \underset{l_1}{min}\bigg\{min\bigg\{\frac{1}{3\kappa_{\boldsymbol{\Sigma}_{l_1}^{0}}},\frac{1}{3\kappa_{\boldsymbol{\Sigma}_{l_1}^{0}}^{3}\kappa_{\boldsymbol{\eta}_{l_1}^{0}}}\bigg\}\bigg\}=\frac{C_1}{d}.
		$$
		From Lemma \ref{Lemma 0.4}, we have $||\boldsymbol{\Delta}_{l}||_{\infty}\leq r_l$, for all $l=1,2,\ldots,L$. Thus from Lemma \ref{Lemma 0.3}, we have $||\boldsymbol{R}_l||_{\infty}\leq \frac{3}{2}d ||\boldsymbol{\Delta}_{l}||_{\infty}^{2}\kappa^{3}_{\boldsymbol{\Sigma}^{0}_{l}}$. Substituting
		$$
		||\boldsymbol{\Delta}_{l}||_{\infty}\leq r_l \leq \frac{1}{d} \underset{l_1}{min}\bigg\{ \frac{1}{3\kappa_{\boldsymbol{\Sigma}_{l_1}^{0}}^{3}\kappa_{\boldsymbol{\eta}_{l_1}^{0}}}\bigg\},
		$$
		which implies $||\boldsymbol{\Delta}_{l}||_{\infty}\leq 1/(3d\kappa_{\boldsymbol{\Sigma}_{l}^{0}}^{3}\kappa_{\boldsymbol{\eta}_{l}^{0}})$, for all $l=1,2,\ldots,L$, we have 
		$$
		||\boldsymbol{R}_{l}||_{\infty}\leq \frac{\alpha}{8}\bigg(\frac{\lambda}{L}+\mu\bigg).
		$$
		Hence, $max\{||\boldsymbol{W}_{l}||_{\infty},||\boldsymbol{R}_{l}||_{\infty}\}\leq \frac{\alpha}{8}\bigg(\frac{\lambda}{L}+\mu\bigg)$, and thus strict dual feasibility conditions are satisfied under $\{\mathcal{A}_{l},l=1,2,\ldots,L\}$. Since conditioned on the event that the edge set has been identified, the primal solution is the unique solution of the optimization problem in (\ref{(27)}), we have $\tilde{\boldsymbol{\Omega}}_{l}=\hat{\boldsymbol{\Omega}}_{l},l=1,2,\ldots,L$. Hence, in $\mathcal{A}_{l}$, we have $\hat{\boldsymbol{\Omega}}_{l}$ satisfying
		\begin{equation}
			\label{(83)}
			||\hat{\boldsymbol{\Omega}}_{l}-\boldsymbol{\Omega}_{l}^{0}||_{\infty}\leq 2\kappa_{\boldsymbol{\eta}_{l}^{0}}\bigg(1+\frac{8}{\alpha}\bigg)\delta,
		\end{equation}
		for all $l=1,2,\ldots,L$. Thus the above statements hold for all $l=1,2,\ldots,L$ in $\mathcal{B}_{\infty}(C_1/d)$ with probability at least $1-\frac{L}{p^{\gamma-2}}\rightarrow 1$, for finite $L$ under the assumptions
		\begin{description}
			\item[(A1)]\label{ass:A1} $|\overline{\Omega}_{ij}^{0}|
			\ge r+r^{1/4},
			\qquad (i,j)\in E.$
			
			\item[(A2)]\label{ass:A2} $\underset{l}{max}\ \underset{(i,j)\in E}{max}\frac{|\Omega_{l,ij}^{0}-\overline{\Omega}_{ij}^{0}|}{|\overline{\Omega}_{ij}^{0}|}\leq  \frac{-1+\sqrt{1+2r^{1/4}}}{2}$
			
			\item[(A3)]\label{ass:A3} $r\rightarrow 0$ as $n\rightarrow \infty$.
		\end{description}
		However, since $r$ is a random quantity, we can use Lemma \ref{Lemma 0.2} and Lemma \ref{Lemma 0.4} to provide non-random bounds on assumptions (A1), (A2), and (A3) with high probability.
		\begin{enumerate}
			\item From Lemma \ref{Lemma 0.4}, since $r_l\leq C_1/d$,  where $$
			C_1= \underset{l_1}{min}\left\{ min\left\{\frac{1}{3\kappa_{\boldsymbol{\Sigma}_{l_1}^{0}}},\frac{1}{3\kappa^{3}_{\boldsymbol{\Sigma}_{l_1}^{0}}\kappa_{\boldsymbol{\eta}_{l_1}^{0}}}\right\}\right\},
			$$
			we have $$|\overline{\Omega}_{ij}^{0}|>\frac{C_1}{d}+\bigg(\frac{C_1}{d}\bigg)^{1/4}\implies |\overline{\Omega}_{ij}^{0}|\geq r+r^{1/4}.$$
			\item Note that $r\geq 2\underset{l}{max}\ \kappa_{\boldsymbol{\eta}_{l}^{0}}(8/\alpha)\delta$ with high probability, which implies $r=O_{p}(\delta)$. Thus, with any bounded constant $c$, with $\delta=o(1)$,
			$$
			\underset{l}{max}\ \underset{(i,j)\in E}{max}\frac{|\Omega_{l,ij}^{0}-\overline{\Omega}_{ij}^{0}|}{|\overline{\Omega}_{ij}^{0}|}\leq \frac{-1+\sqrt{1+2c\delta^{1/4}}}{2}
			$$
			serves a bound for $\underset{l}{max}\ \underset{(i,j)\in E}{max}\frac{|\Omega_{l,ij}^{0}-\overline{\Omega}_{ij}^{0}|}{|\overline{\Omega}_{ij}^{0}|}\leq  \frac{-1+\sqrt{1+2r^{1/4}}}{2}$ to hold.
		\end{enumerate}
		
	\end{proof}
	\subsubsection{Proof of relevant lemmas}
	\begin{proof}[Proof of Lemma \ref{Lemma 0.3}]
		\begin{eqnarray}\label{(71)}
			\notag
			\boldsymbol{R}_{l}=R(\boldsymbol{\Delta}_{l})&=& (\Tilde{\boldsymbol{\Omega}}_{l})^{-1}-(\boldsymbol{\Omega}_{l}^{0})^{-1}+(\boldsymbol{\Omega}_{l}^{0})^{-1}\ \boldsymbol{\Delta}_{l}\ (\boldsymbol{\Omega}_{l}^{0})^{-1}\\
			&=&     (\boldsymbol{\Omega}_{l}^{0}+\boldsymbol{\Delta}_{l})^{-1}-(\boldsymbol{\Omega}_{l}^{0})^{-1}+(\boldsymbol{\Omega}_{l}^{0})^{-1}\ \boldsymbol{\Delta}_{l}\ (\boldsymbol{\Omega}_{l}^{0})^{-1}
		\end{eqnarray}
		Now,
		\begin{eqnarray}\label{(72)}
			\notag
			(\boldsymbol{\Omega}_{l}^{0}+\boldsymbol{\Delta}_{l})^{-1}&=& (\boldsymbol{I}+(\boldsymbol{\Omega}_{l}^{0})^{-1}\boldsymbol{\Delta}_{l})^{-1}(\boldsymbol{\Omega}_{l}^{0})^{-1}\\
			\notag
			&=&\sum_{i=0}^{\infty}(-1)^{i}((\boldsymbol{\Omega}_{l}^{0})^{-1}\boldsymbol{\Delta}_{l})^{i}(\boldsymbol{\Omega}_{l}^{0})^{-1}\\
			\notag
			&=& (\boldsymbol{\Omega}_{l}^{0})^{-1}-(\boldsymbol{\Omega}_{l}^{0})^{-1}\boldsymbol{\Delta}_{l}(\boldsymbol{\Omega}_{l}^{0})^{-1}+\sum_{i=2}^{\infty}(-1)^{i}((\boldsymbol{\Omega}_{l}^{0})^{-1}\boldsymbol{\Delta}_{l})^{i}(\boldsymbol{\Omega}_{l}^{0})^{-1}\\
			&=& (\boldsymbol{\Omega}_{l}^{0})^{-1}-(\boldsymbol{\Omega}_{l}^{0})^{-1}\boldsymbol{\Delta}_{l}(\boldsymbol{\Omega}_{l}^{0})^{-1}+(\boldsymbol{\Omega}_{l}^{0})^{-1}\boldsymbol{\Delta}_{l}(\boldsymbol{\Omega}_{l}^{0})^{-1}\boldsymbol{\Delta}_{l}\boldsymbol{J}_{l}(\boldsymbol{\Omega}_{l}^{0})^{-1}.
		\end{eqnarray}
		Hence, we have $\boldsymbol{R}_{l}=((\boldsymbol{\Omega}_{l}^{0})^{-1}\boldsymbol{\Delta}_{l})^{2}\boldsymbol{J}_{l}(\boldsymbol{\Omega}_{l}^{0})^{-1}$. Now,
		\begin{eqnarray}\label{(73)}
			\notag
			||\boldsymbol{R}_{l}||_{\infty}&=&\underset{i,j}{max}|\boldsymbol{e}^{T}_i (\boldsymbol{\Omega}_{l}^{0})^{-1}\boldsymbol{\Delta}_{l}(\boldsymbol{\Omega}_{l}^{0})^{-1}\boldsymbol{\Delta}_{l}\boldsymbol{J}_{l}(\boldsymbol{\Omega}_{l}^{0})^{-1}\boldsymbol{e}_{j}|\\
			\notag
			&\leq & \underset{i}{max}||\boldsymbol{e}^{T}_i (\boldsymbol{\Omega}_{l}^{0})^{-1}\boldsymbol{\Delta}_{l}||_{\infty} \ \underset{j}{max}||(\boldsymbol{\Omega}_{l}^{0})^{-1}\boldsymbol{\Delta}_{l}\boldsymbol{J}_{l}(\boldsymbol{\Omega}_{l}^{0})^{-1}\boldsymbol{e}_{j}||_{1}\\
			\notag
			&\leq & \underset{i}{max}||\boldsymbol{e}^{T}_i (\boldsymbol{\Omega}_{l}^{0})^{-1}||_{1}||\boldsymbol{\Delta}_{l}||_{\infty}\ \underset{j}{max}||(\boldsymbol{\Omega}_{l}^{0})^{-1}\boldsymbol{\Delta}_{l}\boldsymbol{J}_{l}(\boldsymbol{\Omega}_{l}^{0})^{-1}\boldsymbol{e}_{j}||_{1}\\
			\notag
			&\leq & |||(\boldsymbol{\Omega}_{l}^{0})^{-1}|||_{\infty}||\boldsymbol{\Delta}_{l}||_{\infty}|||(\boldsymbol{\Omega}_{l}^{0})^{-1}\boldsymbol{J}_{l}\boldsymbol{\Delta}_{l}(\boldsymbol{\Omega}_{l}^{0})^{-1}|||_{\infty}\\
			&\leq & d ||\boldsymbol{\Delta}_{l}||^{2}_{\infty} \kappa_{\boldsymbol{\Sigma}_{l}^{0}}^{3}|||\boldsymbol{J}_{l}^{T}|||_{\infty}
		\end{eqnarray}
		Note that $|||\boldsymbol{J}_{l}^{T}|||_{\infty}\leq \sum_{i=0}^{\infty}|||\boldsymbol{\Delta}_{l}(\boldsymbol{\Omega}_{l}^{0})^{-1}|||_{\infty}^{i}\leq \frac{1}{1-|||(\boldsymbol{\Omega}_{l}^{0})^{-1}|||_{\infty}|||\boldsymbol{\Delta}_{l}|||_{\infty}}\leq 3/2$. Hence, we have
		\begin{equation}\label{(74)}
			||\boldsymbol{R}_l||_{\infty}\leq \frac{3}{2}d ||\boldsymbol{\Delta}_{l}||_{\infty}^{2}\kappa^{3}_{\boldsymbol{\Sigma}^{0}_{l}}.
		\end{equation}
	\end{proof}
	\begin{proof}[Proof of Lemma \ref{Lemma 0.4}]
		Define $G(\boldsymbol{\Omega}_{l(E)})=-[(\boldsymbol{\Omega}_{l})^{-1}]_{(E)}+\boldsymbol{S}_{l(E)}+\frac{\lambda}{L} \tilde{\boldsymbol{Z}}_{l(E)}+\mu \tilde{\boldsymbol{M}}_{l(E)} $ as the gradient function of the optimization problem (\ref{(27)}) in the set $E$. Since $Q(\boldsymbol{\Omega})$ is strictly convex in $\mathcal{B}_{\infty}(C_1/d)$ if $\mu$ satisfies (\ref{(49)}), we have a unique minimum of $Q(\boldsymbol{\Omega})$, which is $\Tilde{\boldsymbol{\Omega}}=\{\Tilde{\boldsymbol{\Omega}}_{l(E)},l=1,2,\ldots,L\}$. Consequently this gradient function shall vanish at $\Tilde{\boldsymbol{\Omega}}_{l(E)}$ and is a unique solution of the equation $G(\boldsymbol{\Omega}_{l(E)})=0$. Define the supnorm ball as $\mathcal{B}(r_{l})=\{\boldsymbol{B}_{(E)}:||\boldsymbol{B}_{(E)}||_{\infty}\leq r_{l}\}$. Note that $\cup_{l}\mathcal{B}(r_{l})\subseteq \mathcal{B}_{\infty}(C_1/d)$. The idea is to define a continuous map $F$, such that any value from the convex set $\mathcal{B}(r_{l})$ shall be mapped onto $\mathcal{B}(r_{l})$ itself, and hence we have $F(\mathcal{B}(r_{l}))\subset\mathcal{B}(r_{l})$. Using Brouwer's fixed point theorem, we have some fixed point $\boldsymbol{\Delta}_{l(E)}$ in $\mathcal{B}(r_{l})$. From the uniqueness of $G(\boldsymbol{\Omega}_{l(E)})=0$ in $\mathcal{B}_{\infty}(C_1/d)$, we conclude that $||\boldsymbol{\Delta}_{l}||_{\infty}\leq r_{l}$, for all $l=1,2,\ldots,L$. Define the continuous $F$ as 
		\begin{equation}\label{(80)}
			F(vec(\boldsymbol{\Delta}_{l(E)}))=-(\boldsymbol{\eta}^{0}_{lEE})^{-1}\ vec(G(\boldsymbol{\Omega}_{l(E)}^{0}+\boldsymbol{\Delta}_{l(E)}))+vec(\boldsymbol{\Delta}_{l(E)}).
		\end{equation}
		Note that $F(vec(\boldsymbol{\Delta}_{l(E)}))=vec(\boldsymbol{\Delta}_{l(E)})$ if and only if $G(\boldsymbol{\Omega}_{l(E)}^{0}+\boldsymbol{\Delta}_{l(E)})=G(\Tilde{\boldsymbol{\Omega}}_{l})=0$. Let $\boldsymbol{\Delta}_{l}\in \mathbb{R}^{p\times p}$ such that $\boldsymbol{\Delta}_{l(E)}\in \mathcal{B}(r_l)$ and $\boldsymbol{\Delta}_{l(E^c)}=0$. Now,
		\begin{eqnarray*}
			G(\boldsymbol{\Omega}_{l(E)}^{0}+\boldsymbol{\Delta}_{l(E)})&=&-[(\boldsymbol{\Omega}_{l}^{0}+\boldsymbol{\Delta}_{l})^{-1}]_{(E)}+\boldsymbol{S}_{l(E)}+\frac{\lambda}{L} \tilde{\boldsymbol{Z}}_{l(E)}+\mu \tilde{\boldsymbol{M}}_{l(E)}\\
			&=& -[(\boldsymbol{\Omega}_{l}^{0}+\boldsymbol{\Delta}_{l})^{-1}]_{(E)}+ \boldsymbol{\Sigma}_{l(E)}^{0}+\boldsymbol{W}_{l(E)}+\frac{\lambda}{L}\tilde{\boldsymbol{Z}}_{l(E)}+\mu \tilde{\boldsymbol{M}}_{l(E)}.
		\end{eqnarray*}
		Thus,
		\begin{eqnarray*}
			F(vec(\boldsymbol{\Delta}_{l(E)})) &=& (\boldsymbol{\eta}^{0}_{lEE})^{-1}\ (vec(((\boldsymbol{\Omega}_{l}^{0}+\boldsymbol{\Delta}_{l})^{-1})_{(E)})-vec(\boldsymbol{\Sigma}^{0}_{l(E)})-vec(\boldsymbol{W}_{l(E)})\\
			& &-\frac{\lambda}{L}vec(\tilde{\boldsymbol{Z}}_{l(E)})-\mu vec(\tilde{\boldsymbol{M}}_{l(E)}))+vec(\boldsymbol{\Delta}_{l(E)})\\
			&=& (\boldsymbol{\eta}^{0}_{lEE})^{-1}\ (vec(((\boldsymbol{\Omega}_{l}^{0}+\boldsymbol{\Delta}_{l})^{-1})_{(E)})-vec(((\boldsymbol{\Omega}_{l}^{0})^{-1})_{(E)})-vec(\boldsymbol{W}_{l(E)})\\
			& &-\frac{\lambda}{L}vec(\tilde{\boldsymbol{Z}}_{l(E)})-\mu vec(\tilde{\boldsymbol{M}}_{l(E)}))+vec(\boldsymbol{\Delta}_{l(E)}).
		\end{eqnarray*}
		From (\ref{(72)}), we have
		$$
		vec(((\boldsymbol{\Omega}_{l}^{0}+\boldsymbol{\Delta}_{l})^{-1})_{(E)})-vec(((\boldsymbol{\Omega}_{l}^{0})^{-1})_{(E)})+\boldsymbol{\eta}_{lEE}^{0}vec(\boldsymbol{\Delta}_{l(E)})=vec((((\boldsymbol{\Omega}_{l}^{0})^{-1}\boldsymbol{\Delta}_{l})^{2}\boldsymbol{J}_{l}(\boldsymbol{\Omega}_{l}^{0})^{-1})_{(E)}).
		$$
		Hence, we have
		\begin{eqnarray}\label{(81)}
			\notag
			F(vec(\boldsymbol{\Delta}_{l(E)})) &=& (\boldsymbol{\eta}^{0}_{lEE})^{-1}\ vec((((\boldsymbol{\Omega}_{l}^{0})^{-1}\boldsymbol{\Delta}_{l})^{2}\boldsymbol{J}_{l}(\boldsymbol{\Omega}_{l}^{0})^{-1})_{(E)})-(\boldsymbol{\eta}^{0}_{lEE})^{-1}\bigg(vec(\boldsymbol{W}_{l(E)})\\
			\notag
			& &+\frac{\lambda}{L}vec(\tilde{\boldsymbol{Z}}_{l(E)})+\mu vec(\tilde{\boldsymbol{M}}_{l(E)})\bigg)\\
			&=& T_{1l}+T_{2l}
		\end{eqnarray}
		\begin{enumerate}
			\item Note that
			$$
			||T_{1l}||_{\infty}=||(\boldsymbol{\eta}^{0}_{lEE})^{-1}\ vec((((\boldsymbol{\Omega}_{l}^{0})^{-1}\boldsymbol{\Delta}_{l})^{2}\boldsymbol{J}_{l}(\boldsymbol{\Omega}_{l}^{0})^{-1})_{(E)})||_{\infty}\leq \kappa_{\boldsymbol{\eta}_{l}^{0}}||\boldsymbol{R}_{l}||_{\infty}
			$$
			Now, since from our assumptions, $||\boldsymbol{\Delta}_{l}||_{\infty}\leq r_l\leq \frac{1}{d}\underset{l_1}{min}\frac{1}{3\kappa_{\boldsymbol{\Sigma}_{l_1}^{0}}}\leq \frac{1}{3d\kappa_{\boldsymbol{\Sigma}_{l}^{0}}},l=1,2,\ldots,L$, from Lemma \ref{Lemma 0.3}, we have $||\boldsymbol{R}_l||_{\infty}\leq \frac{3}{2}d ||\boldsymbol{\Delta}_{l}||_{\infty}^{2}\kappa^{3}_{\boldsymbol{\Sigma}^{0}_{l}}$. Also, $||\boldsymbol{\Delta}_{l}||_{\infty}\leq r_l\leq \frac{1}{3d\kappa_{\boldsymbol{\Sigma}_{l}^{0}}^{3}\kappa_{\boldsymbol{\eta}_{l}^{0}}},l=1,2,\ldots,L$. Thus, $||\boldsymbol{R}_l||_{\infty}\leq \frac{r_l}{2\kappa_{\boldsymbol{\eta}_{l}^{0}}}$, which implies $||T_{1l}||_{\infty}\leq r_{l}/2$.
			\item Again, note that
			$$
			||T_{2l}||_{\infty}\leq \kappa_{\boldsymbol{\eta}_{l}^{0}}\bigg(||\boldsymbol{W}_{l}||_{\infty}+\frac{\lambda}{L}+\mu ||\tilde{\boldsymbol{M}}_{l}||_{\infty}\bigg).
			$$
			Since, $\tilde{\Omega}_{l,ij}=\Omega_{l,ij}^{0}+\Delta_{l,ij}$, we have
			\begin{eqnarray*}
				\tilde{M}_{l,ij}&=& \frac{2(\Omega_{l,ij}^{0}-\overline{\Omega}_{ij}^{0}+\Delta_{l,ij}-\overline{\Delta}_{ij})}{(\overline{\Omega}_{ij}^{0}+\overline{\Delta}_{ij})^2}-\frac{2}{L}\sum_{k=1}^{L} \frac{(\Omega_{k,ij}^{0}-\overline{\Omega}_{ij}^{0}+\Delta_{k,ij}-\overline{\Delta}_{ij})^{2}}{(\overline{\Omega}_{ij}^{0}+\overline{\Delta}_{ij})^3}\\
				\implies|\tilde{M}_{l,ij}|&\leq & \frac{2(|\Omega_{l,ij}^{0}-\overline{\Omega}_{ij}^{0}|+|\Delta_{l,ij}-\overline{\Delta}_{ij}|)}{(\overline{\Omega}_{ij}^{0}+\overline{\Delta}_{ij})^2}+\frac{2}{L}\sum_{k=1}^{L} \frac{(|\Omega_{k,ij}^{0}-\overline{\Omega}_{ij}^{0}|+|\Delta_{k,ij}-\overline{\Delta}_{ij}|)^{2}}{|\overline{\Omega}_{ij}^{0}+\overline{\Delta}_{ij}|^3}\\
				&=& \beta_{1}+\beta_{2}
			\end{eqnarray*}
			Since, $\boldsymbol{\Delta}_{l(E)}\in \mathcal{B}(r_l)$, we have $|\Delta_{l,ij}-\overline{\Delta}_{ij}|\leq |\Delta_{l,ij}|+|\overline{\Delta}_{ij}|\leq r_l+\overline{r}\leq 2r$. Hence, we have,
			\begin{equation*}
				|\tilde{M}_{l,ij}|\leq  \frac{2\  \underset{l}{max}(|\Omega_{l,ij}^{0}-\overline{\Omega}_{ij}^{0}|+2\ r)}{(\overline{\Omega}_{ij}^{0}+\overline{\Delta}_{ij})^2}+ \frac{2\ \underset{l}{max}\ (|\Omega_{l,ij}^{0}-\overline{\Omega}_{ij}^{0}|+2r)^{2}}{|\overline{\Omega}_{ij}^{0}+\overline{\Delta}_{ij}|^3}.
			\end{equation*}
			Since $-r_l\leq \Delta_{l,ij}\leq r_l,l=1,2,\ldots L$, we have, $|\overline{\Omega}_{ij}^{0}+\overline{\Delta}_{ij}|\geq max(0,|\overline{\Omega}_{ij}^{0}|-r)$. From (\ref{(76)}), we have, $|\overline{\Omega}_{ij}^{0}+\overline{\Delta}_{ij}|\geq |\overline{\Omega}_{ij}^{0}|-r$. Then, we have 
			\begin{eqnarray*}
				\beta_{1} &\leq & \frac{2\ (\underset{l}{max}|\Omega_{l,ij}^{0}-\overline{\Omega}_{ij}^{0}|+2\ r)}{(|\overline{\Omega}_{ij}^{0}|-r)^{2}}.
			\end{eqnarray*}
			Similarly,
			\begin{equation*}
				\beta_{2} \leq \frac{2\ (\underset{l}{max}|\Omega_{l,ij}^{0}-\overline{\Omega}_{ij}^{0}|+2\ r)^{2}}{(|\overline{\Omega}_{ij}^{0}|-r)^{3}}.
			\end{equation*}
			Then,
			\begin{equation*}
				|\tilde{M}_{l,ij}|\leq  \frac{2\ (\underset{l}{max}|\Omega_{l,ij}^{0}-\overline{\Omega}_{ij}^{0}|+2\ r)}{(|\overline{\Omega}_{ij}^{0}|-r)^{2}}\bigg(1+\frac{\underset{l}{max}|\Omega_{l,ij}^{0}-\overline{\Omega}_{ij}^{0}|+2\ r}{|\overline{\Omega}_{ij}^{0}|-r}\bigg).
			\end{equation*}
			Note that
			\begin{equation*}
				\frac{\underset{l}{max}|\Omega_{l,ij}^{0}-\overline{\Omega}_{ij}^{0}|+2\ r}{|\overline{\Omega}_{ij}^{0}|-r}=\frac{\underset{l}{max}|\Omega_{l,ij}^{0}-\overline{\Omega}_{ij}^{0}|+2\ r}{|\overline{\Omega}_{ij}^{0}|}\times \frac{1}{1-\frac{r}{|\overline{\Omega}_{ij}^{0}|}}.
			\end{equation*}
			From (\ref{(76)}) and (\ref{(78)}), we have $\frac{1}{1-\frac{r}{|\overline{\Omega}_{ij}^{0}|}}\leq 1$, for sufficiently large $n$. Again,
			\begin{eqnarray*}
				\frac{\underset{l}{max}|\Omega_{l,ij}^{0}-\overline{\Omega}_{ij}^{0}|+2\ r}{|\overline{\Omega}_{ij}^{0}|} & \leq & \frac{\underset{l}{max}|\Omega_{l,ij}^{0}-\overline{\Omega}_{ij}^{0}|}{|\overline{\Omega}_{ij}^{0}|}+ \frac{2r}{r+r^{1/4}}\\
				&\rightarrow &\frac{\underset{l}{max}|\Omega_{l,ij}^{0}-\overline{\Omega}_{ij}^{0}|}{|\overline{\Omega}_{ij}^{0}|}=\theta_{ij}
			\end{eqnarray*}
			Thus, for sufficiently large $n$, we have
			\begin{equation*}
				\frac{\underset{l}{max}|\Omega_{l,ij}^{0}-\overline{\Omega}_{ij}^{0}|+2\ r}{|\overline{\Omega}_{ij}^{0}|-r}\leq \theta_{ij}.
			\end{equation*}
			Hence,
			\begin{eqnarray*}
				|\tilde{M}_{l,ij}| & \leq & \frac{2\ (\underset{l}{max}|\Omega_{l,ij}^{0}-\overline{\Omega}_{ij}^{0}|+2\ r)}{(|\overline{\Omega}_{ij}^{0}|-r)^{2}}(1+\theta_{ij})\\
				& = & \frac{2\ (\underset{l}{max}|\Omega_{l,ij}^{0}-\overline{\Omega}_{ij}^{0}|+2\ r)}{(|\overline{\Omega}_{ij}^{0}|-r)}\times \frac{(1+\theta_{ij})}{(|\overline{\Omega}_{ij}^{0}|-r)}\\
				&\leq& 2\frac{\theta_{ij}(1+\theta_{ij})}{r^{1/4}}
			\end{eqnarray*}
			From (\ref{(77)}) we have $|\tilde{M}_{l,ij}|\leq 1$, and hence $||T_{2l}||_{\infty}\leq r_l/2$ for sufficiently large $n$.
		\end{enumerate}
		Thus for sufficiently large sample size $n_l$, we have $||F(vec(\boldsymbol{\Delta}_{l(E)}))||_{\infty}\leq r_l$. So $F(\mathcal{B}(r_{l}))\subset\mathcal{B}(r_{l})$.
	\end{proof}
	\subsection{Algorithmic convergence with increasing ADMM steps}
	This subsection empirically demonstrates the convergence behavior of the Mglasso algorithm by tracking the augmented Lagrangian objective value as a function of ADMM iterations for both chain and star graph structures at dimension $p = 25$. Results are presented in Figures \ref{Figure 15} and \ref{Figure 16} for two sample sizes, $n = 50$ and $n = 500$, illustrating that the augmented Lagrangian decreases monotonically across ADMM steps in all settings, reaching near-convergence within relatively few iterations. The convergence is notably faster and smoother at the larger sample size, consistent with the improved conditioning of the optimization problem when more data are available. Together, these plots provide empirical evidence that the alternating ADMM-with-gradient-descent procedure reliably minimizes the objective and is robust across varying sample regimes.
	\begin{figure}
		
		\centering{
			\subfloat[$n=50$]{
				
				\includegraphics[scale=0.3]{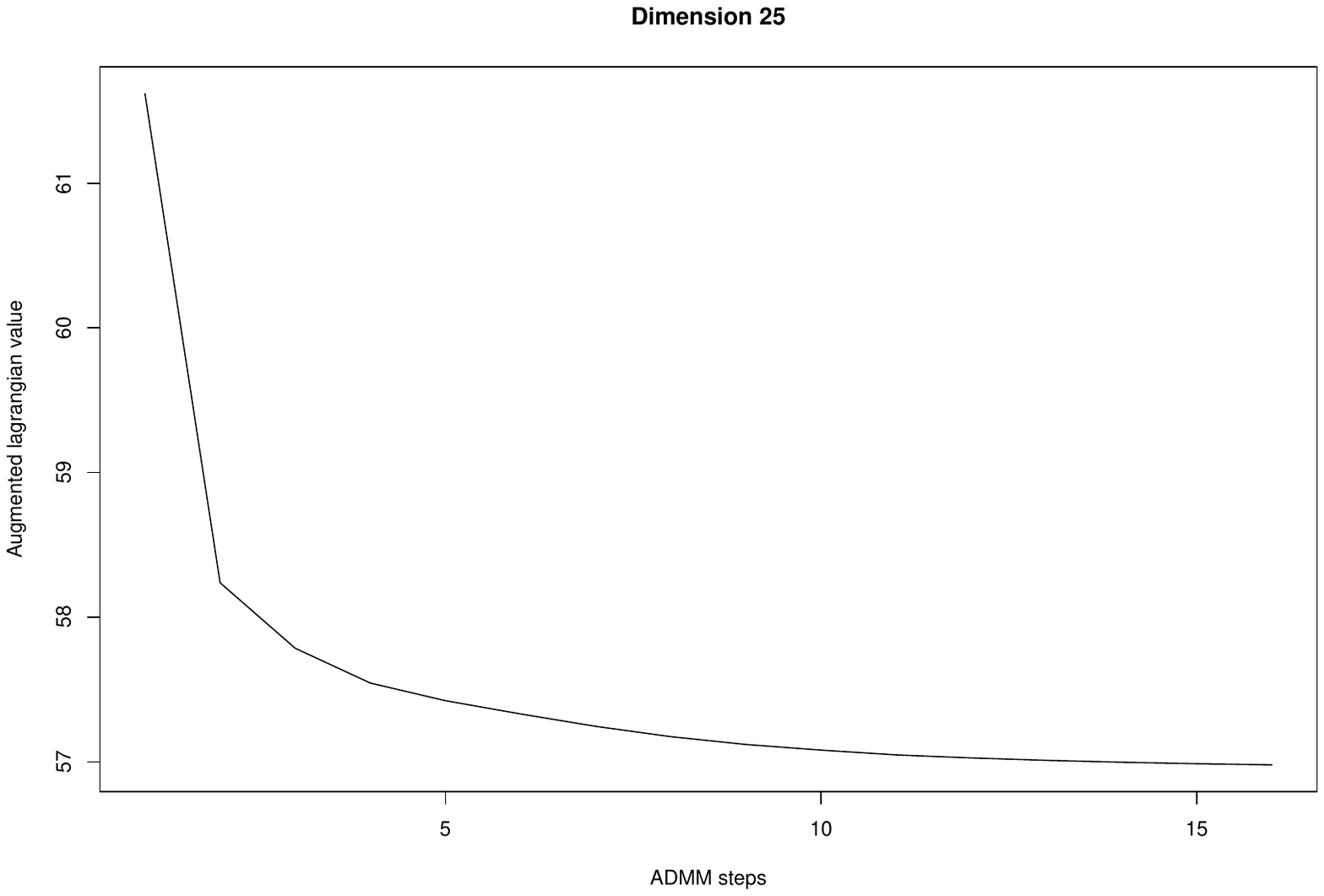}}
			\subfloat[$n=500$]{
				
				\includegraphics[scale=0.3]{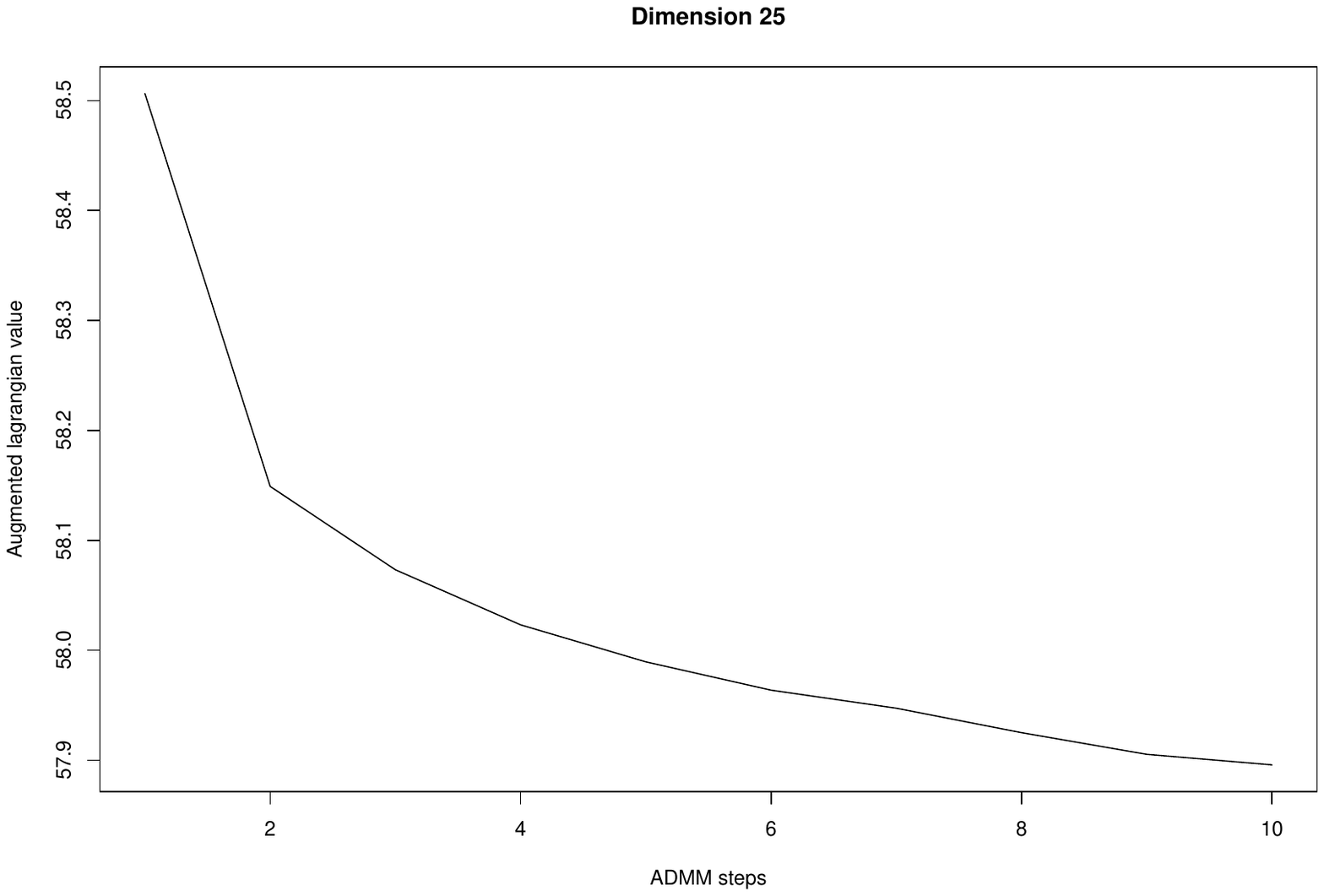}}				
		}
		
		\caption{\label{Figure 15}Comparison for Chain graphs (Augmented Lagrangian vs ADMM steps) for sample sizes $n=50$ (Panel a) and $n=500$ (Panel b)}
		
	\end{figure}
	\begin{figure}
		
		\centering{
			\subfloat[$n=50$]{
				
				\includegraphics[scale=0.3]{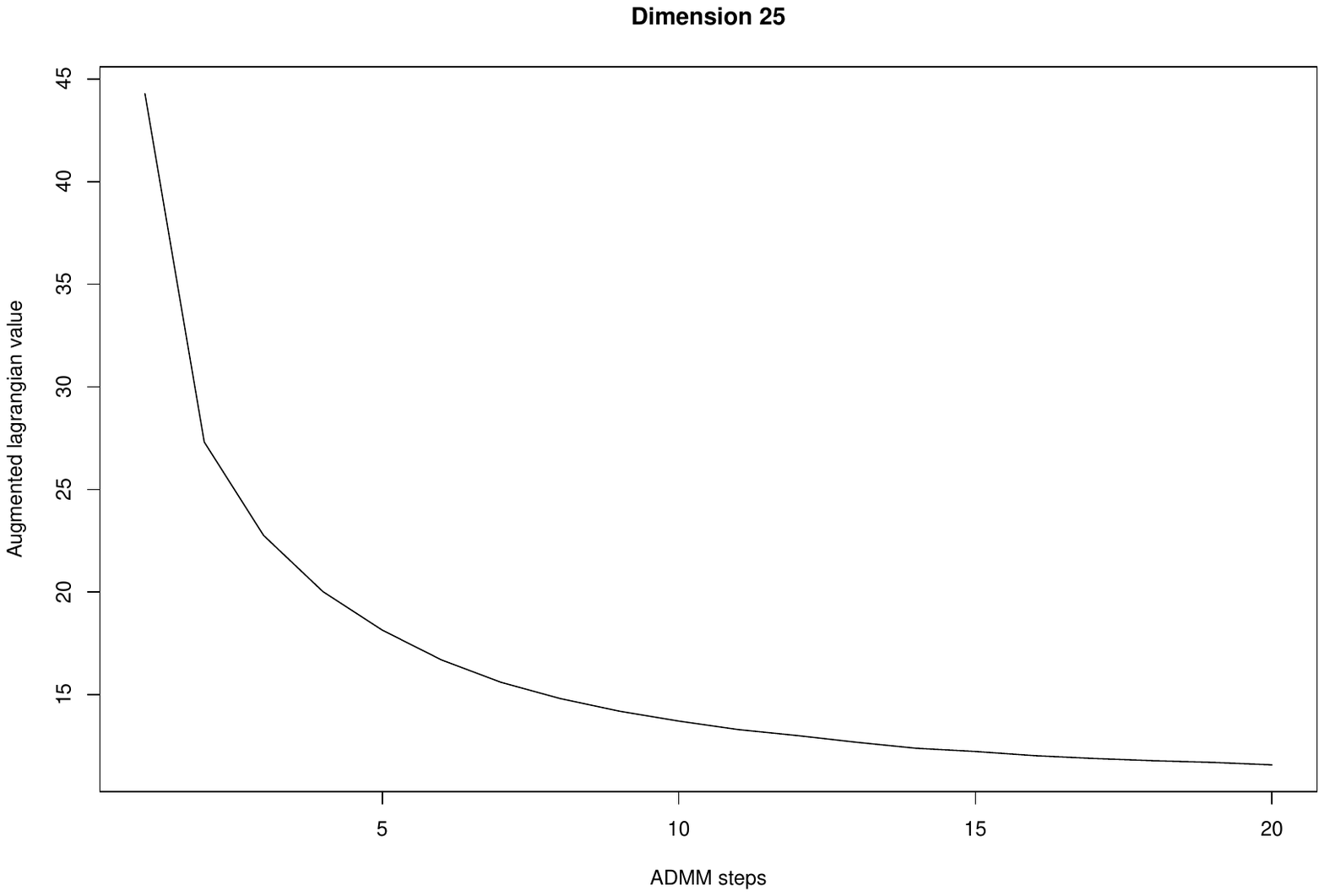}}
			\subfloat[$n=500$]{
				
				\includegraphics[scale=0.3]{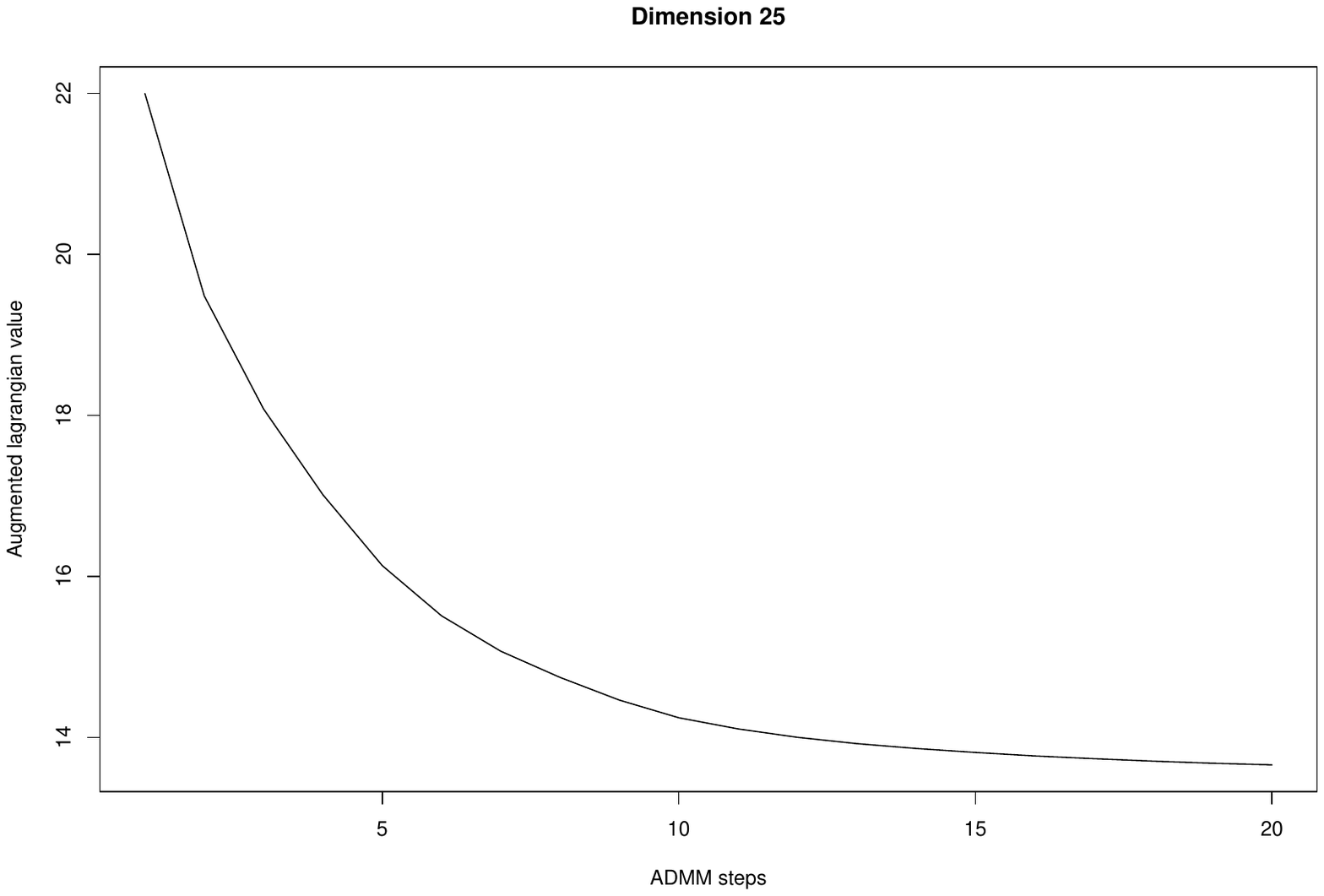}}				
		}
		
		\caption{\label{Figure 16}Comparison for Star graphs (Augmented Lagrangian vs ADMM steps) for sample sizes $n=50$ (Panel a) and $n=500$ (Panel b)}
		
	\end{figure}
	\subsection{Brief summary of recovered network and gene details for dataset GSE25066}
	The estimated graph exhibits a modular organisation consistent with
	the known functional architecture of the KEGG Breast Cancer signalling
	pathway, though the recovered connectivity reflects transcriptional
	co-regulatory relationships among genes rather than the biochemical
	interaction topology of the pathway. This distinction is important
	for interpreting both the recovered edges and the isolated or sparsely
	connected nodes.
	
	Within the cell cycle and DNA damage response module, the network
	recovers a coherent backbone comprising BRCA1, E2F1, CDK4, RB1, GADD45G,
	and peripheral nodes BAK1, DDB2, and E2F3. The CDK4--E2F1--RB1 axis
	reflects the canonical G1/S regulatory circuit in which CDK4-mediated
	phosphorylation of RB1 releases E2F transcription factors to drive
	S-phase entry (\cite{sherr1999cdk}). The BRCA1--E2F1 co-expression
	is consistent with the established regulation of BRCA1 transcription
	by RB:E2F repressor complexes across the cell cycle (\cite{bindra2006basal}; \cite{quaas2026brca1}; \cite{de2010transcriptional}). GADD45G connects via
	a cross-module edge to FZD7, consistent with the broader pattern of
	DNA damage response genes connecting to adjacent pathway modules through
	shared transcriptional regulatory elements rather than remaining confined
	within a single module. 
	
	Within the Notch signalling module, the network recovers a single
	internal edge between NOTCH3 and HES1, while NOTCH2 and NOTCH4 are
	disconnected from this pair and from each other internally, though
	both connect to the broader network via cross-pathway edges. The NOTCH3--HES1
	conditional dependence is the most biologically defensible recovery
	within this module: NOTCH3 is the predominantly expressed Notch receptor
	in luminal ER-positive breast cancer, directly associates with ER$\alpha$
	at the transcriptional level, and its expression is significantly
	correlated with relapse-free survival specifically in ER$\alpha$-positive
	patients (\cite{dou2017notch3}). HES1 is the primary canonical transcriptional
	target of Notch signalling and is highly expressed in the luminal
	breast cancer subtype (\cite{yousefi2022notch}), making the NOTCH3--HES1
	co-variation the most stable and consistent signal across the ER-stratified
	GSE25066 cohort. The internal disconnection of NOTCH2 reflects its
	well-documented context-dependent and partially divergent role in
	breast cancer: increased NOTCH2 expression has been associated with
	better survival in ER-positive luminal patients and with anti-tumorigenic
	activity in experimental models, making its transcriptional co-variation
	with the oncogenic NOTCH3--HES1 programme inconsistent across a heterogeneous
	patient population (\cite{nandi2020many}).
	NOTCH4, by contrast, is restricted to basal and myoepithelial compartments
	containing the mammary stem cell population, with NOTCH4 mRNA expressed
	at higher levels in basal cells than in luminal progenitors (\cite{harrison2010regulation}); in bulk microarray data from a predominantly luminal
	cohort such as GSE25066, its signal is diluted by the dominant luminal
	cell population, suppressing its conditional co-variation with NOTCH3
	and HES1 below the detection threshold of the graphical lasso. 
	
	By contrast, the PI3K/AKT/mTOR and RAS/MAPK modules display sparse
	internal connectivity, with several canonical components --- including
	PTEN, AKT2, KRAS, and MAP2K1 --- appearing as isolated nodes. This
	pattern is consistent with the predominantly post-translational nature
	of signalling through these cascades: AKT activation, RAS--RAF--MEK
	relay, and mTOR complex assembly are all regulated by phosphorylation
	events whose upstream mRNA levels need not co-vary. PTEN and AKT2
	are dysregulated in breast cancer primarily through genomic deletion
	and post-translational mechanisms rather than transcriptional co-regulation,
	explaining their isolation from the PI3K cluster (\cite{miricescu2020pi3k}). Similarly, KRAS dysregulation in breast
	cancer is predominantly mutational rather than transcriptional, and
	the isolation of MAP2K1 alongside retention of MAP2K2 reflects the
	near-perfect functional redundancy of MEK1 and MEK2 and the tendency
	of lasso-type penalties to select one representative from a collinear
	pair. The single recovered internal edge, FGF1--ARAF, is interpretable
	because both genes are subject to transcriptional regulation across
	ER subtypes, unlike the post-translational relay nodes of the cascade
	(\cite{bahar2023targeting}). 
	
	The Wnt/$\beta$-catenin module displays sparse internal
	connectivity, recovering only two internal edges: LRP6--FZD7 and
	FZD1--TCF7, while the canonical intracellular relay components ---
	DVL1, DVL3, APC, APC2, CTNNB1, and CSNK1A1 --- are internally disconnected.
	This is expected: the destruction complex relay is entirely post-translational,
	with APC, AXIN, CK1$\alpha$, and GSK3$\beta$
	sequentially phosphorylating $\beta$-catenin for proteasomal
	degradation, meaning the mRNA levels of these components need not
	co-vary for the pathway to function. Furthermore, CTNNB1 mRNA is constitutively
	expressed in breast cancer without the mutational dysregulation seen
	in colorectal cancer --- APC alterations account for only 2.4\% of
	breast cancer cases and CTNNB1 alterations are roughly ten-fold less
	frequent than in colorectal cancer --- suppressing their transcriptional
	variance across the cohort. The recovered LRP6--FZD7 edge is biologically
	interpretable, as LRP6 overexpression and elevated FZD7 protein levels
	have both been reported in breast cancer, supporting their co-upregulation
	as components of the same receptor complex (\cite{abreu2022wnt}). 
	
	A notable structural feature of the network is that genes from sparsely
	connected modules are not isolated from the broader graph --- rather,
	they connect to other pathway modules through cross-pathway edges
	that reflect shared transcriptional regulatory relationships with
	genes from adjacent modules. This produces a topology in which the
	Wnt, PI3K, and MAPK components serve as bridging nodes linking the
	denser cell cycle and Notch clusters, capturing the cross-pathway
	transcriptional crosstalk that characterises breast cancer signalling.
	While the estimated network does not imply causal directionality,
	it provides a statistically grounded and biologically interpretable
	representation of the shared conditional dependence structure underlying
	the KEGG Breast Cancer pathway across ER-positive and ER-negative
	tumour strata, with differences in regulatory intensity between subtypes
	captured through the population-specific matrices rather than through changes
	in network topology.
	
	\begin{figure}
		
		\centering{
			\subfloat[ER-positive]{
				
				\includegraphics[scale=0.3]{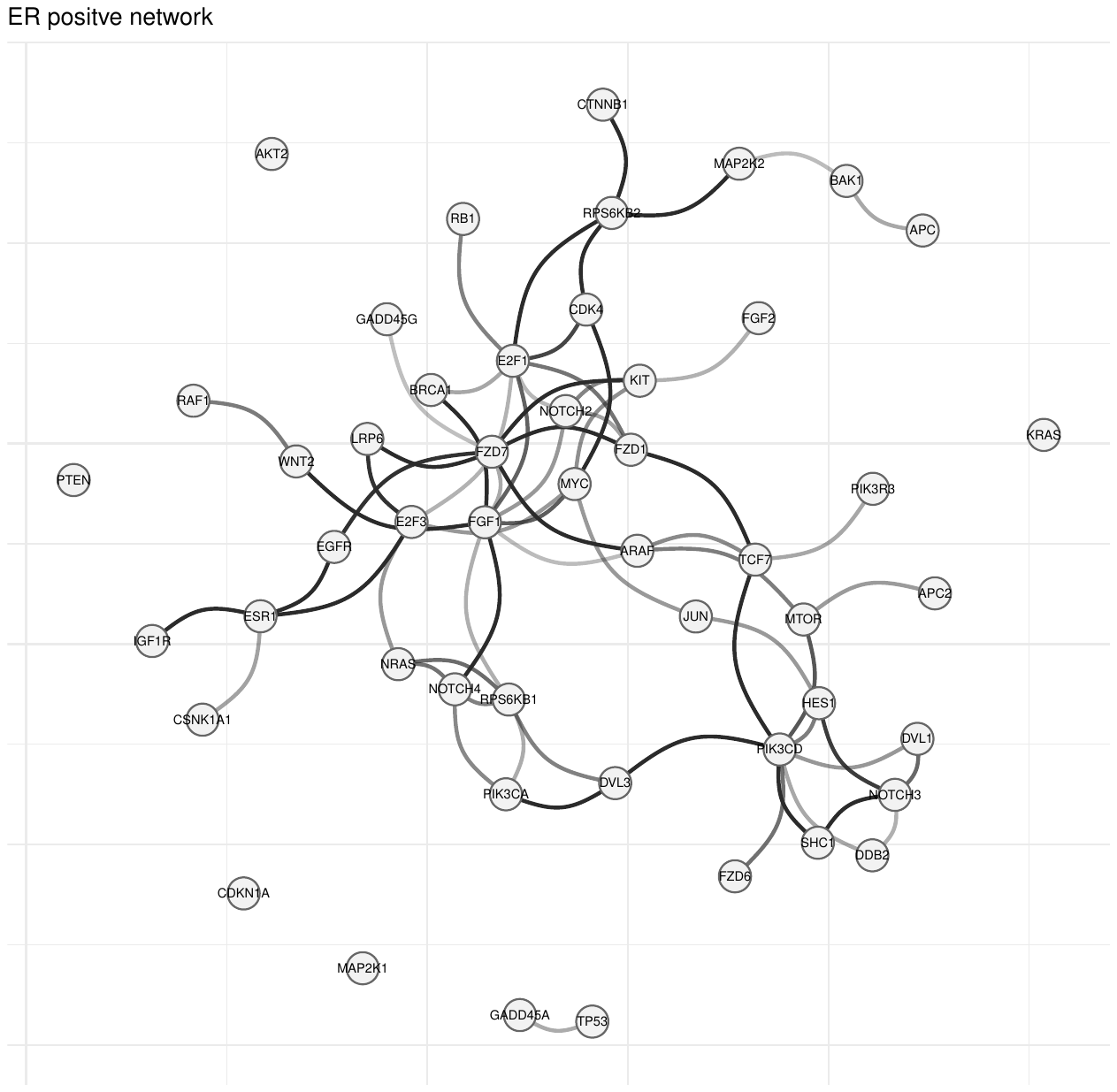}}
			\subfloat[ER-negative]{
				
				\includegraphics[scale=0.3]{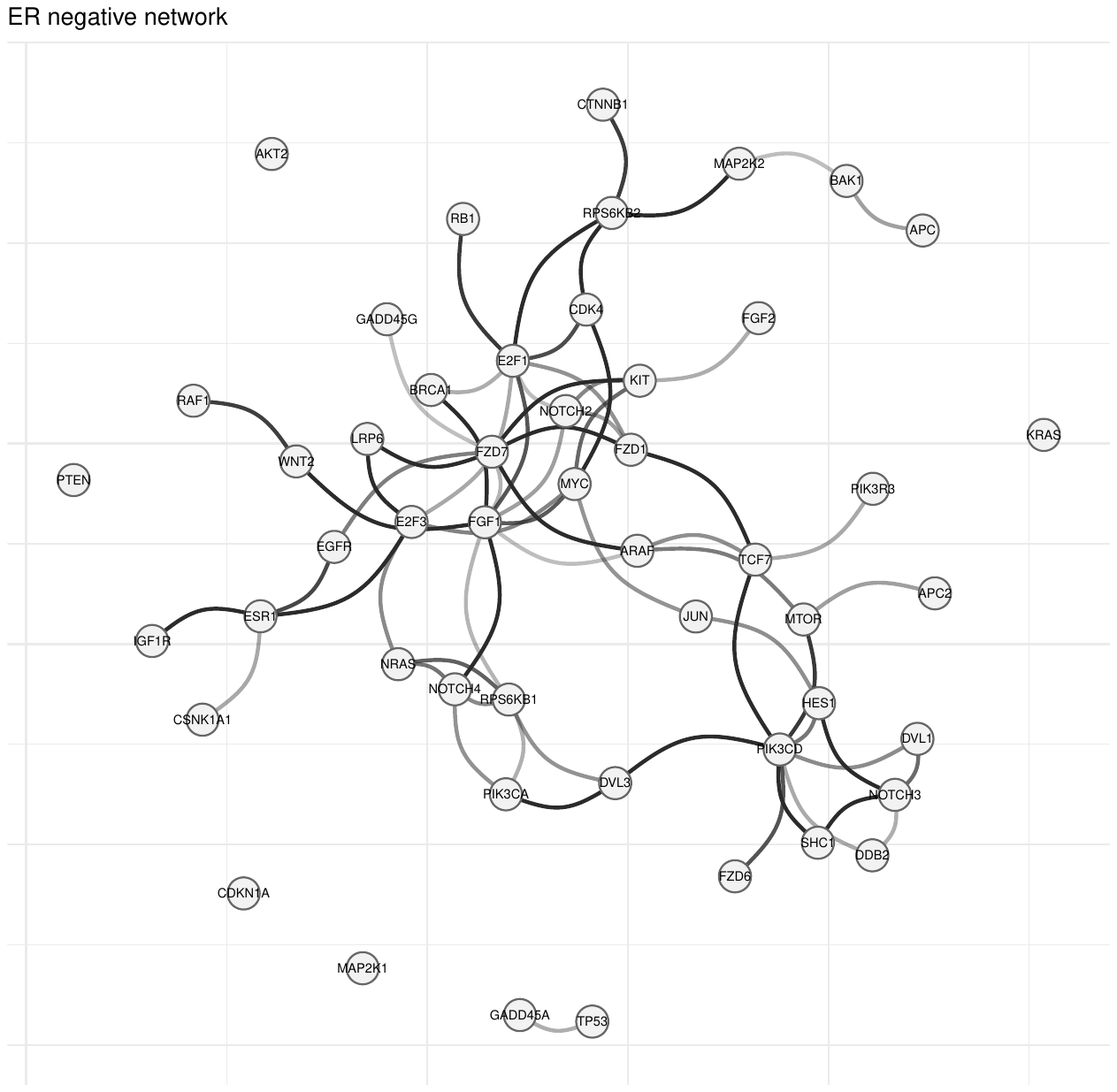}}				
		}
		
		\caption{\label{Figure 17}Population-specific gene regulatory network estimated for the ER-positive subgroup with $N_{1}=297$(Panel a) and ER-negative with $N_{2}=205$ (Panel b). Edge weights reflect the relative strength of conditional dependencies, capturing amplified regulatory interactions characteristic of estrogen receptor-driven tumours. Greater edge opacity indicates a stronger conditional dependence. Differences in edge intensity between the networks reflect differential suppression or upregulation of individual pathway interactions.}
		
	\end{figure}
	
	\newcolumntype{L}[1]{>{\raggedright\arraybackslash}p{#1}}
	
	\begin{landscape}          
		\small
		\setlength{\tabcolsep}{3pt}
		\renewcommand{\arraystretch}{1.05}
		
		\begin{longtable}{|L{3.0cm}|L{3.0cm}|L{3.5cm}|L{12.0cm}|}
			\caption{Functional annotation of the $p=50$ genes retained after
				mapping the GSE25066 expression data to the KEGG Breast Cancer
				signalling pathway (hsa05224). Genes are grouped by pathway module,
				with canonical roles and specific relevance to ER-positive breast
				cancer biology noted for each.}
			\label{tbl-genes}\\
			\hline
			\textbf{Gene Symbol} &
			\textbf{Functional Module} &
			\textbf{Canonical Cellular Function} &
			\textbf{Specific Relevance in ER+ Breast Cancer} \\
			\hline\hline
			\endfirsthead
			
			\multicolumn{4}{l}{\small\textit{(continued from previous page)}}\\[2pt]
			\hline
			\textbf{Gene Symbol} &
			\textbf{Functional Module} &
			\textbf{Canonical Cellular Function} &
			\textbf{Specific Relevance in ER+ Breast Cancer} \\
			\hline\hline
			\endhead
			
			\hline
			\multicolumn{4}{r}{\small\textit{(continued on next page)}}\\
			\endfoot
			
			\hline
			\endlastfoot
			
			AKT2 &
			PI3K/AKT signalling &
			Serine/threonine kinase (cell survival) &
			Downstream effector of PI3K; hyperactivation promotes survival
			and resistance to anti-estrogen therapies. \\
			\hline
			
			APC / APC2 &
			Wnt/$\beta$-catenin signalling &
			WNT pathway tumor suppressors &
			Loss or downregulation can unleash Wnt signaling,
			which heavily cross-talks with ER to drive tumor growth. \\
			\hline
			
			ARAF / RAF1 &
			RAS/MAPK signalling &
			Kinases in the MAPK cascade &
			MAPK pathway hyperactivation is a classic mechanism tumors
			use to bypass the need for estrogen (endocrine resistance). \\
			\hline
			
			BAK1 &
			Apoptosis / DNA damage &
			Pro-apoptotic protein &
			Often suppressed in ER+ tumors to evade programmed cell
			death induced by treatments. \\
			\hline
			
			BRCA1 &
			Apoptosis / DNA damage &
			DNA damage repair &
			While classically linked to Triple-Negative Breast Cancer,
			mutations or altered expression still impact genomic stability
			in ER+ subtypes. \\
			\hline
			
			CDK4 &
			Cell cycle regulation &
			Cell cycle regulator (G1 phase) &
			Highly actionable target. CDK4/6 inhibitors
			(e.g., Palbociclib) are standard-of-care treatments
			for advanced ER+ breast cancer. \\
			\hline
			
			CDKN1A (p21) &
			Cell cycle / Apoptosis &
			Cell cycle inhibitor &
			Often deregulated; normally halts the cell cycle,
			but ER+ tumors bypass it to maintain unchecked division. \\
			\hline
			
			CSNK1A1 &
			Wnt/$\beta$-catenin signalling &
			Kinase regulating WNT/cell division &
			Modulates both Wnt and p53 pathways,
			influencing tumor cell proliferation. \\
			\hline
			
			CTNNB1 &
			Wnt/$\beta$-catenin signalling &
			Beta-catenin (WNT signaling) &
			Interacts directly with the Estrogen Receptor; nuclear accumulation
			often correlates with poor prognosis in ER+ disease. \\
			\hline
			
			DDB2 &
			Apoptosis / DNA damage &
			DNA repair protein &
			Expression levels can influence the tumor's response
			to DNA-damaging chemotherapies. \\
			\hline
			
			DVL1 / DVL3 &
			Wnt/$\beta$-catenin signalling &
			WNT signaling mediators &
			Upregulation facilitates Wnt pathway hyperactivity,
			contributing to tumor progression and therapy resistance. \\
			\hline
			
			E2F1 / E2F3 &
			Cell cycle regulation &
			Transcription factors (cell cycle) &
			Drive the transition into the S-phase of the cell cycle;
			often overactive in ER+ tumors,
			targeted indirectly by CDK4/6 inhibitors. \\
			\hline
			
			EGFR &
			RAS/MAPK signalling &
			Growth factor receptor &
			Overexpression is a well-known escape mechanism for
			ER+ tumors treated with tamoxifen or aromatase inhibitors. \\
			\hline
			
			ESR1 &
			Oestrogen receptor / Transcription &
			Estrogen Receptor Alpha &
			The primary driver of ER+ breast cancer.
			Mutations in this gene (e.g., Y537S, D538G)
			frequently emerge to cause acquired resistance to therapy. \\
			\hline
			
			FGF1 / FGF2 &
			RAS/MAPK signalling &
			Fibroblast growth factors &
			Secreted ligands that activate FGFRs,
			providing alternative, estrogen-independent growth signals
			to the tumor. \\
			\hline
			
			FZD1/6/7 &
			Wnt/$\beta$-catenin signalling &
			Receptors for WNT ligands &
			Cell surface receptors that capture Wnt signals,
			driving the Wnt/beta-catenin cross-talk with ESR1. \\
			\hline
			
			GADD45A/G &
			Apoptosis / DNA damage &
			Stress/DNA damage responders &
			Modulate cell cycle arrest and apoptosis
			in response to chemotherapy; often silenced in tumors. \\
			\hline
			
			HES1 &
			Notch signalling &
			Notch pathway transcription factor &
			Associated with maintaining cancer stem cell populations
			and promoting metastasis in ER+ breast cancer. \\
			\hline
			
			IGF1R &
			PI3K/AKT signalling &
			Insulin-like growth factor receptor &
			Heavily cross-talks with ESR1.
			IGF1R signaling can activate the Estrogen Receptor
			even in the absence of estrogen. \\
			\hline
			
			JUN &
			RAS/MAPK signalling &
			AP-1 transcription factor subunit &
			Collaborates with ESR1 at the genomic level to drive transcription
			of genes promoting tumor proliferation. \\
			\hline
			
			KIT &
			RAS/MAPK signalling &
			Receptor tyrosine kinase &
			Overexpression can stimulate PI3K/AKT and MAPK pathways,
			supporting tumor growth. \\
			\hline
			
			KRAS / NRAS &
			RAS/MAPK signalling &
			GTPase signaling switches &
			While activating mutations are rare in ER+ breast cancer,
			altered expression still impacts MAPK signaling. \\
			\hline
			
			LRP6 &
			Wnt/$\beta$-catenin signalling &
			WNT co-receptor &
			Essential for canonical Wnt signaling;
			its upregulation supports Wnt/ER cross-talk. \\
			\hline
			
			MAP2K1 / MAP2K2 &
			RAS/MAPK signalling &
			MEK1 and MEK2 kinases &
			Central nodes in the MAPK pathway;
			targeted by MEK inhibitors to combat endocrine resistance. \\
			\hline
			
			MTOR &
			PI3K/AKT signalling &
			Master metabolic regulator &
			Highly actionable target. mTOR inhibitors (e.g., Everolimus)
			are used in ER+ breast cancer to overcome endocrine resistance. \\
			\hline
			
			MYC &
			Oestrogen receptor / Transcription &
			Oncogenic transcription factor &
			Amplification or overexpression is strongly associated
			with aggressive ER+ tumors, poor outcome, and therapy resistance. \\
			\hline
			
			NOTCH2/3/4 &
			Notch signalling &
			Receptors for Notch signaling &
			Linked to the expansion of breast cancer stem cells
			and resistance to endocrine therapy. \\
			\hline
			
			PIK3CA &
			PI3K/AKT signalling &
			PI3K catalytic subunit alpha &
			One of the most frequently mutated genes in ER+ breast cancer
			($\sim$40\%). Mutations hyperactivate the pathway
			and drive endocrine resistance. \\
			\hline
			
			PIK3CD / PIK3R3 &
			PI3K/AKT signalling &
			PI3K pathway components &
			Modulate the overall activity of the PI3K/AKT/mTOR cascade,
			influencing tumor survival. \\
			\hline
			
			PTEN &
			PI3K/AKT signalling &
			Tumor suppressor (PI3K inhibitor) &
			Frequently lost or downregulated in ER+ breast cancer,
			leading to unchecked PI3K pathway activation. \\
			\hline
			
			RB1 &
			Cell cycle regulation &
			Retinoblastoma tumor suppressor &
			The direct target of CDK4/6. Loss of RB1 is a primary mechanism
			by which ER+ tumors become resistant to CDK4/6 inhibitors. \\
			\hline
			
			RPS6KB1/2 &
			PI3K/AKT signalling &
			Ribosomal S6 kinases &
			Downstream effectors of mTOR; highly active in resistant ER+ tumors,
			driving the protein synthesis required for rapid growth. \\
			\hline
			
			SHC1 &
			RAS/MAPK signalling &
			Signaling adapter protein &
			Links activated growth factor receptors (like EGFR/IGF1R)
			to the MAPK and PI3K pathways. \\
			\hline
			
			TCF7 &
			Wnt/$\beta$-catenin signalling &
			WNT transcription factor &
			Partners with beta-catenin (CTNNB1)
			to express genes promoting cell division. \\
			\hline
			
			TP53 &
			Apoptosis / DNA damage &
			``Guardian of the genome'' &
			Mutated in $\sim$20\% of luminal A (ER+) tumors,
			but higher in luminal B (a more aggressive ER+ subtype);
			correlates with worse prognosis. \\
			\hline
			
			WNT2 &
			Wnt/$\beta$-catenin signalling &
			Secreted WNT ligand &
			Upregulated in the tumor microenvironment; promotes tumor growth,
			invasion, and interactions with stromal cells. \\
			\hline
			
		\end{longtable}
	\end{landscape}
	\subsection{Full tables corresponding to analysis of the illustrative example with heterogeneous sparsity}
	\begin{table}[htbp]
		\centering
		\caption{ Number of replications (out of $B=100$) in which the estimated edge
			sets $\hat{E}_{i}$ are identical across all three populations, reported
			as a function of sample size $n$.}
		\label{Table 8.1}
		
		\begin{tabular}{lcccccc}
			\toprule
			$n=100$ & $n=200$ & $n=300$ & $n=400$ & $n=500$ & $n=600$ \\
			\midrule
			
			100 & 100 & 99 & 97 & 98 & 100 \\
			
			\bottomrule
		\end{tabular}
	\end{table}
	\begin{table}[htbp]
		\centering
		\caption{ Replication frequency (out of $B=100$) with which each common edge
			in $E_{common}$ is correctly identified by the Mglasso estimator
			across all three populations jointly, as a function of sample size
			$n$.}
		\label{Table 8.2}
		
		\begin{tabular}{lccccccc}
			\toprule
			Common edges & $n=100$ & $n=200$ & $n=300$ & $n=400$ & $n=500$ & $n=600$ \\
			\midrule
			
			(1,2) & 75 & 99 & 100 & 100 & 100 & 100\\ 
			(3,4) & 84 & 96 & 99 & 100 & 100 & 100\\
			(6,7) & 79 & 95 & 100 & 100 & 100 & 100\\
			(10,11) & 80 & 97 & 100 & 100 & 100 & 100\\
			(14,15) & 88 & 98 & 100 & 100 & 100 & 100\\
			(17,18) & 81 & 98 & 100 & 100 & 100 & 100\\
			(19,20) & 88 & 99 & 100 & 100 & 100 & 100\\ 
			(20,21) & 82 & 97 & 99 & 100 & 100 & 100\\
			(23,24) & 86 & 98 & 100 & 100 & 100 & 100\\
			
			\bottomrule
		\end{tabular}
	\end{table}
	\begin{table}[htbp]
		
		\centering
		\caption{ \footnotesize Replication frequency (out of $B=100$) with which each uncommon edge
			in $E_{uncommon}$ is identified in the estimated precision matrix
			as a function of sample size $n$. Rows are left blank where the edge
			is truly absent in that population.}
		\label{Table 8.3}
		\footnotesize
		\begin{tabular}{lcccccccc}
			\toprule
			Unommon edges & & $n=100$ & $n=200$ & $n=300$ & $n=400$ & $n=500$ & $n=600$ \\
			\midrule
			
			& Population 1 & 31 & 48 & 39 & 45 & 36 & 38\\
			(8,9)& Population 2 & 31 & 48 & 39 & 45 & 36 & 38\\
			& Population 3 &  &  &  &  &  & \\
			& Population 1 & 29 & 45 & 36 & 43 & 40 & 24\\
			(12,13)& Population 2 & 29 & 45 & 36 & 43 & 40 & 24\\
			& Population 3 &  &  &  &  &  & \\
			& Population 1 & 33 & 37 & 41 & 38 & 36 & 39\\
			(22,23)& Population 2 & 33 & 37 & 41 & 38 & 36 & 39\\
			& Population 3 &  &  &  &  &  & \\
			& Population 1 & 12 & 38 & 45 & 39 & 37 & 42\\
			(5,6)& Population 2 &  &  &  &  &  & \\
			& Population 3 & 12 & 38 & 45 & 39 & 37 & 42\\
			& Population 1 & 16 & 43 & 47 & 40 & 40 & 40\\
			(11,12)& Population 2 &  &  &  &  &  & \\
			& Population 3 & 16 & 43 & 47 & 40 & 40 & 40\\ 
			& Population 1 & 17 & 45 & 42 & 41 & 38 & 34\\
			(15,16)& Population 2 &  &  &  &  &  & \\
			& Population 3 & 17 & 45 & 42 & 41 & 38 & 34\\
			& Population 1 &  &  &  &  &  & \\
			(13,14)& Population 2 & 12 & 23 & 20 & 14 & 18 & 18\\
			& Population 3 &  &  &  &  &  & \\
			& Population 1 &  &  &  &  &  & \\
			(9,10)& Population 2 &  &  &  &  &  & \\
			& Population 3 & 4 & 11 & 20 & 17 & 14 & 8\\
			
			\bottomrule
		\end{tabular}
	\end{table}
	\begin{table}[htbp]
		\centering
		\caption{ Average number of true 8 uncommon edges $E_{uncommon}$ identified
			by the estimated precision matrices across different sample sizes
			$(n)$. Values are averaged over $B=100$ replications}
		\label{Table 8.4}
		
		\begin{tabular}{lcccccc}
			\toprule
			$n=100$ & $n=200$ & $n=300$ & $n=400$ & $n=500$ & $n=600$ \\
			\midrule
			
			1.54 & 2.9 & 2.9 & 2.76 & 2.58 & 2.43 \\
			
			\bottomrule
		\end{tabular}
	\end{table}
	\begin{table}[htbp]
		\centering
		\caption{ Average number of 283 pairs of coordinates from $E^{c}$ falsely identified
			by the estimated precision matrices across different sample sizes
			$(n)$. Values are averaged over $B=100$ replications for each population}
		\label{Table 8.5}
		
		\begin{tabular}{lcccccc}
			\toprule
			$n=100$ & $n=200$ & $n=300$ & $n=400$ & $n=500$ & $n=600$ \\
			\midrule
			
			0.16 & 1.1 & 7.72 & 7.21 & 6.06 & 5.16 \\
			
			\bottomrule
		\end{tabular}
	\end{table}
	
	\bibliographystyle{apalike}
	\bibliography{references}
\end{document}